\documentclass[12pt]{article}
\usepackage{latexsym, amssymb, amsthm}
\usepackage{dsfont}
\usepackage{color}

\DeclareMathAlphabet\gothic{U}{euf}{m}{n}

\makeatletter
\def\eqnarray{\stepcounter{equation}\let\@currentlabel=\theequation
\global\@eqnswtrue
\tabskip\@centering\let\\=\@eqncr
$$\halign to \displaywidth\bgroup\hfil\global\@eqcnt\z@
  $\displaystyle\tabskip\z@{##}$&\global\@eqcnt\@ne
  \hfil$\displaystyle{{}##{}}$\hfil
  &\global\@eqcnt\tw@ $\displaystyle{##}$\hfil
  \tabskip\@centering&\llap{##}\tabskip\z@\cr}

\def\endeqnarray{\@@eqncr\egroup
      \global\advance\c@equation\m@ne$$\global\@ignoretrue}

\def\@yeqncr{\@ifnextchar [{\@xeqncr}{\@xeqncr[5pt]}}
\makeatother

\begin{document}
\bibliographystyle{tom}

\newtheorem{lemma}{Lemma}[section]
\newtheorem{thm}[lemma]{Theorem}
\newtheorem{cor}[lemma]{Corollary}
\newtheorem{prop}[lemma]{Proposition}
\newtheorem{ddefinition}[lemma]{Definition}

\theoremstyle{definition}

\newtheorem{remark}[lemma]{Remark}
\newtheorem{exam}[lemma]{Example}

\newcommand{\gota}{\gothic{a}}
\newcommand{\gotb}{\gothic{b}}
\newcommand{\gotc}{\gothic{c}}
\newcommand{\gote}{\gothic{e}}
\newcommand{\gotf}{\gothic{f}}
\newcommand{\gotg}{\gothic{g}}
\newcommand{\gothh}{\gothic{h}}
\newcommand{\gotk}{\gothic{k}}
\newcommand{\gotl}{\gothic{l}}
\newcommand{\gotm}{\gothic{m}}
\newcommand{\gotn}{\gothic{n}}
\newcommand{\gotp}{\gothic{p}}
\newcommand{\gotq}{\gothic{q}}
\newcommand{\gotr}{\gothic{r}}
\newcommand{\gots}{\gothic{s}}
\newcommand{\gott}{\gothic{t}}
\newcommand{\gotu}{\gothic{u}}
\newcommand{\gotv}{\gothic{v}}
\newcommand{\gotw}{\gothic{w}}
\newcommand{\gotz}{\gothic{z}}
\newcommand{\gotA}{\gothic{A}}
\newcommand{\gotB}{\gothic{B}}
\newcommand{\gotG}{\gothic{G}}
\newcommand{\gotL}{\gothic{L}}
\newcommand{\gotS}{\gothic{S}}
\newcommand{\gotT}{\gothic{T}}

\newcounter{teller}
\renewcommand{\theteller}{(\alph{teller})}
\newenvironment{tabel}{\begin{list}%
{\rm  (\alph{teller})\hfill}{\usecounter{teller} \leftmargin=1.1cm
\labelwidth=1.1cm \labelsep=0cm \parsep=0cm}
                      }{\end{list}}

\newcounter{tellerr}
\renewcommand{\thetellerr}{(\roman{tellerr})}
\newenvironment{tabeleq}{\begin{list}%
{\rm  (\roman{tellerr})\hfill}{\usecounter{tellerr} \leftmargin=1.1cm
\labelwidth=1.1cm \labelsep=0cm \parsep=0cm}
                         }{\end{list}}

\newcounter{tellerrr}
\renewcommand{\thetellerrr}{(\Roman{tellerrr})}
\newenvironment{tabelR}{\begin{list}%
{\rm  (\Roman{tellerrr})\hfill}{\usecounter{tellerrr} \leftmargin=1.1cm
\labelwidth=1.1cm \labelsep=0cm \parsep=0cm}
                         }{\end{list}}

\newcounter{proofstep}
\newcommand{\nextstep}{\refstepcounter{proofstep}\vertspace \par 
          \noindent{\bf Step \theproofstep.} \hspace{5pt}}
\newcommand{\firststep}{\setcounter{proofstep}{0}\nextstep}

\newcommand{\Ni}{\mathds{N}}
\newcommand{\Qi}{\mathds{Q}}
\newcommand{\Ri}{\mathds{R}}
\newcommand{\Ci}{\mathds{C}}
\newcommand{\Ti}{\mathds{T}}
\newcommand{\Zi}{\mathds{Z}}
\newcommand{\Fi}{\mathds{F}}

\renewcommand{\proofname}{{\bf Proof}}

\newcommand{\vertspace}{\vskip10.0pt plus 4.0pt minus 6.0pt}

\newcommand{\simh}{{\stackrel{{\rm cap}}{\sim}}}
\newcommand{\ad}{{\mathop{\rm ad}}}
\newcommand{\Ad}{{\mathop{\rm Ad}}}
\newcommand{\alg}{{\mathop{\rm alg}}}
\newcommand{\clalg}{{\mathop{\overline{\rm alg}}}}
\newcommand{\Aut}{\mathop{\rm Aut}}
\newcommand{\arccot}{\mathop{\rm arccot}}
\newcommand{\capp}{{\mathop{\rm cap}}}
\newcommand{\rcapp}{{\mathop{\rm rcap}}}
\newcommand{\diam}{\mathop{\rm diam}}
\newcommand{\divv}{\mathop{\rm div}}
\newcommand{\dom}{\mathop{\rm dom}}
\newcommand{\codim}{\mathop{\rm codim}}
\newcommand{\RRe}{\mathop{\rm Re}}
\newcommand{\IIm}{\mathop{\rm Im}}
\newcommand{\tr}{{\mathop{\rm tr \,}}}
\newcommand{\Tr}{{\mathop{\rm Tr \,}}}
\newcommand{\Vol}{{\mathop{\rm Vol}}}
\newcommand{\card}{{\mathop{\rm card}}}
\newcommand{\rank}{\mathop{\rm rank}}
\newcommand{\supp}{\mathop{\rm supp}}
\newcommand{\sgn}{\mathop{\rm sgn}}
\newcommand{\essinf}{\mathop{\rm ess\,inf}}
\newcommand{\esssup}{\mathop{\rm ess\,sup}}
\newcommand{\Int}{\mathop{\rm Int}}
\newcommand{\lcm}{\mathop{\rm lcm}}
\newcommand{\loc}{{\rm loc}}
\newcommand{\HS}{{\rm HS}}
\newcommand{\Trn}{{\rm Tr}}
\newcommand{\n}{{\rm N}}
\newcommand{\WOT}{{\rm WOT}}

\newcommand{\at}{@}

\newcommand{\mod}{\mathop{\rm mod}}
\newcommand{\spann}{\mathop{\rm span}}
\newcommand{\one}{\mathds{1}}

\hyphenation{groups}
\hyphenation{unitary}

\newcommand{\tfrac}[2]{{\textstyle \frac{#1}{#2}}}

\newcommand{\ca}{{\cal A}}
\newcommand{\cb}{{\cal B}}
\newcommand{\cc}{{\cal C}}
\newcommand{\cd}{{\cal D}}
\newcommand{\ce}{{\cal E}}
\newcommand{\cf}{{\cal F}}
\newcommand{\ch}{{\cal H}}
\newcommand{\chs}{{\cal HS}}
\newcommand{\ci}{{\cal I}}
\newcommand{\ck}{{\cal K}}
\newcommand{\cl}{{\cal L}}
\newcommand{\cm}{{\cal M}}
\newcommand{\cn}{{\cal N}}
\newcommand{\co}{{\cal O}}
\newcommand{\cp}{{\cal P}}
\newcommand{\cs}{{\cal S}}
\newcommand{\ct}{{\cal T}}
\newcommand{\cx}{{\cal X}}
\newcommand{\cy}{{\cal Y}}
\newcommand{\cz}{{\cal Z}}

\newlength{\hightcharacter}
\newlength{\widthcharacter}
\newcommand{\covsup}[1]{\settowidth{\widthcharacter}{$#1$}\addtolength{\widthcharacter}{-0.15em}\settoheight{\hightcharacter}{$#1$}\addtolength{\hightcharacter}{0.1ex}#1\raisebox{\hightcharacter}[0pt][0pt]{\makebox[0pt]{\hspace{-\widthcharacter}$\scriptstyle\circ$}}}
\newcommand{\cov}[1]{\settowidth{\widthcharacter}{$#1$}\addtolength{\widthcharacter}{-0.15em}\settoheight{\hightcharacter}{$#1$}\addtolength{\hightcharacter}{0.1ex}#1\raisebox{\hightcharacter}{\makebox[0pt]{\hspace{-\widthcharacter}$\scriptstyle\circ$}}}
\newcommand{\scov}[1]{\settowidth{\widthcharacter}{$#1$}\addtolength{\widthcharacter}{-0.15em}\settoheight{\hightcharacter}{$#1$}\addtolength{\hightcharacter}{0.1ex}#1\raisebox{0.7\hightcharacter}{\makebox[0pt]{\hspace{-\widthcharacter}$\scriptstyle\circ$}}}

\thispagestyle{empty}

\vspace*{1cm}
\begin{center}
{\Large\bf Dirichlet-to-Neumann and elliptic operators \\[2mm]
on $C^{1+\kappa}$-domains: Poisson and Gaussian bounds\\[10mm]
\large A.F.M. ter Elst$^1$ and E.M.  Ouhabaz$^2$}

\end{center}

\vspace{5mm}

\begin{center}
{\bf Abstract}
\end{center}

\begin{list}{}{\leftmargin=1.8cm \rightmargin=1.8cm \listparindent=10mm 
   \parsep=0pt}
\item
We prove Poisson upper bounds for the heat kernel  of the Dirichlet-to-Neumann operator  
with variable H\"older coefficients 
when  the underlying domain
is bounded and has a $C^{1+\kappa}$-boundary for some $\kappa > 0$.  
We also prove a number of other results such as gradient estimates for  heat kernels and 
Green functions $G$ of elliptic operators with possibly complex-valued coefficients. 
We establish H\"older continuity 
of $\nabla_x \nabla_y G$ up to the boundary. 
These results are used to prove $L_p$-estimates for commutators of 
Dirichlet-to-Neumann operators 
on the boundary of $C^{1+ \kappa}$-domains.  
Such estimates are the keystone in our approach for the  Poisson bounds. 
\end{list}

\vspace{5mm}
\noindent
May 2017

\vspace{5mm}

\vspace{5mm}

\vspace{15mm}

\noindent
{\bf Home institutions:}    \\[3mm]
\begin{tabular}{@{}cl@{\hspace{10mm}}cl}
1. & Department of Mathematics  & 
  2. & Institut de Math\'ematiques de Bordeaux \\
& University of Auckland   & 
  & Universit\'e de Bordeaux, UMR 5251,  \\
& Private bag 92019 & 
  &351, Cours de la Lib\'eration  \\
& Auckland 1142 & 
  &  33405 Talence \\
& New Zealand  & 
  & France\\
  & terelst@math.auckland.ac.nz&
 & Elmaati.Ouhabaz@math.u-bordeaux.fr \\[8mm]
\end{tabular}

\newpage

\tableofcontents

\section{Introduction} \label{S1}

Let $\Omega \subset \Ri^d$ be a bounded connected open set with 
Lipschitz boundary and $d \geq 2$.
 Denote by  $\Gamma = \partial \Omega$  the boundary of $\Omega$, 
endowed  with the 
$(d-1)$-dimensional Hausdorff measure.
Note that $\Gamma$ is not connected in general.
Let $C := (c_{kl})_{1\le k, l \le d}$ be real-valued matrix satisfying the usual 
ellipticity condition and $c_{kl} = c_{lk}  \in L_\infty(\Omega)$ for all 
$k, l \in \{1, \ldots, d\}$.
 Let $V \in L_\infty(\Omega,\Ri)$.  
The Dirichlet-to-Neumann operator $\cn_V$ is an unbounded operator on 
$L_2(\Gamma)$ defined as follows. 
Given  $\varphi \in L_2(\Gamma)$,
we solve (if possible) the Dirichlet problem
\begin{eqnarray*}
-\sum_{k,l= 1}^d \partial_l\left(c_{kl} \partial_k\,  u \right) + V u & = & 0 
      \quad \mbox{weakly on } \Omega,   \\[0pt]
u|_\Gamma & = & \varphi \nonumber
\end{eqnarray*}
with $u \in W^{1,2}(\Omega)$.
We define  the weak conormal derivative $\partial_\nu^C u$, which is 
formally equal to
\[
\sum_{k,l=1}^d n_l \, c_{kl} \, \partial_l u, 
\]
where $(n_1, \ldots, n_d)$ is the outer normal vector to $\Omega$. 
If $u$ has a weak conormal derivative in 
$L_2(\Gamma)$, then we say that $\varphi \in D(\cn_V)$ and 
$\cn_V \varphi = \partial_\nu^C u$.
If $V = 0$ we write $\cn$ instead of $\cn_0$. 
See the beginning of Section~\ref{S2} for more details on this  definition.
In particular, we shall always assume that
$0 \notin \sigma(A_D+V)$, where 
$A_D := -\sum_{k,l= 1}^d \partial_l\left(c_{kl} \partial_k \right)$ and subject 
to the Dirichlet boundary condition. 

The Dirichlet-to-Neumann operator, also known as voltage-to-current map,  
arises in the problem  of electrical impedance tomography and in various 
inverse problems (e.g., Calder\'on's problem).  
It is also used in the theory of homogenization and  analysis of elliptic systems with rapidly 
oscillating coefficients (see Kenig, Lin and Shen \cite{KLS} and the references there). 
Our aim in the present  paper is to address  another problem, 
namely upper bounds for the heat kernel associated with the Dirichlet-to-Neumann operator. 
Heat kernel bounds (mainly Gaussian bounds) for various differential operators on 
domains of $\Ri^d$ as well as on Riemannian manifolds have attracted 
a lot of attention in recent years. 
It turns out that they are a powerful tool to 
attack problems in harmonic analysis, such as Calder\'on-Zygmund operators, 
Riesz transforms, spectral multipliers as well as other problems in spectral theory 
and evolution equations. 
See for example the monograph \cite{Ouh5} and the references therein. 

It is well known that $\cn_V$ is a lower-bounded and self-adjoint operator on $L_2(\Gamma)$
with compact resolvent.
Therefore, $- \cn_V$ generates a $C_0$-semigroup $S^V$ on $L_2(\Gamma)$.
If $\Omega$ has  $C^\infty$-boundary, $c_{kl} = \delta_{kl}$  and $V \ge 0$, 
then it was shown by ter Elst and Ouhabaz \cite{EO4} that 
$S^V$ is given by a kernel which satisfies  Poisson upper bounds. 
In the present paper we extend considerably this result to deal with variable 
H\"older-continuous coefficients $c_{kl}$ and less regular domains.
Our main result in this direction reads as follows. 

\begin{thm} \label{thm1.1}
Suppose $\Omega \subset \Ri^d$ is bounded connected with a $C^{1+ \kappa}$-boundary 
$\Gamma$ for some $\kappa \in (0,1)$.
Suppose also each $c_{kl} = c_{lk}$ is real valued and 
H\"older continuous on~$\Omega$.
Let $V \in L_\infty(\Omega,\Ri)$ and suppose that 
$0 \notin\sigma(A_D +V)$.
Denote by $\cn_V$ the corresponding Dirichlet-to-Neumann operator. 
Then the semigroup generated by $-\cn_V$ has a kernel $K^V$ and  there exists a $c > 0$ such that 
\[
| K^V_t(z,w) | 
\leq \frac{c \, (t \wedge 1)^{-(d-1)} \, e^{-\lambda_1 t}}
         {\displaystyle \Big( 1 + \frac{|z-w|}{t} \Big)^d }
\]
for all $z, w \in \Gamma$ and $t > 0$, where $\lambda_1 $ is the 
first eigenvalue of the operator $\cn_V$.
\end{thm}

One immediate consequence of this result is that the semigroup $S^V$ acts as a  
holomorphic semigroup on $L_1(\Gamma)$.
Even the existence of  such semigroup on $L_1(\Gamma)$ as a $C_0$-semigroup is new in this generality.
The holomorphy of the semigroup follows as in Theorem 7.1 in \cite{EO4}.
 We can also draw further information, for example $\cn_V$ has a holomorphic 
functional calculus on $L_p(\Gamma)$ for all $p \in (1, \infty)$ with angle 
$\mu \in (\frac{\pi(d-1)}{2d}, \pi)$, see Theorem 7.2 in \cite{EO4}. 
The previous theorem has  another  consequence.
It allows to establish  existence results for evolution equations on $C(\Gamma)$ (the space
 of continuous functions on $\Gamma$).
This subject will be addressed in a forthcoming paper.

The strategy of proof is similar to  \cite{EO4} in the sense that we prove appropriate 
$L_p$--$L_q$ estimates for iterated commutators of 
the semigroup $S^V = (e^{-t\cn_V})_{t > 0}$ with $M_g$, 
a multiplication operator by a Lipschitz continuous function  
$g$ on~$\Gamma$.
In \cite{EO4} these estimates are essentially deduced from $L_p$--$L_q$ estimates 
of $S$ together with 
commutator estimates of Coifman--Meyer for pseudo-differential operators and this 
is the reason why we assumed  there that the boundary is  $C^\infty$.

One cannot use these commutator results of  Coifman--Meyer on $C^{1+\kappa}$-domains and 
this is the major difficulty here. 
We shall instead rely solely on a recent $L_2$--$L_2$ estimate for the commutator 
$[\cn, M_g]$ proved by Shen \cite{She2}. The result of \cite{She2} is valid even for 
$\Omega$ with Lipschitz boundary.
We extend this commutator estimate to $L_p(\Gamma)$ for all $p \in (1, \infty)$ under 
the assumption that $\Omega$ has a $C^{1+\kappa}$-boundary by appealing to 
Calder\'on--Zygmund theory. In order to do so we need appropriate bounds for the 
Schwartz kernel $K_{\cn_V}$ of $\cn_V$, namely
\[
| K_{\cn_V}(z,w) | \le \frac{c}{| z-w|^d}
\]
and
\[
| K_{\cn_V}(z,w) - K_{\cn_V}(z',w') | 
\le c \, \frac{ (| z-z'| + | w-w'|)^\kappa}{ | z-w|^{d+ \kappa}}
\]
for all $z, z', w, w' \in \Gamma$ with $z \not= w$ and $|z-z'| + |w-w'| \le \frac{1}{2}|z-w|$.
It turns out that one can express  the Schwartz kernel $K_{\cn_V}$ in terms of the 
trace on the boundary of second order derivatives $\partial_k^{(1)} \partial_l^{(2)} G_V$ 
of the Green function $G_V$ of $A_D + V$. 
Therefore, we need appropriate bounds and H\"older continuity for  
$\partial_k^{(1)} \partial_l^{(2)} G_V$. 
We take the opportunity to prove these bounds in the general setting of 
elliptic operators with complex-valued coefficients. 
We prove that the heat kernel $H_t$ of $A_D + V$ satisfies bounds
\[
|(\partial_x^\alpha \, \partial_y^\beta \, H_t)(x,y)|
\leq a \, t^{-d/2} \, t^{-(|\alpha| + |\beta|)/2} \, 
      e^{-b \, \frac{|x-y|^2}{t}} \, e^{\omega t}
\]
and 
\begin{eqnarray*}
\lefteqn{
|(\partial_x^\alpha \, \partial_y^\beta \, H_t)(x+h,y+k) 
     - (\partial_x^\alpha \, \partial_y^\beta \, H_t)(x,y)|
} \hspace{20mm}  \\*
& \leq & a \, t^{-d/2} \, t^{-(|\alpha| + |\beta|)/2} \, 
    \left( \frac{|h| + |k|}{\sqrt{t} + |x-y|} \right)^\kappa \, 
      e^{-b \, \frac{|x-y|^2}{t}}  \, e^{\omega t}
\end{eqnarray*}
for all $x,y \in \Omega$ and $h,k \in \Ri^d$ with $x+h,y+k \in \Omega$, 
$|h| + |k| \leq \tau \, \sqrt{t} + \tau' \, |x-y|$ and all $|\alpha|, | \beta| \le 1$.
These bounds are proved using Morrey and Campanato  spaces.
The idea of using these spaces in order to obtain Gaussian upper bounds together with 
H\"older regularity for  heat kernels of elliptic operators on $\Ri^d$ originates in a work of 
Auscher  \cite{Aus1}, see also ter Elst and Robinson \cite{ER15}
and for derivatives of the kernel on Lie groups see \cite{ER19}. 
Here the new difficulty is that we have boundary conditions and the approach  needs to  
be adjusted to this setting. 
In addition,  not only Gaussian upper bounds for the heat kernel are proved here 
but also Gaussian upper bounds and H\"older continuity for the derivatives 
$\partial_x^\alpha \, \partial_y^\beta \, H_t$ with $|\alpha|, | \beta | \le 1$.  
In order to obtain the necessary De Giorgi or energy estimates
for derivatives of weak solutions, we use estimates of Campanato \cite{Cam1}.

The previous  bounds on the heat kernel readily imply for the Green function 
$G_V$ the bounds
\begin{eqnarray*}
|(\partial_x^\alpha \, \partial_y^\beta \, G_V)(x,y)| 
& \leq & c \, |x-y|^{-(d-2+|\alpha| + |\beta|)}  \quad \mbox{and}    \\
|(\partial_x^\alpha \, \partial_y^\beta \, G_V)(x',y') 
    - (\partial_x^\alpha \, \partial_y^\beta \, G_V)(x,y)| 
& \leq & c \, \frac{(|x'-x| + |y'-y|)^\kappa}{|x-y|^{d-2+|\alpha| + |\beta| +\kappa}}   
\end{eqnarray*}
for all $x,x',y,y' \in \Omega$ with $x \neq y$ and
$|x-x'| + |y-y'| \leq \frac{1}{2} \, |x-y|$ if $d \geq 3$.
 If $\RRe V \geq 0$,  we have uniform constants $c$ (with respect to the 
coefficients $c_{kl}$ and $V$), a very useful fact when using approximation by smooth coefficients 
as we shall do in our proofs.  
If $d = 2$, then the estimates are the same when $|\alpha| + |\beta| \not= 0$.
Otherwise a logarithmic term appears.

These estimates on the 
Green function are used to  prove the  previous estimates on the Schwartz kernel 
$K_{\cn_V}$ of the Dirichlet-to-Neumann operator. 

We emphasize that if $c_{kl} = c_{lk}$ are real-valued then upper bounds  for 
$\nabla_x \nabla_y G$ are known 
(see for example Avellaneda and Lin \cite{AL1} and Kenig, Lin and Shen \cite{KLS}).
Note however that 
H\"older continuity of $\nabla_x \nabla_y G$ as stated above seems to be missing  
in the literature.  
 
 We return now to final step used in the proof of the Poisson bounds. 
Once $L_p$--$L_p$ estimates for the commutator $[\cn_V, M_g]$ are proved we obtain 
$L_p$--$L_q$ bounds for iterated commutators 
$\delta^j_g(\cn_V) := [ M_g, [\ldots,M_g, \cn_V]\ldots]]$ for all $j \in \{1,\dots,d\}$.
This together with  $L_p$--$L_q$
estimates for the semigroup $S^V$ is  used  to estimate the $L_1$--$L_\infty$ norm
of the iterated commutator $\delta^d_g(S^V_t) := [ M_g, [\ldots,[M_g, S^V_t]\ldots]]$.
 We then optimize over $g$ and obtain the Poisson bounds.

\subsection*{Notation} 
Throughout this paper we use the following notation. 
 For a function $R$ of two variables we denote by
$\partial_k^{(j)}R$ the $k^{\rm th}$-partial derivatives with respect to the 
$j^{\rm th}$ variable with $j = 1, 2$. 
We identify a uniformly continuous function 
on $\Omega$ with a uniformly continuous function on $\overline \Omega$. 
We emphasise that a function in $C^1(\Omega)$ 
is not bounded in general, nor
it is an element of $L_1(\Omega)$ in general, even if $\Omega$ is bounded.
We define 
\[
C^1(\overline \Omega) = \{ u \in C^1(\Omega) : 
\partial_k u \mbox{ is uniformly continuous for all } k \in \{ 1,\ldots,d \} \} 
.  \]
For a bounded domain $\Omega$ with Lipschitz boundary $\Gamma$,  let
$C^{0,1}(\Gamma)$ denote the space of Lipschitz continuous functions on $\Gamma$. 
It is endowed with the norm
\[
\| g \|_{C^{0,1}(\Gamma)} = \| g \|_{L_\infty(\Gamma)} 
   + \sup_{z,w \in \Gamma, \; z \not= w}  \frac{ | g(z) - g(w) |}{|z-w|}.
\]
For all $g \in  C^{0,1}(\Gamma)$ we use the notation 
${\rm Lip}_\Gamma(g) = \sup_{z,w \in \Gamma, \; z \not= w}  \frac{ | g(z) - g(w) |}{|z-w|}$. 
If $f \in L_\infty(\Omega)$ and $p \in [1,\infty]$, then we denote by $M_f$ the 
multiplication operator by the function $f$ on $L_p(\Omega)$.
Finally, the $L_p$--$L_q$ norm of an operator $T$ will be denotes by $\| T \|_{p\to q}$.

\section{Preliminaries and the first auxiliary results}\label{S2}

We assume throughout this section that $\Omega$ is a bounded 
Lipschitz domain of $\Ri^d$ with $d \geq 2$.
We assume that $c_{kl} = c_{lk}  \in L_\infty(\Omega, \Ri)$ such that
\[
\sum_{k,l=1}^d c_{kl}(x) \, \xi_k \, \overline{\xi_j} \ge \mu \, |\xi|^2
\]
for all $\xi \in \mathbb{C}^d$ and a.e.\ $x \in \Omega$, where $\mu > 0$ is a 
positive constant.
Let $V \in L_\infty (\Omega,\Ri)$ be a real-valued potential.
We define the space $H_V$ of harmonic functions for the operator
$-\sum_{k,l= 1}^d \partial_l\left(c_{kl} \partial_k \right) + V$ by
\[
H_V =\{ u \in W^{1,2}(\Omega):  
  - \sum_{k,l= 1}^d \partial_l \left(c_{kl} \, \partial_k\, u \right) + V u = 0 
    \mbox{ weakly on } \Omega \}.
\]
Here and in what follows 
$-\sum_{k,l= 1}^d \partial_l\left(c_{kl} \, \partial_k\, u \right)  + V u = 0$ 
{\bf weakly on $\Omega$} means that $u \in W^{1,2}(\Omega)$ and 
\[
\int_\Omega \sum_{k,l=1}^d c_{kl} \, (\partial_k u) \, \overline{\partial_l \chi} 
    + \int_\Omega V \, u \, \overline \chi  = 0 
\]
for all $\chi \in C_c^\infty(\Omega)$.

Define the continuous sesquilinear form 
$\gota_V \colon W^{1,2}(\Omega) \times W^{1,2}(\Omega) \to \Ci$ by
\begin{equation}
\gota_V(u,v) 
= \int_\Omega \sum_{k,l=1}^d c_{kl} \, (\partial_k u) \, \overline{\partial_l v} 
    + \int_\Omega V \, u \, \overline v.  
\label{formaV}
\end{equation}
It is clear that $H_V$ is a closed subspace of $W^{1,2}(\Omega)$
and 
\[
H_V = \{ u \in W^{1,2}(\Omega) : \gota_V(u,v) = 0 \mbox{ for all } v \in \ker \Tr \},  
\]
where $\Tr \colon W^{1,2}(\Omega) \to L_2(\Gamma)$ is the trace operator.

Define the form $\gota_V^D \colon W^{1,2}_0(\Omega) \times W^{1,2}_0(\Omega) \to \Ci$ by
$\gota_V^D = { \gota_V}|_{W^{1,2}_0(\Omega) \times W^{1,2}_0(\Omega)}$.
Then the associated operator is  $A_D + V$, where 
$A_D$ is the operator associated to the form $\gota_0^D$.
Formally, $A_D = -\sum_{k,l= 1}^d \partial_l\left(c_{kl} \, \partial_k \right)$,  
subject to the Dirichlet boundary condition. 

As in Section~2 in \cite{EO4}, one proves easily that if  $0 \notin \sigma(A_D+ V)$, 
then the  space $W^{1,2}(\Omega)$ has the decomposition
\begin{equation} \label{eq2.1}
 W^{1,2}(\Omega) = W_0^{1,2}(\Omega) \oplus H_V.  
 \end{equation}
In particular
\[
\Tr(H_V) = \Tr(W^{1,2}(\Omega)).
\]
A direct  corollary is that  $\Tr$ is  
injective as an operator from $H_V$ into $L_2(\Gamma)$.
Thus, we may define the form 
$\gotb_V \colon \Tr(W^{1,2}(\Omega)) \times \Tr(W^{1,2}(\Omega)) \to \Ci$ by 
\[
\gotb_V(\varphi, \psi) = \gota_V(u,v)
,   \]
where $u, v \in H_V$ are such that $\Tr u = \varphi$ and $\Tr v = \psi$.
One obtains as in \cite{EO4} that $\gotb_V$
is bounded from below and is a closed symmetric form.
Hence there exists an associated self-adjoint operator $\cn_V$ associated with $\gotb_V$.
This is the Dirichlet-to-Neumann operator. 

Let $u \in W^{1,2}(\Omega)$ and $f \in L_2(\Omega)$.
We say that $\ca u = f$ if $\gota_0(u,v) = (f,v)_{L_2(\Omega)}$ 
for all $v \in W^{1,2}_0(\Omega)$.
In particular, $u \in H_V$ if and only if $\ca u = - V \, u$.
If $u \in W^{1,2}(\Omega)$, then we say that $\ca u \in L_2(\Omega)$
if there exists an $f \in L_2(\Omega)$ such that $\ca u = f$.
Let $u \in W^{1,2}(\Omega)$ with $\ca u \in L_2(\Omega)$.
Then we say that $u$ has a {\bf weak conormal derivative} 
if there exists a $\psi \in L_2(\Gamma)$ such that 
\begin{equation}
\int_\Omega \sum_{k,l=1}^d c_{kl} \, (\partial_k u) \, \overline{\partial_l v}
   - \int_\Omega (\ca u) \, \overline v
= \int_\Gamma \psi \, \overline{\Tr v}
\label{elpbdton201;15}
\end{equation}
for all $v \in W^{1,2}(\Omega)$.
In that case $\psi$ is unique and we write $\partial_\nu^C u = \psi$.

We next present a couple of equivalent descriptions for the 
Dirichlet-to-Neumann operator $\cn_V$.

\begin{lemma} \label{ldtnpc239}
Let $\varphi,\psi \in L_2(\Gamma)$.
Then the following are equivalent.
\begin{tabeleq}
\item \label{ldtnpc239-1}
$\varphi \in D(\cn_V)$ and $\cn_V \varphi = \psi$.
\item \label{ldtnpc239-2}
There exists a $u \in H_V$ such that $\Tr u = \varphi$ and $\partial_\nu^C u = \psi$.
\item \label{ldtnpc239-3}
There exists a $u \in W^{1,2}(\Omega)$ such that $\Tr u = \varphi$ and 
\begin{equation}
\gota_V(u,v) = (\psi, \Tr v)_{L_2(\Gamma)}
\label{eSdtnpc2;40}
\end{equation}
for all $v \in W^{1,2}(\Omega)$.
\end{tabeleq}
\end{lemma}
\begin{proof}
`\ref{ldtnpc239-1}$\Rightarrow$\ref{ldtnpc239-2}'.
By definition there exists a $u \in H_V$ such that $\Tr u = \varphi$ and 
\begin{equation}
\int_\Omega \sum_{k,l=1}^d c_{kl} (\partial_k u) \, \overline{\partial_l v}
    + \int_\Omega V \, u  \, \overline v  
= \gota_V(u,v)
= \int_\Gamma  \psi \, \overline{\Tr v} 
\label{elpbdton201;10}
\end{equation}
for all $v \in H_V$.
Since $u \in H_V$ obviously (\ref{elpbdton201;10}) is valid for all 
$v \in W^{1,2}_0(\Omega)$.
Then by (\ref{eq2.1}) one deduces that (\ref{elpbdton201;10}) is valid 
for all $v \in W^{1,2}(\Omega)$.
Moreover, since $\ca u + V u = 0$ it follows from (\ref{elpbdton201;15})
that $\partial_\nu^C u = \psi$.

`\ref{ldtnpc239-2}$\Rightarrow$\ref{ldtnpc239-3}'.
Since $u \in H_V$ it follows that $\ca u + V u = 0$.
Then (\ref{elpbdton201;15}) implies that (\ref{elpbdton201;10}) is 
valid for all $v \in W^{1,2}(\Omega)$. 
But this is just (\ref{eSdtnpc2;40}).

`\ref{ldtnpc239-3}$\Rightarrow$\ref{ldtnpc239-1}'.
Now (\ref{elpbdton201;10}) with $v \in W^{1,2}_0(\Omega)$ gives 
$\ca u + V \, u = 0$, that is $u \in H_V$.
By definition of $\gotb_V$ one deduces that 
$\gotb_V(\varphi,\tau) = (\psi,\tau)_{L_2(\Gamma)}$
for all $\tau \in \Tr (H_V)$.
Hence Condition~\ref{ldtnpc239-1} is valid.
\end{proof}

For additional information regarding Condition~\ref{ldtnpc239-3} we refer to
\cite{AE2}.

The self-adjoint operator $-\cn_V$ generates a quasi-contraction holomorphic semigroup 
$S^V$ on $L_2(\Gamma)$.  
When $V = 0$ we write for simplicity  $\cn = \cn_0$ and $S = S^0$.
We also denote by $\lambda_1$ the first eigenvalue of the Dirichlet-to-Neumann operator 
$\cn_V$ without specifying the dependence on $V$. 

We summarize in the following two theorems some important properties of the 
semigroups $S^V$ and $S$.
The proofs are  the same as in 
\cite{EO4} Section~2, where these results are proved in the case 
$c_{kl} = \delta_{kl}$. 

\begin{thm}\label{th2.2} 
Suppose that $\Omega$ is bounded Lipschitz, $c_{kl} = c_{lk} \in L_\infty(\Omega,\Ri)$ 
satisfying the ellipticity condition and $V \in L_\infty(\Omega,\Ri)$ with
$0 \notin \sigma(A_D+V)$. 
\begin{tabel}
\item \label{th2.2-1} 
If  $A_D + V \ge 0$  
then the semigroup $S^V$ is positive (it maps positive functions on 
$\Gamma$ into positive functions). 
\item \label{th2.2-2} 
If $V \ge 0$ then $S^V$ is sub-Markovian.
Therefore $S^V$ acts as a contraction $C_0$-semigroup on $L_p(\Gamma)$ for all 
$p \in [1, \infty).$ 
\item \label{th2.2-3} 
If $V \ge 0$ then $S_t^V \varphi \le S_t \varphi$ for all $ t \ge 0$ and all 
positive $\varphi \in L_2(\Gamma)$. 
\end{tabel}
\end{thm}

We note that in the first assertion, if the assumption $A_D + V \ge 0$ is not satisfied 
then the semigroup $S^V$ may not be positive for all $t > 0$ (see \cite{Daners4}). 
This is the reason why our Poisson bound in the main theorem is formulated for 
$| K^V_t(x,y)|$ and not  for $K^V_t(x,y)$. 

Now we state $L_p$--$L_q$ estimates for the semigroup $S^V$.
Note that $\lambda_1 \geq 0$ in the next theorem.

\begin{thm}\label{th2.3}
Let $0 \le V \in L_\infty(\Omega)$ and let
$\lambda_1 \in \sigma(\cn_V)$ be the first eigenvalue of $\cn_V$.
Then for all $1 \le p \le q \le \infty$ and $t > 0$ the operator 
$S_t^V$ is bounded from $L_p(\Gamma)$ into $L_q(\Gamma)$.
Moreover, there exists a $C > 0$ such that 
\[
\|S_t^V \|_{p \to q} 
\le C \, (t \wedge 1)^{-(d-1)(\frac{1}{p} - \frac{1}{q})} \, e^{-\lambda_1 t} 
\]
for all $t > 0$ and $p,q \in [1,\infty]$ with $p \leq q$.
\end{thm}

Actually, it will follow from Theorem~\ref{thm1.1} that 
this theorem is also valid for general 
$V \in L_\infty(\Omega)$, possibly with $\lambda_1 < 0$.

We finish this section with a known formula. 
Again let $V \in L_\infty(\Omega)$ with $0 \notin \sigma(A_D+V)$. 
Define the {\bf harmonic lifting} $\gamma_V \colon \Tr(W^{1,2}(\Omega)) \to H^1(\Omega)$
as follows.
Given $\varphi \in H^{1/2}(\Gamma) := \Tr(W^{1,2}(\Omega))$ 
it follows from (\ref{eq2.1}) that 
there exists a unique $u \in H_V$ with $\Tr u = \varphi$. 
We define 
\[ 
\gamma_V  \varphi  := u
. \]
There is a simple relation between $\gamma_V$ and $\gamma_0$, where the latter is 
the harmonic lifting in case $V = 0$.
Let $\varphi \in \Tr(W^{1,2}(\Omega))$. 
Write $u_0 = \gamma_0 \varphi$ and $u = \gamma_V \varphi$.
Then $u - u_0 \in W^{1,2}_0(\Omega)$.
Moreover, $(\ca + V) u = 0$ and $\ca u_0 = 0$.
So $(A_D + V) (u - u_0) = (\ca + V) (u - u_0) = - V u_0$.
Therefore $u - u_0 = - (A_D + V)^{-1} \, M_V u_0$
and 
\begin{equation}
\gamma_V
= \gamma_0 - (A_D + V)^{-1} \, M_V \, \gamma_0 .
\label{eSdtnpc2;30}
\end{equation}
We shall use this relation in Sections~\ref{Sdtnpc5new}.

\section{Heat kernel bounds for elliptic operators on $C^{1+\kappa}$-domains}

\label{S3}
\label{Sdtnpc4}

Let $\Omega \subset \Ri^d$ be an open set.
Let $\mu,M > 0$.
We define $\ce(\Omega,\mu,M)$
to be the set of all measurable $C \colon \Omega \to \Ci^{d \times d}$ such that 
\[
\begin{array}{ll}
\RRe \langle C(x) \, \xi, \xi \rangle \geq \mu \, |\xi|^2
   & \mbox{for all } x \in \Omega \mbox{ and } \xi \in \Ci^d \mbox{, and,}  \\[5pt]
\|C(x)\| \leq M
   & \mbox{for all } x \in \Omega .
\end{array}
\]
where $\|C(x)\|$ is 
the $\ell_2$-norm of $C(x)$ in $\Ci^d$ and $\langle \cdot , \cdot \rangle$ is the 
inner product on $\Ci^d$.
Here and in the sequel $c_{kl}(x)$ is the appropriate
matrix coefficient of $C(x)$.
We define $\ce(\Omega) = \bigcup_{\mu,M > 0} \ce(\Omega,\mu,M)$.
For all $C \in \ce(\Omega)$ define the closed sectorial form
\[
\gota \colon W^{1,2}_0(\Omega) \times W^{1,2}_0(\Omega) \to \Ci
\]
 by
\[
\gota(u,v)
= \int_\Omega \sum_{k,l=1}^d c_{kl} \, (\partial_k u) \, \overline{(\partial_l v)}
\]
and let $A^C_D$ be the associated operator.
Note that $A^C_D$ has Dirichlet boundary conditions.
If no confusion is possible then we drop the $C$ and write 
$A_D = A^C_D$.

Let $\kappa \in (0,1)$.
The space $C^\kappa(\Omega)$ is the space of all H\"older continuous functions 
of order $\kappa$ on $\Omega$ with semi-norm
\[
|||u|||_{C^\kappa(\Omega)}
= \sup \{ \frac{|u(x) - u(y)|}{|x-y|^\kappa} : x,y \in \Omega, \; 0 < |x-y| \leq 1 \}
.  \]
Let $\mu,M > 0$.
We define $\ce^\kappa(\Omega,\mu,M)$
to be the set of all continuous $C \in \ce(\Omega,\mu,M)$ such that 
\[
\begin{array}{ll}
c_{kl} \in C^\kappa(\Omega)
   & \mbox{for all } k,l \in \{ 1,\ldots,d \} \mbox{, and}  \\[5pt]
|||c_{kl}|||_{C^\kappa(\Omega)} \leq M  
   & \mbox{for all } k,l \in \{ 1,\ldots,d \} .
\end{array}
\]
Define $\ce^\kappa(\Omega) = \bigcup_{\mu,M > 0} \ce^\kappa(\Omega,\mu,M)$.

The main theorem of this section is the following.

\begin{thm} \label{tdtnpc401}
Let $\kappa,\tau' \in (0,1)$ and $\mu,M,\tau > 0$.
Let $\Omega \subset \Ri^d$ be an open bounded set with a $C^{1+\kappa}$-boundary.
Then there exist $a,b > 0$ and $\omega \in \Ri$ such that for every $C \in \ce^\kappa(\Omega,\mu,M)$
and $V \in L_\infty(\Omega,\Ri)$ with $\|V\|_\infty \leq M$ there exists 
a function $(t,x,y) \mapsto H_t(x,y)$ from 
$(0,\infty) \times \Omega \times \Omega$ into $\Ci$
such that the following is valid.
\begin{tabel}
\item \label{tdtnpc401-1}
The function $(t,x,y) \mapsto H_t(x,y)$ is continuous 
from $(0,\infty) \times \Omega \times \Omega$ into $\Ci$.
\item \label{tdtnpc401-2}
For all $t \in (0,\infty)$ the function $H_t$ is the kernel of 
the operator $e^{-t (A_D + V)}$.
\item \label{tdtnpc401-3}
For all $t \in (0,\infty)$ the
function $H_t$ is once differentiable in each variable and the 
derivative with respect to one variable is differentiable in the 
other variable.
Moreover, for every multi-index $\alpha,\beta$ with 
$0 \leq |\alpha|,|\beta| \leq 1$ one has 
\[
|(\partial_x^\alpha \, \partial_y^\beta \, H_t)(x,y)|
\leq a \, t^{-d/2} \, t^{-(|\alpha| + |\beta|)/2} \, 
      e^{-b \, \frac{|x-y|^2}{t}} \, e^{\omega t}
\]
and 
\begin{eqnarray*}
\lefteqn{
|(\partial_x^\alpha \, \partial_y^\beta \, H_t)(x+h,y+k) 
     - (\partial_x^\alpha \, \partial_y^\beta \, H_t)(x,y)|
} \hspace{20mm}  \\*
& \leq & a \, t^{-d/2} \, t^{-(|\alpha| + |\beta|)/2} \, 
    \left( \frac{|h| + |k|}{\sqrt{t} + |x-y|} \right)^\kappa \, 
      e^{-b \, \frac{|x-y|^2}{t}}  \, e^{\omega t}
\end{eqnarray*}
for all $x,y \in \Omega$ and $h,k \in \Ri^d$ with $x+h,y+k \in \Omega$ and 
$|h| + |k| \leq \tau \, \sqrt{t} + \tau' \, |x-y|$.
\item \label{tdtnpc401-4}
If $\RRe V \geq 0$, then $\omega < 0$.
\end{tabel}
\end{thm}

The proof requires a lot of preparation.
First we introduce the pointwise Morrey and Campanato semi-norms
as in \cite{ERe2}.

Let $\Omega \subset \Ri^d$ be open.
For all $x \in \Ri^d$ and $r > 0$ define $\Omega(x,r) = \Omega \cap B(x,r)$.
For all $\gamma \in [0,d]$, $R_e \in (0,1]$ and $x \in \Omega$ define 
$\|\cdot\|_{M,\gamma,x,\Omega,R_e} \colon L_2(\Omega) \to [0,\infty]$ by
\[
\|u\|_{M,\gamma,x,\Omega,R_e}
= \sup_{r \in (0,R_e]}
     \Big( r^{-\gamma} \int_{\Omega(x,r)} |u|^2 \Big)^{1/2}
 .  \]
Next, for all $\gamma \in [0,d+2]$, $R_e \in (0,1]$ and $x \in \Omega$ define 
$|||\cdot|||_{\cm,\gamma,x,\Omega,R_e} \colon L_2(\Omega) \to [0,\infty]$ by
\[
|||u|||_{\cm,\gamma,x,\Omega,R_e}
= \sup_{r \in (0,R_e]}
     \Big( r^{-\gamma} \int_{\Omega(x,r)} |u - \langle u\rangle_{\Omega(x,r)}|^2 \Big)^{1/2}
,  \]
where for an $L_2$ function $v$ we denote by 
$\langle v \rangle_{D} = \frac{1}{|D|} \int_D v$
the average of $v$ over a bounded measurable subset $D$ of the domain of $v$ with $|D| > 0$.
If no confusion is possible, then we drop the dependence of $\Omega$.

As for Morrey and Campanato spaces, one has the following 
connections.

\begin{lemma} \label{ldtnpc402}
\mbox{}
\begin{tabel}
\item \label{ldtnpc402-1}
For all $\gamma \in [0,d)$, $\tilde c > 0$ and $R_e \in (0,1]$
there exist $c_1,c_2 > 0$ such that 
\[
|||u|||_{\cm,\gamma,x,R_e}^2 
\leq \|u\|_{M,\gamma,x,R_e}^2
\leq c_1 \, |||u|||_{\cm,\gamma,x,R_e}^2 + c_2 \, \int_{\Omega(x,R_e)} |u|^2
\]
for all open $\Omega \subset \Ri^d$, $x \in \Omega$ and $u \in L_2(\Omega)$
such that $|\Omega(x,r)| \geq \tilde c \, r^d$ for all $r \in (0,R_e]$.
\item \label{ldtnpc402-2}
Let $\Omega \subset \Ri^d$ be open, 
$\gamma \in (d,d+2)$, $\tilde c > 0$, $x \in \Omega$, $u \in L_2(\Omega)$
and $R_e \in (0,1]$.
Assume that $|||u|||_{\cm,\gamma,x,R_e} < \infty$ and 
$|\Omega(x,r)| \geq \tilde c \, r^d$ for all $r \in (0,R_e]$.
Then $\lim_{R \downarrow 0} \langle u \rangle_{\Omega(x,R)}$ exists.
Write $\hat u(x) = \lim_{R \downarrow 0} \langle u \rangle_{\Omega(x,R)}$.
Then 
\[
|\langle u \rangle_{\Omega(x,R)} - \hat u(x)|
\leq \frac{2^{1+d/2}}{\sqrt{\tilde c} (1 - 2^{-(\gamma - d)/2})} \,
         R^{(\gamma - d)/2} \,  
      |||u|||_{\cm,\gamma,x,R_e} 
\]
for all $R \in (0,R_e]$.
\item \label{ldtnpc402-3}
Let $\gamma \in (d,d+2)$ and $\tilde c > 0$.
Then there exists a $c > 0$ such that 
\[
|\hat u(x) - \hat u(y)|
\leq c \, (|||u|||_{\cm,\gamma,x,R_e} + |||u|||_{\cm,\gamma,y,R_e}) \, |x-y|^{(\gamma - d)/2}
\]
for all open $\Omega \subset \Ri^d$, 
$x,y \in \Omega$, $R_e \in (0,1]$ and $u \in L_2(\Omega)$ such that 
$|||u|||_{\cm,\gamma,x,R_e} < \infty$, 
$|||u|||_{\cm,\gamma,y,R_e} < \infty$, $|x-y| \leq \frac{R_e}{2}$
and, in addition, $|\Omega(x,r)| \geq \tilde c \, r^d$ 
and $|\Omega(y,r)| \geq \tilde c \, r^d$ for all $r \in (0,R_e]$, where 
$\hat u(x)$ and $\hat u(y)$ are as in {\rm \ref{ldtnpc402-2}}.
\end{tabel}
\end{lemma}
\begin{proof}
See the appendix in \cite{ERe2}.
\end{proof}

Let $\Omega \subset \Ri^d$ be open and $C \in \ce(\Omega$.
Let $u \in W^{1,2}(\Omega)$.
Then we say that 
{\bf $A^C u = 0$ weakly on $\Omega$} if 
\begin{equation}
\int_\Omega \sum_{k,l=1}^d c_{kl} \, (\partial_k u) \, \overline{(\partial_l v)}
= 0
\label{eSdtnpc4;1}
\end{equation}
for all $v \in C_c^\infty(\Omega)$.
Then by density (\ref{eSdtnpc4;1}) is valid for all $v \in W^{1,2}_0(\Omega)$.

We need various De Giorgi estimates.
First we need interior De Giorgi estimates.

\begin{lemma} \label{ldtnpc403}
Let $\Omega \subset \Ri^d$ be open and $\mu,M > 0$.
Then there exists a $c_{DG} > 0$
such that 
\begin{eqnarray}
\int_{B(x,r)} |\nabla u|^2 
& \leq & c_{DG} \, \Big( \frac{r}{R} \Big)^d \int_{B(x,R)} |\nabla u|^2 \mbox{ and }
\label{eldtnpc403;1}  \\
\int_{B(x,r)} |\partial_k u - \langle \partial_k u\rangle_{B(x,r)}|^2
& \leq & c_{DG} \, \Big( \frac{r}{R} \Big)^{d+2} 
          \int_{B(x,R)} |\partial_k u - \langle \partial_k u\rangle_{B(x,R)}|^2 
\label{eldtnpc403;2}
\end{eqnarray}
for all $k \in \{ 1,\ldots,d \} $,
$x \in \Omega$, $R \in (0,1]$, $r \in (0,R]$, $u \in W^{1,2}(B(x,R))$ and 
{\bf constant} coefficient 
$C \in \ce(\Omega,\mu,M)$  satisfying $B(x,R) \subset \Omega$ and
$A^C u = 0$ weakly on $B(x,R)$.
\end{lemma}
\begin{proof}
The estimate (\ref{eldtnpc403;1}) was first proved by De Giorgi.
For a proof, see Corollario [7.I] in Campanato \cite{Cam1}.
The estimate (\ref{eldtnpc403;2}) is in Corollario [7.II] 
of the same paper. 
The uniformity of the constants follows from the proof.
Note that the coefficients can be complex and non-symmetric.
The proofs in \cite{Cam1} also work for complex and non-symmetric 
coefficients with obvious modifications.
\end{proof}

We also need De Giorgi estimates on the boundary.
Define 
\[
E = (-4,4)^d
\quad \mbox{and} \quad
E^- = (-4,4)^{d-1} \times (-4,0)
\]
the open cube in $\Ri^d$ and 
its lower half $E^-$.
The {\bf midplate} is $P = E \cap \{ x \in \Ri^d : x_d = 0 \} $
We also need the cubes, lower halfs and midplates with half and a quarter 
sizes, denoted by $\frac{1}{2} \, E$, $\frac{1}{2} \, E^-$, $\frac{1}{2} \, P$, etc.
Recall that $E^-(x,r) = E^- \cap B(x,r)$ for all $x \in \Ri^d$ and $r > 0$.

\begin{lemma} \label{ldtnpc404}
There exists a $c_{DG} > 0$
such that 
\begin{eqnarray}
\int_{E^-(x,r)} |\partial_i u|^2 
& \leq & c_{DG} \, \Big( \frac{r}{R} \Big)^{d+2} \int_{E^-(x,R)} |\partial_i u|^2  ,
\label{eldtnpc404;1}  \\
\int_{E^-(x,r)} |\partial_d u|^2 
& \leq & c_{DG} \, \Big( \frac{r}{R} \Big)^d \int_{E^-(x,R)} |\partial_d u|^2  \mbox{ and}  
\label{eldtnpc404;2}  \\
\int_{E^-(x,r)} |\partial_d u - \langle \partial_d u\rangle_{E^-(x,r)}|^2
& \leq & c_{DG} \, \Big( \frac{r}{R} \Big)^{d+2} 
        \int_{E^-(x,R)} |\partial_d u - \langle \partial_d u\rangle_{E^-(x,R)}|^2
\label{eldtnpc404;3} 
\end{eqnarray}
for all $i \in \{ 1,\ldots,d-1 \} $, 
$x \in \frac{1}{2} \, P$, $R \in (0,1]$, $r \in (0,R]$, $u \in W^{1,2}(E^-(x,R))$ and 
{\bf constant} coefficient 
$C \in \ce(E^-,\mu,M)$ satisfying $(\Tr u)|_{P \cap E(x,R)} = 0$ and
$A^C u = 0$ weakly on $E^-(x,R)$.
\end{lemma}
\begin{proof}
Estimate (\ref{eldtnpc404;1}) is Corollario~[11.I] and the other two 
are in Lemma~[11.II] in \cite{Cam1}.
Again the uniformity of the constants follows from the proof and the 
coefficients can be complex.
\end{proof}

We now turn to regularity.
Close to the boundary we have to take a coordinate transformation. 
For good bounds we have to combine the coordinate transformation 
together with the regularity improvement theorem.
In the next lemma we first collect some easy estimates.

\begin{lemma} \label{ldtnpc410}
Let $\kappa \in (0,1)$ and $K \geq 1$.
Let $\Omega,U \subset \Ri^d$ open.
Let $\Phi$ be a $C^{1+\kappa}$-diffeomorphism from $U$ onto $E$ such that 
$\Phi(U \cap \Omega) = E^-$ and $\Phi(U \cap \partial \Omega) = P$.
Suppose that $K$ is larger than the Lipschitz constant for $\Phi$ and $\Phi^{-1}$.
Moreover, suppose that $|||(D \Phi)_{ij}|||_{C^\kappa} \leq K$ and 
$|||(D (\Phi^{-1}))_{ij}|||_{C^\kappa} \leq K$ for all $i,j \in \{ 1,\ldots,d \} $,
where $D \Phi$ denotes the derivative of $\Phi$.
Then one has the following.
\begin{tabel} 
\item \label{ldtnpc410-1}
Let $\mu,M > 0$.
Let $C \in \ce^\kappa(\Omega,\mu,M)$.
Define $C^\Phi \colon E^- \to \Ci^{d \times d}$ by 
\begin{equation}
C^\Phi(y)
= \frac{1}{\big| \det(D\Phi)(\Phi^{-1}(y))\big|} \, 
      (D\Phi)(\Phi^{-1}(y)) \;
      A(\Phi^{-1}(y)) \; \bigl( D \Phi \bigr)^T
      (\Phi^{-1} (y))
.  
\label{eldtnpc410;1}
\end{equation}
Then $C^\Phi \in \ce^\kappa(E^-,(d! K^{d+2})^{-1} \mu, d! d^2 K^{d+2} M)$.
Moreover, if $u,v \in W^{1,2}(U \cap \Omega)$, then 
\[
\sum_{k,l=1}^d \int_{U \cap \Omega} c_{kl} \, (\partial_k u) \, \overline{(\partial_l v)}
= \sum_{k,l=1}^d \int_{E^-} 
   (C^\Phi)_{kl} \, (\partial_k (u \circ \Phi^{-1})) \, \overline{(\partial_l (v \circ \Phi^{-1}))}
.  \]
\item \label{ldtnpc410-2}
If $x,x' \in U \cap \Omega$, then $|\Phi(x) - \Phi(x')| \leq K \, |x-x'|$.
Conversely, if $y,y' \in E^-$, then $|\Phi^{-1}(y) - \Phi^{-1}(y')| \leq K \, |y-y'|$.
\item \label{ldtnpc410-4}
If $u \in W^{1,2}(\Omega)$, then 
\[
(d \, K)^{-1} \, \|\nabla(u \circ \Phi^{-1})\|_{L_\infty(\frac{1}{2} E^-)}
\leq \|\nabla u\|_{L_\infty(\Phi^{-1}(\frac{1}{2} E^-))}
\leq d \, K \, \|\nabla(u \circ \Phi^{-1})\|_{L_\infty(\frac{1}{2} E^-)}
,  \]
possibly both norms are infinite.
\item \label{ldtnpc410-6}
If $u \in W^{1,2}_0(\Omega)$, then $(\Tr (u \circ \Phi^{-1}))|_P = 0$.
\end{tabel}
\end{lemma}
\begin{proof}
Statements~\ref{ldtnpc410-1}--\ref{ldtnpc410-4} are elementary.
For the proof of Statement~\ref{ldtnpc410-6}, first note that the 
map $v \mapsto v \circ \Phi^{-1}$ is continuous from $W^{1,2}(\Omega)$ into 
$W^{1,2}(E^-)$ and the map $v \mapsto \one_P \cdot \Tr v$ is continuous 
from $W^{1,2}(E^-)$ into $L_2(\partial E^-)$.
So the map $v \mapsto \one_P \cdot \Tr (v \circ \Phi^{-1})$ is continuous
from $W^{1,2}(\Omega)$ into $L_2(\partial E^-)$.
Obviously $\one_P \cdot \Tr (v \circ \Phi^{-1}) = 0$ for all $v \in C_c^\infty(\Omega)$.
Then Statement~\ref{ldtnpc410-6} follows from the density of $C_c^\infty(\Omega)$ 
in $W^{1,2}_0(\Omega)$.
\end{proof}

The first regularity lemma is with half-balls and points on the boundary.
It is a variation of Teorema~[13.I] in \cite{Cam1}, with an additional 
term $(f_0,v)_{L_2(\Omega)}$.
The most interesting case occurs for $\delta = 0$ in the next lemma, 
but we also need the lemma with $\delta > 0$ to avoid a technical 
complication in the proof of Proposition~\ref{pdtnpc414}.

\begin{lemma} \label{ldtnpc406}
Let $\kappa \in (0,1)$, $K \geq 1$, $\delta \in [0,\kappa]$ and $\mu,M > 0$.
Then there exists a $c \geq 1$ such that the following is valid.

Let $\Omega,U \subset \Ri^d$ open.
Let $\Phi$ be a $C^{1+\kappa}$-diffeomorphism from $U$ onto $E$ such that 
$\Phi(U \cap \Omega) = E^-$ and $\Phi(U \cap \partial \Omega) = P$.
Suppose that $K$ is larger than the Lipschitz constant for $\Phi$ and $\Phi^{-1}$.
Moreover, suppose that $|||(D \Phi)_{ij}|||_{C^\kappa} \leq K$ and 
$|||(D (\Phi^{-1}))_{ij}|||_{C^\kappa} \leq K$ for all $i,j \in \{ 1,\ldots,d \} $,
where $D \Phi$ denotes the derivative of $\Phi$.
Let $C \in \ce^\kappa(\Omega,\mu,M)$, $u \in W^{1,2}_0(\Omega)$ and 
$f_0,f_1,\ldots,f_d \in L_2(\Omega)$ and suppose that 
\begin{equation}
\int_\Omega \sum_{k,l=1}^d c_{kl} \, (\partial_k u) \, \overline{(\partial_l v)}
= (f_0,v)_{L_2(\Omega)} - \sum_{k=1}^d (f_k, \partial_k v)_{L_2(\Omega)}
\label{eldtnpc406;14}
\end{equation}
for all $v \in W^{1,2}_0(\Omega)$.
Define $\tilde u \colon E^- \to \Ci$ by $\tilde u = u \circ \Phi^{-1}$.
Let $x \in \frac{1}{2} \, P$.
For all $\rho \in (0,1]$ define 
\[
\Psi(\rho) 
= \int_{E^-(x,\rho)} |\partial_d \tilde u - \langle \partial_d \tilde u \rangle_{E^-(x,\rho)}|^2
   + \sum_{i=1}^{d-1} \int_{E^-(x,\rho)} |\partial_i \tilde u|^2
,  \]
set $\gamma = d + 2 \kappa - \delta$
and 
\[
c_0 = \|f_0 \circ \Phi^{-1}\|_{M,\gamma - 2,x,E^-,1}
          + \sum_{k=1}^d |||f_k \circ \Phi^{-1}|||_{\cm,\gamma,x,E^-,1}
          + \|\nabla \tilde u\|_{M,d - \delta,x,E^-,1}
.  \]
Then 
\[
\Psi(r)
\leq c \Big( \frac{r}{R} \Big)^\gamma \Psi(R)
   + c \, c_0^2 \, r^\gamma
\]
for all $r,R \in (0,1]$ with $0 < r \leq R$.
\end{lemma}
\begin{proof}
Let $c_{DG} > 0$ be as in Lemma~\ref{ldtnpc404}, but with $\mu$ replaced by 
$(d! K^{d+2})^{-1} \mu$ and $M$ replaced by $d! d^2 K^{d+2} M$.
Let $R \in (0,1]$.
Let $C^\Phi$ be as in (\ref{eldtnpc410;1}).
Since $C^\Phi$ is H\"older continuous, it extends uniquely to a continuous function
on $\overline{E^-}$, which we also denote by $C^\Phi$.
We will freeze the coefficients of $C^\Phi$ at $x$.
There exists a unique $\hat v \in W^{1,2}_0(E^-(x,R))$ such that 
\begin{equation}
\sum_{k,l=1}^d \int_{E^-(x,R)} (C^\Phi)_{kl}(x) \, (\partial_k \hat v) \, \overline{\partial_l \tau}
= \sum_{k,l=1}^d \int_{E^-(x,R)} (C^\Phi)_{kl}(x) \, (\partial_k \tilde u) \, \overline{\partial_l \tau}
\label{ldtnpc406;13}
\end{equation}
for all $\tau \in W^{1,2}_0(E^-(x,R))$.
Define $v \colon \Omega \to \Ci$ by
\[
v(y) = \left\{ \begin{array}{ll}
   \hat v(\Phi(y)) & \mbox{if } y \in \Phi^{-1}(E^-(x,R)) ,  \\[5pt]
   0 & \mbox{if } y \in \Omega \setminus \Phi^{-1}(E^-(x,R)) .
               \end{array} \right.
\]
Then $v \in W^{1,2}_0(\Omega)$.
Set $w = u - v$, $\tilde v = v \circ \Phi^{-1}$ and $\tilde w = w \circ \Phi^{-1}$.
Clearly $w \in W^{1,2}_0(\Omega)$ and $\Tr \tilde w|_P = 0$ by Lemma~\ref{ldtnpc410}\ref{ldtnpc410-6}.
Moreover, $A^{(C^\Phi)(x)} (\tilde w) = 0$ weakly on $E^-(x,R)$ by (\ref{ldtnpc406;13}).
Let $r \in (0,R]$.
Using (\ref{eldtnpc404;3}), one deduces that 
\begin{eqnarray}
\lefteqn{
\int_{E^-(x,r)} |\partial_d \tilde u - \langle \partial_d \tilde u \rangle_{E^-(x,r)}|^2
} \hspace*{10mm}  \nonumber  \\*
& \leq & \int_{E^-(x,r)} |\partial_d \tilde u - \langle \partial_d \tilde w \rangle_{E^-(x,r)}|^2   \nonumber  \\
& \leq & 2 \int_{E^-(x,r)} |\partial_d \tilde w - \langle \partial_d \tilde w \rangle_{E^-(x,r)}|^2
   + 2 \int_{E^-(x,r)} |\nabla \tilde v|^2   \nonumber  \\
& \leq & 2 c_{DG} \, \Big( \frac{r}{R} \Big)^{d+2} 
   \int_{E^-(x,R)} |\partial_d \tilde w - \langle \partial_d \tilde w \rangle_{E^-(x,R)}|^2
   + 2 \int_{E^-(x,R)} |\nabla \tilde v|^2   \nonumber  \\
& \leq & 4 c_{DG} \, \Big( \frac{r}{R} \Big)^{d+2} 
   \int_{E^-(x,R)} |\partial_d \tilde u - \langle \partial_d \tilde u \rangle_{E^-(x,R)}|^2
   + (4 c_{DG} + 2) \int_{E^-(x,R)} |\nabla \tilde v|^2 
.
\label{eldtnpc406;1}
\end{eqnarray}
Similarly, with (\ref{eldtnpc404;1}) one deduces that 
\[
\int_{E^-(x,r)} |\partial_i \tilde u|^2
\leq 4 c_{DG} \, \Big( \frac{r}{R} \Big)^{d+2} \int_{E^-(x,R)} |\partial_i \tilde u|^2
   + (4 c_{DG} + 2) \int_{E^-(x,R)} |\nabla \tilde v|^2 
\]
for all $i \in \{ 1,\ldots,d-1 \} $.
So 
\begin{equation}
\Psi(r)
\leq 4 c_{DG} \, \Big( \frac{r}{R} \Big)^{d+2} \Psi(R)
   + (4 c_{DG} + 2) d \int_{E^-(x,R)} |\nabla \tilde v|^2 
.
\label{ldtnpc406;15}
\end{equation}

We next estimate $\int_{B(x,R)} |\nabla \tilde v|^2$.
Ellipticity, the equality $\tilde v|_{E^-(x,R)} = \hat v|_{E^-(x,R)}$,
(\ref{ldtnpc406;13}), Lemma~\ref{ldtnpc410}\ref{ldtnpc410-1}
and (\ref{eldtnpc406;14}) give
\begin{eqnarray*}
\lefteqn{
(d! K^{d+2})^{-1} \mu \int_{E^-(x,R)} |\nabla \tilde v|^2 
} \hspace{10mm} \\*
& \leq & \RRe \sum_{k,l=1}^d \int_{E^-(x,R)} 
     (C^\Phi)_{kl}(x) \, (\partial_k \hat v) \, \overline{\partial_l \hat v}  \\
& = & \RRe \sum_{k,l=1}^d \int_{E^-(x,R)} 
     (C^\Phi)_{kl}(x) \, (\partial_k \tilde u) \, \overline{\partial_l \hat v}  \\
& = & \RRe \sum_{k,l=1}^d \int_{E^-(x,R)} 
     (C^\Phi)_{kl} \, (\partial_k \tilde u) \, \overline{\partial_l \tilde v}  
   + \RRe \sum_{k,l=1}^d \int_{E^-(x,R)} \Big( C^\Phi)_{kl}(x) - (C^\Phi)_{kl} \Big) \, 
                              (\partial_k \tilde u) \, \overline{\partial_l \tilde v}  \\
& = & \RRe \int_\Omega \sum_{k,l=1}^d c_{kl} \, (\partial_k u) \, \overline{(\partial_l v)}
   + \RRe \sum_{k,l=1}^d \int_{E^-(x,R)} \Big( C^\Phi)_{kl}(x) - (C^\Phi)_{kl} \Big) \, 
                              (\partial_k \tilde u) \, \overline{\partial_l \tilde v}  \\
& = & \RRe (f_0,v)_{L_2(\Omega)} - \RRe \sum_{k=1}^d (f_k, \partial_k v)_{L_2(\Omega)}  \\*
& & \hspace{10mm} {}
   + \RRe \sum_{k,l=1}^d \int_{E^-(x,R)} 
         \Big( C^\Phi)_{kl}(x) - (C^\Phi)_{kl} \Big) \, 
                              (\partial_k \tilde u) \, \overline{\partial_l \tilde v} 
.
\end{eqnarray*}
We estimate the terms separately.
First 
\begin{eqnarray*}
\RRe (f_0,v)_{L_2(\Omega)}
& \leq & d! \, K^d \Big( \int_{E^-(x,R)} 
      |f_0 \circ \Phi^{-1}|^2 \Big)^{1/2} \, \Big( \int_{E^-(x,R)} |\hat v|^2 \Big)^{1/2}  \\
& \leq & d! \, K^d \, \|f_0 \circ \Phi^{-1}\|_{M,\gamma-2,x,E^-,1} \, R^{\frac{\gamma-2}{2}} \, 
         c_D \, R \, \Big( \int_{E^-(x,R)} |\nabla \hat v|^2 \Big)^{1/2}  \\
& = & d! \, K^d \, c_D \, \|f_0 \circ \Phi^{-1}\|_{M,\gamma-2,x,E^-,1} \, R^{\frac{\gamma}{2}}
         \, \Big( \int_{E^-(x,R)} |\nabla \tilde v|^2 \Big)^{1/2}
, 
\end{eqnarray*}
where $c_D$ is the constant in the Dirichlet type Poincar\'e inequality
in the unit half-ball.
Secondly, since $v \in W^{1,2}_0(\Omega)$ and has compact support in $\Ri^d$, it follows that
$\int_{\Phi^{-1}(E^-(x,R))} \partial_k \tilde v
= \int_\Omega \partial_k v = 0$ and therefore
\begin{eqnarray*}
\RRe \sum_{k=1}^d (f_k, \partial_k v)_{L_2(\Omega)}
& = & 
\RRe \sum_{k=1}^d \int_{\Phi^{-1}(E^-(x,R))} 
        (f_k - \langle f_k \circ \Phi^{-1} \rangle_{E^-(x,R)}) \, 
       \overline{\partial_k v}  \\
& \leq & d! \, K^d \sum_{k=1}^d |||f_k \circ \Phi^{-1}|||_{\cm,\gamma,x,E^-,1} \, R^{\frac{\gamma}{2}}
         \, \Big( \int_{E^-(x,R)} |\nabla \tilde v|^2 \Big)^{1/2}
.  
\end{eqnarray*}
Finally, since 
$|(C^\Phi)_{kl}(x) - (C^\Phi)_{kl}(y)| 
\leq |||(C^\Phi)_{kl}|||_{C^\kappa} \, |x-y|^\kappa 
\leq d! \, d^2 \, K^{d+2} \, M \, R^\kappa$
for all $k,l \in \{ 1,\ldots,d \} $ and $y \in E^-(x,R)$, one deduces that 
\begin{eqnarray*}
\lefteqn{
\RRe \sum_{k,l=1}^d \int_{E^-(x,R)} \Big( (C^\Phi)_{kl}(x) - (C^\Phi)_{kl} \Big) \, 
                              (\partial_k \tilde u) \, \overline{\partial_l \tilde v}  
} \hspace*{10mm} \\*
& \leq & d! \, d^2 \, K^{d+2} \, M \, R^\kappa \, \Big( \int_{E^-(x,R)} |\nabla \tilde u|^2 \Big)^{1/2}
    \, \Big( \int_{E^-(x,R)} |\nabla \tilde v|^2 \Big)^{1/2}  \\
& \leq & d! \, d^2 \, K^{d+2} \, M \,  R^{\frac{\gamma}{2}} \, 
      \|\nabla \tilde u\|_{M,d - \delta,x,E^-,1}
    \, \Big( \int_{E^-(x,R)} |\nabla \tilde v|^2 \Big)^{1/2} 
.  
\end{eqnarray*}
Therefore 
\begin{equation}
\Big( \int_{E^-(x,R)} |\nabla \tilde v|^2 \Big)^{1/2}
\leq c_0 \, c_1 \, R^{\frac{\gamma}{2}}
,  
\label{ldtnpc406;16}
\end{equation}
where $c_1 = \mu^{-1} \, d!^2 \, d^2 \, K^{2d+4} (1 + c_D + M)$.

It follows from (\ref{ldtnpc406;15}) and (\ref{ldtnpc406;16}) that 
\[
\Psi(r) 
\leq 4 c_{DG} \, \Big( \frac{r}{R} \Big)^{d+2} \Psi(R)
   + (4 c_{DG} + 2) \, c_0^2 \, c_1^2 \, d \, R^\gamma
\]
for all $0 < r \leq R \leq 1$.
These bounds can be improved by use of Lemma~III.2.1 of \cite{Gia1}.
It follows that there exists an $a > 0$, depending only of $c_{DG}$, $\gamma$
and $d$, such that 
\[
\Psi(r) 
\leq a \, \Big( \frac{r}{R} \Big)^\gamma \Psi(R)
   + a \, (4 c_{DG} + 2) \, c_0^2 \, c_1^2 \, d \, r^\gamma
\]
for all $0 < r \leq R \leq 1$, as required.
\end{proof}

We next turn to interior regularity.

\begin{prop} \label{pdtnpc405}
Let $\kappa \in (0,1)$, $\delta \in [0,\kappa]$ and $\mu,M > 0$.
Then there exists a $c \geq 1$ such that the following is valid.
Let $\Omega \subset \Ri^d$ be an open set.
Let $C \in \ce^\kappa(\Omega,\mu,M)$, $u \in W^{1,2}(\Omega)$ and
$f_0,f_1,\ldots,f_d \in L_2(\Omega)$.
Suppose that 
\[
\int_\Omega \sum_{k,l=1}^d c_{kl} \, (\partial_k u) \, \overline{(\partial_l v)}
= (f_0,v)_{L_2(\Omega)} - \sum_{k=1}^d (f_k, \partial_k v)_{L_2(\Omega)}
\]
for all $v \in W^{1,2}_0(\Omega)$.
Let $r,R,R_e \in (0,1]$ and $x \in \Omega$ and suppose that 
$0 < r \leq R \leq R_e$ and 
$B(x,R_e) \subset \Omega$.
Then 
\[
\Psi_0(r)
\leq c \Big( \frac{r}{R} \Big)^\gamma \Psi_0(R)
   + c \, c_0^2 \, r^\gamma ,
\]
where $\gamma = d + 2 \kappa - \delta$, for all $\rho \in (0,1]$ we define 
\[
\Psi_0(\rho) 
= \sum_{k=1}^d \int_{B(x,\rho)} |\partial_k u - \langle \partial_k u \rangle_{B(x,\rho)}|^2
\]
and where
\[
c_0 = \|f_0\|_{M,\gamma-2,x,\Omega,R_e}
          + \sum_{k=1}^d |||f_k|||_{\cm,\gamma,x,\Omega,R_e}
          + \|\nabla u\|_{M,d - \delta,x,\Omega,R_e}
.  \]
Moreover,  
\[
|||\nabla u|||_{\cm,\gamma,x,\Omega,R_e}
\leq c \, c_0 \, + c \, R_e^{-\kappa} \, \|\nabla u\|_{M,d - \delta,x,\Omega,R_e}
.  \]
\end{prop}
\begin{proof}
Let $c_{DG} > 0$ be as in (\ref{eldtnpc403;2}).
Let $R \in (0,R_e]$.
There exists a unique $v \in W^{1,2}_0(B(x,R))$ such that 
\[
\sum_{k,l=1}^d \int_{B(x,R)} c_{kl}(x) \, (\partial_k v) \, \overline{\partial_l \tau}
= \sum_{k,l=1}^d \int_{B(x,R)} c_{kl}(x) \, (\partial_k u) \, \overline{\partial_l \tau}
\]
for all $\tau \in W^{1,2}_0(B(x,R))$.
Extend $v$ by zero to a function from $\Omega$ into $\Ci$, still denoted by~$v$.
Set $w = u - v$.
Then $w \in W^{1,2}(\Omega)$.
Moreover, $A^{C(x)} w = 0$ weakly on $B(x,R)$.
Let $r \in (0,R]$.
If $k \in \{ 1,\ldots,d \} $, then it follows as in the proof of 
(\ref{eldtnpc406;1}), but now using (\ref{eldtnpc403;2}) instead of 
(\ref{eldtnpc404;3}), that 
\begin{eqnarray*}
\lefteqn{
\int_{B(x,r)} |\partial_k u - \langle \partial_k u \rangle_{B(x,r)}|^2
} \hspace*{10mm}  \nonumber  \\*
& \leq & 4 c_{DG} \, \Big( \frac{r}{R} \Big)^{d+2} 
   \int_{B(x,R)} |\partial_k u - \langle \partial_k u \rangle_{B(x,R)}|^2
   + (4 c_{DG} + 2) \int_{B(x,R)} |\nabla v|^2 
.
\end{eqnarray*}
Then the remaining part of the proof is very similar to the proof of 
Lemma~\ref{ldtnpc406}. 
This time one has to use the Dirichlet type Poincar\'e inequality on the 
full unit ball.
\end{proof}

We combine the last lemma and proposition to obtain estimates close to the 
boundary.

\begin{prop} \label{pdtnpc407}
Let $\kappa \in (0,1)$, $K \geq 1$, $\delta \in [0,\kappa]$ and $\mu,M > 0$.
Then there exists a $c \geq 1$ such that the following is valid.

Let $\Omega,U \subset \Ri^d$ open.
Let $\Phi$ be a $C^{1+\kappa}$-diffeomorphism from $U$ onto $E$ such that 
$\Phi(U \cap \Omega) = E^-$ and $\Phi(U \cap \partial \Omega) = P$.
Suppose that $K$ is larger than the Lipschitz constant for $\Phi$ and $\Phi^{-1}$.
Moreover, suppose that $|||(D \Phi)_{ij}|||_{C^\kappa} \leq K$ and 
$|||(D (\Phi^{-1}))_{ij}|||_{C^\kappa} \leq K$ for all $i,j \in \{ 1,\ldots,d \} $,
where $D \Phi$ denotes the derivative of $\Phi$.
Let $C \in \ce^\kappa(\Omega,\mu,M)$, $u \in W^{1,2}_0(\Omega)$ and 
$f_0,f_1,\ldots,f_d \in L_2(\Omega)$ and suppose that 
\[
\int_\Omega \sum_{k,l=1}^d c_{kl} \, (\partial_k u) \, \overline{(\partial_l v)}
= (f_0,v)_{L_2(\Omega)} - \sum_{k=1}^d (f_k, \partial_k v)_{L_2(\Omega)}
\]
for all $v \in W^{1,2}_0(\Omega)$.
Define $\tilde u \colon E^- \to \Ci$ by $\tilde u = u \circ \Phi^{-1}$.
Then 
\[
|||\nabla \tilde u|||_{\cm,\gamma,x,E^-,1}
\leq c \Big( \|f_0 \circ \Phi^{-1}\|_{M,\gamma-2,x,E^-,1}
          + \sum_{k=1}^d |||f_k \circ \Phi^{-1}|||_{\cm,\gamma,x,E^-,1}
          + \|\nabla \tilde u\|_{M,d - \delta,x,E^-,1}
       \Big)
\]
for all $x \in \frac{1}{2} \, E^-$, where $\gamma = d + 2 \kappa - \delta$.
\end{prop}
\begin{proof}
Let $c \geq 1$ be as in Lemma~\ref{ldtnpc406}.
Set $\tilde u = u \circ \Phi^{-1}$.
For all $x \in \frac{1}{2} \, \overline{E^-}$ and $\rho \in (0,1]$ define 
\begin{eqnarray*}
\Psi(x,\rho) 
& = & \int_{E^-(x,\rho)} |\partial_d \tilde u - \langle \partial_d \tilde u \rangle_{E^-(x,\rho)}|^2
   + \sum_{i=1}^{d-1} \int_{E^-(x,\rho)} |\partial_i \tilde u|^2  \mbox{ and}  \\
\Psi_0(x,\rho) 
& = & \sum_{k=1}^d \int_{E^-(x,\rho)} |\partial_k \tilde u - \langle \partial_k \tilde u \rangle_{E^-(x,\rho)}|^2
.
\end{eqnarray*}
Clearly $\Psi_0(x,\rho) \leq \Psi(x,\rho) \leq \Psi(x,1) \leq \|\nabla \tilde u\|_{L_2(E^-(x,1))}^2
\leq \|\nabla \tilde u\|_{M,d - \delta,x,E^-,1}^2$.

Let $x \in \frac{1}{2} \, E^-$.
Set 
\[
c_0 = \|f_0\|_{M,\gamma-2,x,E^-,1}
          + \sum_{k=1}^d |||f_k|||_{\cm,\gamma,x,E^-,1}
          + \|\nabla \tilde u\|_{M,d - \delta,x,E^-,1}
.  \]
If follows as in the proofs of Proposition~\ref{pdtnpc405} and 
Lemma~\ref{ldtnpc406} that there exists a $\tilde c \geq 1$, 
depending only on $\kappa$, $K$, $\delta$, $\mu$ and $M$, such that 
\begin{equation}
\Psi_0(x,r)
\leq \tilde c \Big( \frac{r}{R} \Big)^\gamma \Psi_0(x,R)
   + \tilde c \, c_0^2 \, r^\gamma ,
\label{epdtnpc407;1}
\end{equation}
for all $r,R \in (0,1]$ with $r \leq R \leq |x_d|$.

Define $y = (x_1,\ldots,x_{d-1},0)$. 
Then $y \in \frac{1}{2} \, P$.
Let $r \in (0,1]$.
We distinguish four cases.

\smallskip

\noindent
{\bf Case 1.} Suppose that $r \leq |x_d| \leq \frac{1}{4}$.  \\
Then (\ref{epdtnpc407;1}), the inclusion $E^-(x,|x_d|) = B(x,|x_d|) \subset E^-(y,2 |x_d|)$,
Lemma~\ref{ldtnpc406} and the inclusion $E^-(y,\frac{1}{2}) \subset E^-(x,1)$ give
\begin{eqnarray*}
\Psi_0(x,r)
& \leq & \tilde c \Big( \frac{r}{|x_d|} \Big)^\gamma \Psi_0(x,|x_d|) 
    + \tilde c \, c_0^2 \, r^\gamma  \\
& \leq & \tilde c \Big( \frac{r}{|x_d|} \Big)^\gamma \Psi(x,|x_d|) 
    + \tilde c \, c_0^2 \, r^\gamma  \\
& \leq & \tilde c \Big( \frac{r}{|x_d|} \Big)^\gamma \Psi(y,2|x_d|) 
    + \tilde c \, c_0^2 \, r^\gamma  \\
& \leq & \tilde c \Big( \frac{r}{|x_d|} \Big)^\gamma 
   \Big( c \, (4|x_d|)^\gamma  \Psi(y,\tfrac{1}{2}) + c \, c_0^2 \, (2|x_d|)^\gamma \Big)
    + \tilde c \, c_0^2 \, r^\gamma  \\
& \leq & 4^\gamma c \, \tilde c \, r^\gamma \, \Psi(x,1)
   + 2^\gamma c \, \tilde c \, c_0^2 \, r^\gamma
   + \tilde c \, c_0^2 \, r^\gamma  \\
& \leq & 4^{\gamma+1} c \, \tilde c \, c_0^2 \, r^\gamma
.
\end{eqnarray*}

\smallskip

\noindent
{\bf Case 2.} Suppose that $|x_d| \leq r \leq \frac{1}{4}$.  \\
Then the inclusion $E^-(x,|x_d|) \subset E^-(y,2r)$
and Lemma~\ref{ldtnpc406} give
\begin{eqnarray*}
\Psi_0(x,r)
& \leq & \Psi(x,r)
\leq \Psi(y,2r)   
\leq c \, (4 r)^\gamma \, \Psi(y,\tfrac{1}{2}) 
    + c \, c_0^2 \, (2r)^\gamma   \\
& \leq & c \, (4 r)^\gamma \, \Psi(x,1) 
    + c \, c_0^2 \, (2r)^\gamma  
\leq 4^{\gamma+1} c^2 \, c_0^2 \, r^\gamma
.
\end{eqnarray*}

\smallskip

\noindent
{\bf Case 3.} Suppose that $r \geq \frac{1}{4}$.  \\
Then $\Psi_0(x,r) \leq \|\nabla \tilde u\|_{L_2(E^-(x,1))}^2 \leq 4^{d+2} c_0^2 \, r^\gamma$.

\smallskip

\noindent
{\bf Case 4.} Suppose that $r \leq \frac{1}{4} \leq |x_d|$.  \\
Then (\ref{epdtnpc407;1}) gives
$\Psi_0(x,r)
\leq c \, (4r)^\gamma \, \Psi_0(x,\frac{1}{4}) + \tilde c \, c_0^2 \, r^\gamma
\leq 4^{\gamma+1} \, c \, \tilde c \, c_0^2 \, r^\gamma$.

\smallskip

The four cases together complete the proof of the proposition.
\end{proof}

Using the De Giorgi estimates (\ref{eldtnpc403;1}) one also has interior
regularity for $A^C$ in the Morrey-region. 
The proposition is a modification of a proposition which appears at many places
in the literature (\cite{Mor}, \cite{GiM} Theorem~5.13, \cite{Aus1} Theorem~3.6, 
\cite{AT2} Lemma~1.12,
\cite{ER15} Proposition~4.2, \cite{DER4} Proposition~A.3.1, \cite{ERe2} Proposition~3.2.)

\begin{prop} \label{pdtnpc408}
Let $\kappa \in (0,1)$, $\mu,M > 0$,
$\gamma \in [0,d)$ and $\delta \in (0,2]$ with 
$\gamma + \delta < d$.
Then there exists an $c > 0$, such that 
the following is valid.
Let $\Omega \subset \Ri^d$ be an open set.
Let $C \in \ce^\kappa(\Omega,\mu,M)$, $u \in W^{1,2}(\Omega)$ and
$f_0,f_1,\ldots,f_d \in L_2(\Omega)$.
Suppose that 
\[
\int_\Omega \sum_{k,l=1}^d c_{kl} \, (\partial_k u) \, \overline{(\partial_l v)}
= (f_0,v)_{L_2(\Omega)} - \sum_{k=1}^d (f_k, \partial_k v)_{L_2(\Omega)}
\]
for all $v \in W^{1,2}_0(\Omega)$.
Let $x \in \Omega$, $R_e \in (0,1]$ and suppose that 
$B(x,R_e) \subset \Omega$.
Then
\[
\|\nabla u\|_{M,\gamma+\delta,x, \Omega,R_e}
\leq c \, \Big(
    \varepsilon^{2-\delta} \|f_0\|_{M,\gamma,x, \Omega,R_e}
  + \sum_{k=1}^d \|f_k\|_{M,\gamma+\delta,x, \Omega,R_e}
  + \varepsilon^{-(\gamma + \delta)} \|\nabla u\|_{L_2(\Omega)} \Big)
\]
for all $\varepsilon \in (0,1]$.
\end{prop}

Similarly, using the De Giorgi estimates (\ref{eldtnpc404;1}) and (\ref{eldtnpc404;2})
one also has boundary regularity in the Morrey region.

\begin{prop} \label{pdtnpc409}
Let $\kappa \in (0,1)$, $K \geq 1$, $\mu,M > 0$,
$\gamma \in [0,d)$ and $\delta \in (0,2]$ with 
$\gamma + \delta < d$.
Then there exists an $c > 0$, such that 
the following is valid.

Let $\Omega,U \subset \Ri^d$ open.
Let $\Phi$ be a $C^{1+\kappa}$-diffeomorphism from $U$ onto $E$ such that 
$\Phi(U \cap \Omega) = E^-$ and $\Phi(U \cap \partial \Omega) = P$.
Suppose that $K$ is larger than the Lipschitz constant for $\Phi$ and $\Phi^{-1}$.
Moreover, suppose that $|||(D \Phi)_{ij}|||_{C^\kappa} \leq K$ and 
$|||(D (\Phi^{-1}))_{ij}|||_{C^\kappa} \leq K$ for all $i,j \in \{ 1,\ldots,d \} $,
where $D \Phi$ denotes the derivative of $\Phi$.
Let $C \in \ce^\kappa(\Omega,\mu,M)$, $u \in W^{1,2}_0(\Omega)$ and 
$f_0,f_1,\ldots,f_d \in L_2(\Omega)$ and suppose that 
\[
\int_\Omega \sum_{k,l=1}^d c_{kl} \, (\partial_k u) \, \overline{(\partial_l v)}
= (f_0,v)_{L_2(\Omega)} - \sum_{k=1}^d (f_k, \partial_k v)_{L_2(\Omega)}
\]
for all $v \in W^{1,2}_0(\Omega)$.
Define $\tilde u \colon E^- \to \Ci$ by $\tilde u = u \circ \Phi^{-1}$.
Then
\[
\hspace*{-1pt}
\|\nabla \tilde u\|_{M,\gamma+\delta,x, E^-,1}
\leq c \, \Big(
    \varepsilon^{2-\delta} \|f_0 \circ \Phi^{-1}\|_{M,\gamma,x, E^-,1}
  + \sum_{k=1}^d \|f_k \circ \Phi^{-1}\|_{M,\gamma+\delta,x, E^-,1}
  + \varepsilon^{-(\gamma + \delta)} \|\nabla \tilde u\|_{L_2(\Omega)} \Big)
\]
for all $x \in \frac{1}{2} \, E^-$ and $\varepsilon \in (0,1]$.
\end{prop}

Let $\Omega \subset \Ri^d$ be an open set, 
let $C \in \ce(\Omega)$ and $V \in L_\infty(\Omega)$.
Let $T$ be the semigroup generated by $-(A_D + V)$.
We omit the dependence of $T$ on $C$ and $V$ in our notation, since that will be 
clear from the context.
We also need the Davies perturbation.
Let 
\[
\cd = \{ \psi \in C^\infty_{\rm c}(\Ri^d,\Ri) : \|\nabla \psi\|_\infty \leq 1 \} 
 .  \]
For all $\rho \in \Ri$ and $\psi \in \cd$ define the 
multiplication operator $U_\rho$ by $U_\rho u = e^{-\rho \, \psi} u$.
Note that $U_\rho u \in W^{1,2}_0(\Omega)$ for all 
$u \in W^{1,2}_0(\Omega)$.
Let $T^\rho_t = U_\rho \, T_t \, U_{-\rho}$ be the Davies perturbation
for all $t > 0$.
Let $- A^{(\rho)}$ be the generator of $(T^\rho_t)_{t > 0}$. 
Then $A^{(\rho)}$ is the operator associated with the form
$\gotl^{(\rho)}$ with form domain $D(\gotl^{(\rho)}) = W^{1,2}_0(\Omega)$ and  
\begin{equation}
\gotl^{(\rho)}(u,v)
= \gota(u,v)
   + \int_\Omega \sum_{i=1}^d \Big( a^{(\rho)}_i \, (\partial_i u) \, \overline v
                             + b^{(\rho)}_i \, u \, \overline{(\partial_i v)} \Big)
   + \int_\Omega a^{(\rho)}_0 \, u \, \overline v
\label{epdir505;5}
\end{equation}
with 
\[
a^{(\rho)}_k =  - \rho \sum_{l=1}^d  c_{kl} \, \partial_l \psi
\quad , \quad
b^{(\rho)}_k =  \rho \sum_{l=1}^d a_{lk} \, \partial_l \psi
\]
and 
\[
a^{(\rho)}_0
= V - \rho^2 \sum_{k,l=1}^d c_{kl} \, (\partial_k \psi) \, \partial_l \psi
.  \]

We start with $L_2$-estimates for the perturbed semigroup.

\begin{lemma} \label{ldtnpc413}
Let $\Omega \subset \Ri^d$ be a bounded open set.
For all $\mu,M > 0$ there exist $c_0,\omega_0,\omega_1 > 0$ 
such that 
\begin{equation}
\|T^\rho_t u\|_{L_2(\Omega)} 
\leq e^{-\omega_1 t} \, e^{\|(\RRe V)^-\|_\infty t} \, e^{\omega_0 \rho^2 t} \, \|u\|_{L_2(\Omega)}
\;\;\; , \;\;\;
\|\nabla T^\rho_t u\|_{L_2(\Omega)} 
     \leq c_0 \, t^{-1/2} \, e^{\omega_0 (1+\rho^2) t} \,
 \|u\|_{L_2(\Omega)}
\label{eldtnpc413;1}
\end{equation}
and 
\[
\|A^{(\rho)} \, T^\rho_t u\|_{L_2(\Omega)} 
\leq c_0 \, t^{-1} \, e^{\omega_0 (1+\rho^2) t} \, \|u\|_{L_2(\Omega)}
\]
for all $\kappa \in (0,1)$,
$C \in \ce^\kappa(\Omega,\mu,M)$,
$V \in L_\infty(\Omega)$,
$u \in L_2(\Omega)$, $t > 0$, $\rho \in \Ri$ and $\psi \in \cd$.
\end{lemma}
\begin{proof} 
By the Dirichlet type Poincar\'e inequality there exists a $\lambda > 0$ 
such that $\lambda \int_\Omega |u|^2 \leq \int_\Omega |\nabla u|^2$ for all 
$u \in W^{1,2}_0(\Omega)$.
Without loss of generality we may assume that $\mu \leq 1 \leq M$.
Let $u \in L_2(\Omega)$.
It follows from (\ref{epdir505;5}) that 
\begin{eqnarray*}
\mu \, \|\nabla T^\rho_t u\|_{L_2(\Omega)}^2
& \leq & \RRe \gota(T^\rho_t u)  \\
& \leq & \RRe \gota(T^\rho_t u) + ((\RRe V)^+ T^\rho_t u, T^\rho_t u)_{L_2(\Omega)}  \\
& \leq & \RRe \gotl^{(\rho)}(T^\rho_t u) 
   + 2 d \, M \, |\rho| \, \|\nabla T^\rho_t u\|_{L_2(\Omega)} \, 
                                          \|T^\rho_t u\|_{L_2(\Omega)} \\*
& & \hspace*{40mm} {}  
   + \|(\RRe V)^-\|_\infty \, \|T^\rho_t u\|_{L_2(\Omega)}^2
   + d \, M \, \rho^2 \, \|T^\rho_t u\|_{L_2(\Omega)}^2  \\
& \leq & \RRe \gotl^{(\rho)}(T^\rho_t u) 
   + \tfrac{1}{2} \, \mu \|\nabla T^\rho_t u\|_{L_2(\Omega)}^2
   + \frac{2 d^2 \, M^2 \, \rho^2}{\mu} \, \|T^\rho_t u\|_{L_2(\Omega)}^2  \\*
& & \hspace*{40mm} {}
   + \|(\RRe V)^-\|_\infty \, \|T^\rho_t u\|_{L_2(\Omega)}^2
   + d \, M  \,\rho^2 \, \|T^\rho_t u\|_{L_2(\Omega)}^2 
\end{eqnarray*}
for all $t > 0$.
So 
\begin{eqnarray*}
\tfrac{1}{2} \, \lambda \, \mu \, \|T^\rho_t u\|_{L_2(\Omega)}^2
& \leq & \tfrac{1}{2} \, \mu \, \|\nabla T^\rho_t u\|_{L_2(\Omega)}^2  \nonumber  \\
& \leq & \RRe \gotl^{(\rho)}(T^\rho_t u) 
   + (\|(\RRe V)^-\|_\infty + \omega_1 \, \rho^2) \, \|T^\rho_t u\|_{L_2(\Omega)}^2 
,
\end{eqnarray*}
where $\omega_1 = 3 d^2 \, M^2 \, \mu^{-1}$.
Hence 
\begin{eqnarray*}
\frac{d}{dt} \|T^\rho_t u\|_{L_2(\Omega)}^2
& = & - 2 \RRe (A^{(\rho)} T^\rho_t u, T^\rho_t u)_{L_2(\Omega)}  \\
& = & - 2 \RRe \gotl^{(\rho)}(T^\rho_t u) 
\leq 2 (- \tfrac{1}{2} \, \lambda \, \mu + \|(\RRe V)^-\|_\infty 
        + \omega_1 \, \rho^2) \, \|T^\rho_t u\|_{L_2(\Omega)}^2 
\end{eqnarray*}
for all $t > 0$.
This implies that 
\[
\|T^\rho_t u\|_{L_2(\Omega)}
\leq e^{- \tfrac{1}{2} \, \lambda \mu t} 
   \, e^{\|(\RRe V)^-\|_\infty t} \, e^{\omega_1 \rho^2 t} \, \|u\|_{L_2(\Omega)}
\]
for all $t > 0$.

The other estimates of the lemma follow as in the proof of Lemma~2.1 in \cite{EO2}.
\end{proof}

By a Neumann type Poincar\'e inequality there is a relation
between the Campanato norm and the Morrey norm of the gradient of a function.

\begin{lemma} \label{ldtnpc417}
There exists a $c_N > 0$ such that 
\[
|||u|||_{\cm,\gamma+2,x,E^-,1}
\leq c_N \, \|\nabla u\|_{M,\gamma,x,E^-,1}
\]
and 
\[
|||v|||_{\cm,\gamma+2,y,\Omega,R_e}
\leq c_N \, \|\nabla v\|_{M,\gamma,y,\Omega,R_e}
\]
for all $\gamma \in [0,d)$, $u \in W^{1,2}(E^-)$, $x \in \frac{1}{2} \, E^-$,
open $\Omega \subset \Ri^d$, $v \in W^{1,2}(\Omega)$, $y \in \Omega$ and 
$R_e \in (0,1]$ with $B(y,R_e) \subset \Omega$.
\end{lemma}

Next we consider $L_2$--$W^{1+\kappa,\infty}$ estimates 
for the perturbed semigroup.
We start with bounds close to the boundary.

\begin{prop} \label{pdtnpc414}
Let $\kappa \in (0,1)$, $K \geq 1$ and $\mu,M > 0$.
Then there exist  $c,\omega > 0$ such that 
such that the following is valid.

Let $\Omega,U \subset \Ri^d$ open.
Let $\Phi$ be a $C^{1+\kappa}$-diffeomorphism from $U$ onto $E$ such that 
$\Phi(U \cap \Omega) = E^-$ and $\Phi(U \cap \partial \Omega) = P$.
Suppose that $K$ is larger than the Lipschitz constant for $\Phi$ and $\Phi^{-1}$.
Moreover, suppose that $|||(D \Phi)_{ij}|||_{C^\kappa} \leq K$ and 
$|||(D (\Phi^{-1}))_{ij}|||_{C^\kappa} \leq K$ for all $i,j \in \{ 1,\ldots,d \} $,
where $D \Phi$ denotes the derivative of $\Phi$.
Let $C \in \ce^\kappa(\Omega,\mu,M)$,
$V \in L_\infty(\Omega)$,
$t > 0$, $u \in L_2(\Omega)$, $\rho \in \Ri$ and $\psi \in \cd$, 
with $\|V\|_\infty \leq M$.
Then $\nabla ((T^\rho_t u) \circ \Phi^{-1})$ is continuous on 
$\frac{1}{2} \, E^-$.
Moreover,
\begin{eqnarray}
\|T^\rho_t u\|_{L_\infty(\Phi^{-1}(\tfrac{1}{2} \, E^-))}
& \leq & c \, t^{-d/4} \, e^{\omega (1 + \rho^2) t} \, \|u\|_{L_2(\Omega)} ,
   \nonumber  \\
\|\nabla T^\rho_t u\|_{L_\infty(\Phi^{-1}(\tfrac{1}{2} \, E^-))}
& \leq & c \, t^{-d/4} \, t^{-1/2} \, e^{\omega (1 + \rho^2) t} \, \|u\|_{L_2(\Omega)} 
   \label{epdtnpc414;2}  \mbox{, and } \\
|(\nabla T^\rho_t u)(x) - (\nabla T^\rho_t u)(y)|
& \leq & c \, t^{-d/4} \, t^{-1/2} \, t^{-\kappa/2} \, e^{\omega (1 + \rho^2) t} \, \|u\|_{L_2(\Omega)} 
    \, |x-y|^\kappa
  \label{epdtnpc414;3} 
\end{eqnarray}
for all $x,y \in \Phi^{-1}(\tfrac{1}{4} \, E^-)$ with $|x-y| \leq \frac{1}{4K}$.
\end{prop}
\begin{proof} 
For all $\gamma \in [0,d-2)$ let $P(\gamma)$ be the hypothesis 
\begin{list}{}{\leftmargin=1.8cm \rightmargin=1.8cm \listparindent=12pt}
\item
There exist $c,\omega > 0$, depending only on $\Omega$, $\kappa$, $\mu$ and $M$, such that 
\[
\|(T^\rho_t u) \circ \Phi^{-1}\|_{M,\gamma,x, E^-,1}
\leq c \, t^{-\gamma / 4} \, e^{\omega (1 + \rho^2) t} \, \|u\|_{L_2(\Omega)}
\]
and 
\begin{equation}
\|\nabla ((T^\rho_t u) \circ \Phi^{-1})\|_{M,\gamma,x, E^-,1}
\leq c \, t^{-\gamma / 4} \, t^{-1/2} \, e^{\omega (1 + \rho^2) t} \, \|u\|_{L_2(\Omega)}
\label{epdtnpc414;5}
\end{equation}
for all $t > 0$, $u \in L_2(\Omega)$, $\rho \in \Ri$, $\psi \in \cd$
and $x \in \frac{1}{2} \, E^-$.
\end{list}

Clearly $P(0)$ is valid by Lemma~\ref{ldtnpc413}.
Arguing as in the proof of Proposition~4.3 in \cite{ER15},
Lemma~3.3 in \cite{EO1} or Lemma~7.1 in \cite{ERe2},
it follows from Lemma~\ref{ldtnpc413} and Proposition~\ref{pdtnpc409}
that $P(\gamma)$ is valid for all $\gamma \in [0,d)$.

For all $\gamma \in [0,d+2\kappa]$ let $P'(\gamma)$ be the hypothesis 
\begin{list}{}{\leftmargin=1.8cm \rightmargin=1.8cm \listparindent=12pt}
\item
There exist $c,\omega > 0$, depending only on $\Omega$, $\kappa$, $\mu$ and $M$, such that 
\begin{equation}
|||(T^\rho_t u) \circ \Phi^{-1}|||_{\cm,\gamma,x, E^-,1}
\leq c \, t^{-\gamma / 4} \, e^{\omega (1 + \rho^2) t} \, \|u\|_{L_2(\Omega)}
\label{epdtnpc414;6}
\end{equation}
and 
\[
|||\nabla ((T^\rho_t u) \circ \Phi^{-1})|||_{\cm,\gamma,x, E^-,1}
\leq c \, t^{-\gamma / 4} \, t^{-1/2} \, e^{\omega (1 + \rho^2) t} \, \|u\|_{L_2(\Omega)}
\]
for all $t > 0$, $u \in L_2(\Omega)$, $\rho \in \Ri$, $\psi \in \cd$
and $x \in \frac{1}{2} \, E^-$.
\end{list}

If $\gamma \in [0,d)$, then $P(\gamma)$ and Lemma~\ref{ldtnpc402}\ref{ldtnpc402-1}
imply that $P'(\gamma)$ is valid.
Then the Poincar\'e inequality of Lemma~\ref{ldtnpc417} and (\ref{epdtnpc414;5}) 
give that (\ref{epdtnpc414;6}) is valid for all $\gamma \in [0,d+2 \kappa]$
(even for all $\gamma \in [0,d+2)$).
Arguing similarly, using the regularity estimates of Proposition~\ref{pdtnpc407},
it follows that for all $\delta \in [0,\kappa]$ there exist $c,\omega > 0$, 
depending only on $\Omega$, $\kappa$, $\delta$, $\mu$ and $M$, such that 
\begin{equation}
|||\nabla ((T^\rho_t u) \circ \Phi^{-1})|||_{\cm,\gamma,x, E^-,1}
\leq c \, t^{-\gamma / 4} \, t^{-1/2} \, e^{\omega (1 + \rho^2) t} \, \|u\|_{L_2(\Omega)}
   + c \, \|\nabla ((T^\rho_t u) \circ \Phi^{-1})\|_{M,d - \delta,x,E^-,1}
\label{epdtnpc414;8}
\end{equation}
for all $t > 0$, $u \in L_2(\Omega)$, $\rho \in \Ri$, $\psi \in \cd$
and $x \in \frac{1}{2} \, E^-$, where
$\gamma = d + 2 \kappa - \delta$.
Choose $\delta = \kappa$.
Then (\ref{epdtnpc414;5}) gives
\begin{eqnarray*}
|||\nabla ((T^\rho_t u) \circ \Phi^{-1})|||_{\cm,d + \kappa,x, E^-,1}
& \leq & c \, t^{-(d+\kappa) / 4} \, t^{-1/2} \, e^{\omega (1 + \rho^2) t} \, \|u\|_{L_2(\Omega)}
   \\*
& & \hspace*{30mm} {}
+ c \, \|\nabla ((T^\rho_t u) \circ \Phi^{-1})\|_{M,d - \kappa,x,E^-,1}  \\
& \leq & c' \, t^{-(d+\kappa) / 4} \, t^{-1/2} \, e^{\omega' (1 + \rho^2) t} \, \|u\|_{L_2(\Omega)}
\end{eqnarray*}
for all $t > 0$, $u \in L_2(\Omega)$, $\rho \in \Ri$, $\psi \in \cd$
and $x \in \frac{1}{2} \, E^-$, 
for suitable $c',\omega' > 0$.
So $\lim_{R \downarrow 0} \langle \nabla ((T^\rho_t u) \circ \Phi^{-1}) \rangle_{E^-(x,R)}$
exists for all $x \in \frac{1}{2} \, E^-$ by Lemma~\ref{ldtnpc402}\ref{ldtnpc402-2}.
Therefore the function $\nabla ((T^\rho_t u) \circ \Phi^{-1}$ is continuous on 
$\frac{1}{2} \, E^-$.
Choose $R = t^{1/2} \, e^{-t}$.
Then $R \leq 1$ and Lemma~\ref{ldtnpc402}\ref{ldtnpc402-2}
gives that there exists a $c'' > 0$, depending only 
on $d$ and $\kappa$, such that 
\begin{eqnarray*}
|(\nabla ((T^\rho_t u) (\Phi^{-1}(x))|
& \leq & c'' \, R^{\kappa / 2} \, |||\nabla ((T^\rho_t u) \circ \Phi^{-1})|||_{\cm,d + \kappa,x, E^-,1}
+ \langle \nabla ((T^\rho_t u) \circ \Phi^{-1}) \rangle_{E^-(x,R)}  \\
& \leq & c'' \, t^{-d / 4} \, t^{-1/2} \, e^{\omega' (1 + \rho^2) t} \, \|u\|_{L_2(\Omega)}
+ \omega_d^{-1/2} \, R^{-d/2} \, \|\nabla ((T^\rho_t u) \circ \Phi^{-1})\|_{L_2(E^-)}
\end{eqnarray*}
for all $x \in \frac{1}{2} \, E^-$.
Then (\ref{epdtnpc414;2}) follows from (\ref{eldtnpc413;1}).

Finally use (\ref{epdtnpc414;8}) with $\delta = 0$ and  $x \in \frac{1}{4} \, E^-$,
the quarter lower half of~$E$.
It follows that there are suitable $c''',\omega'' > 0$ such that 
\[
|||\nabla ((T^\rho_t u) \circ \Phi^{-1})|||_{\cm,d + 2\kappa,x, E^-,1}
\leq c \, t^{-(d + 2\kappa) / 4} \, t^{-1/2} \, e^{\omega (1 + \rho^2) t} \, \|u\|_{L_2(\Omega)}
\]
for all $t > 0$, $u \in L_2(\Omega)$, $\rho \in \Ri$, $\psi \in \cd$
and $x \in \frac{1}{4} \, E^-$.
Then (\ref{epdtnpc414;3}) follows from 
Lemma~\ref{ldtnpc402}\ref{ldtnpc402-3}.
\end{proof}

Similar estimates are valid far away from the boundary.

\begin{prop} \label{pdtnpc415}
Let $\kappa \in (0,1)$, $R_e \in (0,1]$ and $\mu,M > 0$.
Then there exist  $c,\omega > 0$ such that 
such that the following is valid.

Let $\Omega \subset \Ri^d$ be open, $x_0 \in \Omega$ and suppose that 
$B(x_0,R_e) \subset \Omega$.
Let $C \in \ce^\kappa(\Omega,\mu,M)$ and $V \in L_\infty(\Omega)$
with $\|V\|_\infty \leq M$.
Then
\begin{eqnarray*}
\|T^\rho_t u\|_{L_\infty(B(x_0, \tfrac{1}{2} \, R_e))}
& \leq & c \, t^{-d/4} \, e^{\omega (1 + \rho^2) t} \, \|u\|_{L_2(\Omega)} , \\
\|\nabla T^\rho_t u\|_{L_\infty(B(x_0, \tfrac{1}{2} \, R_e))}
& \leq & c \, t^{-d/4} \, t^{-1/2} \, e^{\omega (1 + \rho^2) t} \, \|u\|_{L_2(\Omega)} 
     \mbox{, and } \\
|(\nabla T^\rho_t u)(x) - (\nabla T^\rho_t u)(y)|
& \leq & c \, t^{-d/4} \, t^{-1/2} \, t^{-\kappa/2} \, e^{\omega (1 + \rho^2) t} \, \|u\|_{L_2(\Omega)} 
    \, |x-y|^\kappa
\end{eqnarray*}
for all $t > 0$, $u \in L_2(\Omega)$, $\rho \in \Ri$, $\psi \in \cd$
and $x,y \in B(x_0, \tfrac{1}{4} \, R_e)$ with $|x-y| \leq \frac{1}{4} \, R_e$.
\end{prop}
\begin{proof}
This follows similarly to the proof of Proposition~\ref{pdtnpc414},
using Propositions~\ref{pdtnpc405} and \ref{pdtnpc408}
instead of Propositions~\ref{pdtnpc407} and \ref{pdtnpc409}.
We leave the proof to the reader.
\end{proof}

\begin{prop} \label{pdtnpc416}
Let $\kappa \in (0,1)$, and $\mu,M > 0$.
Let $\Omega \subset \Ri^d$ be a bounded open set with $C^{1+\kappa}$-boundary.
Then there exist  $c,\omega > 0$ such that 
such that the following is valid.

Let $C \in \ce^\kappa(\Omega,\mu,M)$ and $V \in L_\infty(\Omega)$
with $\|V\|_\infty \leq M$.
Then
\begin{eqnarray*}
\|T^\rho_t u\|_{L_\infty(\Omega)}
& \leq & c \, t^{-d/4} \, e^{\omega (1 + \rho^2) t} \, \|u\|_{L_2(\Omega)} , \\
\|\nabla T^\rho_t u\|_{L_\infty(\Omega)}
& \leq & c \, t^{-d/4} \, t^{-1/2} \, e^{\omega (1 + \rho^2) t} \, \|u\|_{L_2(\Omega)} 
     \mbox{, and } \\
|||\nabla T^\rho_t u|||_{C^\kappa}
& \leq & c \, t^{-d/4} \, t^{-1/2} \, t^{-\kappa/2} \, e^{\omega (1 + \rho^2) t} \, \|u\|_{L_2(\Omega)} 
\end{eqnarray*}
for all $t > 0$, $u \in L_2(\Omega)$, $\rho \in \Ri$ and $\psi \in \cd$
\end{prop}
\begin{proof}
This follows from a compactness argument from Propositions~\ref{pdtnpc414}
and \ref{pdtnpc415}.
\end{proof}

We can now prove the Gaussian H\"older kernel bounds of Theorem~\ref{tdtnpc401}.

\begin{proof}[{\bf Proof of Theorem~\ref{tdtnpc401}.}]
Let $c_0,\omega_0,\omega_1 > 0$ be as in Lemma~\ref{ldtnpc413}.
Then 
\[
\|T^\rho_t u\|_{L_2(\Omega)} 
\leq e^{-\omega_1 t} \, e^{\|(\RRe V)^-\|_\infty t} \, 
   e^{\omega_0 \rho^2 t} \, \|u\|_{L_2(\Omega)}
\]
for all $u \in L_2(\Omega)$, $t > 0$, $\rho \in \Ri$ and $\psi \in \cd$.
Then the semigroup $(e^{+ (\omega_1 - \|(\RRe V)^-\|_\infty) t} \, T_t)_{t > 0}$ satisfies the 
bounds of Proposition~\ref{pdtnpc416}.
Therefore the Gaussian H\"older kernel bounds of Theorem~\ref{tdtnpc401}
follows as in the proof of Lemma~A.1 in \cite{EO2}.
\end{proof}

Since all our estimates are locally uniform, we also obtain the following 
theorem which is valid for unbounded domains.

\begin{thm} \label{tdtnpc420}
Let $\kappa,\tau' \in (0,1)$ and $\mu,M,\tau,K > 0$.
Then there exist $a,b > 0$ and $\omega \in \Ri$ such that 
the following is valid.

Let $\Omega \subset \Ri^d$ be an open set.
Suppose for all $x \in \partial \Omega$ there exists an open 
neighbourhood $U$ of $x$ and a 
$C^{1+\kappa}$-diffeomorphism $\Phi$ from $U$ onto $E$ such that 
\begin{itemize}
\item
$\Phi(x) = 0$, 
\item
$\Phi(U \cap \Omega) = E^-$
\item
$\Phi(U \cap \partial \Omega) = P$,
\item
$K$ is larger than the Lipschitz constant for $\Phi$ and $\Phi^{-1}$, and 
\item
$|||(D \Phi)_{ij}|||_{C^\kappa} \leq K$ and 
$|||(D (\Phi^{-1}))_{ij}|||_{C^\kappa} \leq K$ for all $i,j \in \{ 1,\ldots,d \} $
where $D \Phi$ denotes the derivative of $\Phi$.
\end{itemize}
Let $C \in \ce^\kappa(\Omega,\mu,M)$
and $V \in L_\infty(\Omega)$ with $\|V\|_\infty \leq M$.
Then there exists 
a function $(t,x,y) \mapsto H_t(x,y)$ from 
$(0,\infty) \times \Omega \times \Omega$ into $\Ci$
such that the following is valid.
\begin{tabel}
\item \label{tdtnpc420-1}
The function $(t,x,y) \mapsto H_t(x,y)$ is continuous 
from $(0,\infty) \times \Omega \times \Omega$ into $\Ci$.
\item \label{tdtnpc420-2}
For all $t \in (0,\infty)$ the function $H_t$ is the kernel of 
the operator $e^{-t (A_D + V)}$.
\item \label{tdtnpc420-3}
For all $t \in (0,\infty)$ the
function $H_t$ is once differentiable in each variable and the 
derivative with respect to one variable is differentiable in the 
other variable.
Moreover, for every multi-index $\alpha,\beta$ with 
$0 \leq |\alpha|,|\beta| \leq 1$ one has 
\[
|(\partial_x^\alpha \, \partial_y^\beta \, H_t)(x,y)|
\leq a \, t^{-d/2} \, t^{-(|\alpha| + |\beta|)/2} \, 
      e^{-b \, \frac{|x-y|^2}{t}} \, e^{\omega t}
\]
and 
\begin{eqnarray*}
\lefteqn{
|(\partial_x^\alpha \, \partial_y^\beta \, H_t)(x+h,y+k) 
     - (\partial_x^\alpha \, \partial_y^\beta \, H_t)(x,y)|
} \hspace{20mm}  \\*
& \leq & a \, t^{-d/2} \, t^{-(|\alpha| + |\beta|)/2} \, 
    \left( \frac{|h| + |k|}{\sqrt{t} + |x-y|} \right)^\kappa \, 
      e^{-b \, \frac{|x-y|^2}{t}}  \, e^{\omega t}
\end{eqnarray*}
for all $x,y \in \Omega$ and $h,k \in \Ri^d$ with $x+h,y+k \in \Omega$ and 
$|h| + |k| \leq \tau \, \sqrt{t} + \tau' \, |x-y|$.
\end{tabel}
\end{thm}

By a small additional argument one can also add first-order terms to the 
operator with $C^\kappa$-coefficients. 
We do not need first-order terms in this paper.

\section{Green function bounds  and regularity properties}\label{S3.5} \label{Sdtnpc5}

This section is devoted to estimates and regularity properties
of the resolvent operators $(A_D + V)^{-1}$.
We prove estimates for the Green function and its derivatives.
We emphasise that in the first theorem
the constants are uniform with respect to the complex coefficients~$C$
if $\RRe V$ is positive.

\begin{thm} \label{tdtnpc501}
Let $\kappa \in (0,1)$ and $\mu,M > 0$.
Let $\Omega \subset \Ri^d$ be an open bounded set with a $C^{1+\kappa}$-boundary.
Then there exists a $c > 0$ such that for all 
$C \in \ce^\kappa(\Omega,\mu,M)$ and
$V \in L_\infty(\Omega)$ with $\RRe V \geq 0$ and $\|V\|_\infty \leq M$
the operator $(A_D + V)^{-1}$ has 
a kernel $G_V \colon \{ (x,y) \in \Omega \times \Omega : x \neq y \} \to \Ci$, 
which is differentiable in each variable and the derivative is 
differentiable in the other variable.
Moreover, for every multi-index $\alpha,\beta$ with 
$0 \leq |\alpha|,|\beta| \leq 1$ the function
$\partial_x^\alpha \, \partial_y^\beta \, G_V$ 
extends to a locally $\kappa$-H\"older continuous function on 
$ \{ (x,y) \in \overline \Omega \times \overline \Omega : x \neq y \} $
with estimates
\begin{equation}
|(\partial_x^\alpha \, \partial_y^\beta \, G_V)(x,y)| 
\leq \left\{ \begin{array}{ll}
c \, |x-y|^{-(d-2+|\alpha| + |\beta|)}  & \mbox{if } d-2+|\alpha| + |\beta| \neq 0 \\[5pt]
c \, \log( 1 + \frac{1}{|x-y|} ) & \mbox{if } d-2+|\alpha| + |\beta| = 0
                 \end{array} \right.
\label{etdtnpc501;13}  
\end{equation}
and 
\begin{equation}
|(\partial_x^\alpha \, \partial_y^\beta \, G_V)(x',y') - (\partial_x^\alpha \, \partial_y^\beta \, G_V)(x,y)| 
\leq c \, \frac{(|x'-x| + |y'-y|)^\kappa}{|x-y|^{d-2+|\alpha| + |\beta| +\kappa}}   \label{etdtnpc501;14} 
\end{equation}
for all $x,x',y,y' \in \Omega$ with $x \neq y$ and
$|x-x'| + |y-y'| \leq \frac{1}{2} \, |x-y|$.
\end{thm}
\begin{proof} 
By Theorem~\ref{tdtnpc401} there are $a,b,\omega > 0$ such that 
the operator $e^{-t (A_D + V)}$ has a kernel $H_t$ for all $t > 0$,
which is once differentiable in each entry,
satisfying the bounds
\begin{equation}
\begin{array}{r@{}c@{}l}
|(\partial_x^\alpha \, \partial_y^\beta \, H_t)(x,y)|
   & {} \leq {} & a \, t^{-d/2} \, t^{-(|\alpha| + |\beta|)/2} \, 
      e^{-b \, \frac{|x-y|^2}{t}} \, e^{-\omega t} , \\[5pt]
|(\partial_x^\alpha \, \partial_y^\beta \, H_t)(x',y') 
     - (\partial_x^\alpha \, \partial_y^\beta \, H_t)(x,y)|
   & {} \leq {} & a \, t^{-d/2} \, t^{-(|\alpha| + |\beta|)/2} \, 
    \left( \frac{(|x'-x| + |y'-y|)}{\sqrt{t}} \right)^\kappa \, 
      e^{-b \, \frac{|x-y|^2}{t}}  \, e^{-\omega t}
\end{array}
\label{etdtnpc501;1}
\end{equation}
for all $x,x',y,y' \in \Omega$ and multi-index $\alpha,\beta$ with 
$0 \leq |\alpha|,|\beta| \leq 1$ and
$|x-x'| + |y-y'| \leq \frac{1}{2} \, |x-y|$.
Define 
\[
G_V(x,y)
= \int_0^\infty H_t(x,y) \, dt
\]
for all $x,y \in \Omega$ with $x \neq y$.
Then $G_V$ is the kernel of the operator
\[
 (A_D + V)^{-1} = \int_0^\infty T_t \, dt
.  \]
The estimates (\ref{etdtnpc501;1}) give
\[
|(\partial_x^\alpha \, \partial_y^\beta \, G_V)(x,y)|
\leq \int_0^\infty a \, t^{-d/2} \, t^{-(|\alpha| + |\beta|)/2} \, e^{-b \, \frac{|x-y|^2}{t}} \, dt
= a \, c_1 \, |x-y|^{-(d-2+|\alpha| + |\beta|)}
\]
and 
\begin{eqnarray*}
\lefteqn{
|(\partial_x^\alpha \, \partial_y^\beta \, G_V)(x',y') 
   - (\partial_x^\alpha \, \partial_y^\beta \, G_V)(x,y)|
} \hspace*{40mm}  \\*
& \leq & \int_0^\infty a \, t^{-d/2} \, t^{-(|\alpha| + |\beta|)/2} \, 
    \left( \frac{(|x'-x| + |y'-y|)}{\sqrt{t}} \right)^\kappa \, 
      e^{-b \, \frac{|x-y|^2}{t}} \, dt \\
& = & a \, c_2 \, \frac{(|x'-x| + |y'-y|)^\kappa}{|x-y|^{d-2+|\alpha| + |\beta| +\kappa}}
,  
\end{eqnarray*}
where $c_1 = \int_0^\infty a \, t^{-d/2} \, t^{-(|\alpha| + |\beta|)/2} \, e^{-b /t} \, dt < \infty$,
under the condition that $d-2+|\alpha| + |\beta| \neq 0$,
and 
$c_2 = \int_0^\infty a \, t^{-d/2} \, t^{-(|\alpha| + |\beta|)/2} \, 
         t^{-\kappa / 2} \, e^{-b /t} \, dt < \infty$.
If $d-2+|\alpha| + |\beta| = 0$ one obtains a logarithmic term, as is well known.
\end{proof}

In the self-adjoint case and real valued $V$ we next drop the condition 
that $V$ is positive.
In contrast to the previous theorem, in this case the constants are not uniform
with respect to the coefficients~$C$.

\begin{thm} \label{tdtnpc502}
Let $\kappa \in (0,1)$ and $\mu,M > 0$.
Let $\Omega \subset \Ri^d$ be an open bounded set with a $C^{1+\kappa}$-boundary.
Let
$C \in \ce^\kappa(\Omega,\mu,M)$ be real symmetric, $k,l \in \{ 1,\ldots,d \} $ and 
$V \in L_\infty(\Omega,\Ri)$.
Suppose that $0 \not\in \sigma(A_D + V)$.
Then the operator $(A_D + V)^{-1}$ has 
a kernel $G_V \colon \{ (x,y) \in \Omega \times \Omega : x \neq y \} \to \Ri$, 
which is differentiable in each variable and the derivative is 
differentiable in the other variable.
Moreover, for every multi-index $\alpha,\beta$ with 
$0 \leq |\alpha|,|\beta| \leq 1$ the function
$\partial_x^\alpha \, \partial_y^\beta \, G_V$ 
extends to a continuous function on 
$ \{ (x,y) \in \overline \Omega \times \overline \Omega : x \neq y \} $
with estimates
\[
|(\partial_x^\alpha \, \partial_y^\beta \, G_V)(x,y)| 
\leq \left\{ \begin{array}{ll}
c \, |x-y|^{-(d-2+|\alpha| + |\beta|)}  & \mbox{if } d-2+|\alpha| + |\beta| \neq 0 \\[5pt]
c \, \log( 1 + \frac{1}{|x-y|} ) & \mbox{if } d-2+|\alpha| + |\beta| = 0
                 \end{array} \right.
\]
and
\[
|(\partial_x^\alpha \, \partial_y^\beta \, G_V)(x',y') - (\partial_x^\alpha \, \partial_y^\beta \, G_V)(x,y)| 
\leq c \, \frac{(|x'-x| + |y'-y|)^\kappa}{|x-y|^{d-2+|\alpha| + |\beta| +\kappa}}  
\]
for all $x,x',y,y' \in \Omega$ with $x \neq y$ and
$|x-x'| + |y-y'| \leq \frac{1}{2} \, |x-y|$.
\end{thm}
\begin{proof}
There exists a $\lambda > 0$ such that $V + \lambda \geq 0$.
Replacing $V$ by $V + \lambda$, it suffices to show that 
for all $\lambda > 0$ and $V \in L_\infty(\Omega,\Ri)$ with $V \geq 0$
and $0 \not\in \sigma(A_D + V - \lambda \, I)$
the operator $(A_D + V - \lambda \, I)^{-1}$ has 
a kernel, denoted by $G_V$,
which is differentiable in each variable and the derivative is 
differentiable in the other variable.
Moreover, for every multi-index $\alpha,\beta$ with 
$0 \leq |\alpha|,|\beta| \leq 1$ the function
$\partial_x^\alpha \, \partial_y^\beta \, G_V$ 
extends to a continuous function on 
$ \{ (x,y) \in \overline \Omega \times \overline \Omega : x \neq y \} $
and
\[
|(\partial_x^\alpha \, \partial_y^\beta \, G_V)(x,y)| 
\leq c \, |x-y|^{-(d-2+|\alpha| + |\beta|)} \]
for all $x,y \in \Omega$ with $x \neq y$.

Since $A_D + V$ is a positive self-adjoint operator with compact resolvent,
there exist an orthonormal basis $(u_n)_{n \in \Ni}$ for $L_2(\Omega)$ 
of eigenfunctions for $A_D + V$, and a sequence $(\lambda_n)_{n \in \Ni}$
in $[0,\infty)$ such that $(A_D + V) u_n = \lambda_n \, u_n$
for all $n \in \Ni$.
There exists an $N \in \Ni$ such that $\lambda_n > \lambda$ for all 
$n \in \{ N+1,N+2,\ldots \} $.
Let $\omega_1 = \min \{ \lambda_n : n \in \{ N+1,N+2,\ldots \} \} $.
Then $\lambda < \omega_1$.
Let $P \colon L_2(\Omega) \to L_2(\Omega)$ be the orthogonal projection
onto $\spann \{ u_1,\ldots,u_N \} $.
Write $T_t = e^{-t(A_D + V)}$ for all $t > 0$.
So $T$ is the semigroup generated by $-(A_D + V)$.
Then $\|(I-P) \, T_t \, (I-P)\|_{2 \to 2} \leq e^{- \omega_1 t}$ for all $t > 0$.
Note that $P$ commutes with $T_t$ and the resolvent
$(A_D + V - \lambda \, I)^{-1}$ for all $t > 0$.
Hence on $L_2(\Omega)$ one has the decomposition
\[
(A_D + V - \lambda \, I)^{-1}
= P \, (A_D + V - \lambda \, I)^{-1} \, P
   + \int_0^\infty e^{\lambda t} \, (I-P) \, T_t \, (I-P) \, dt
.  \]
As a consequence 
\begin{equation}
\partial^\alpha \, (A_D + V - \lambda \, I)^{-1} \, \partial^\beta
= \partial^\alpha \, P \, (A_D + V - \lambda \, I)^{-1} \, P \, \partial^\beta
   + \int_0^\infty e^{\lambda t} \, \partial^\alpha \, (I-P) \, T_t \, (I-P) \, \partial^\beta \, dt
. 
\label{etdtnpc502;1}
\end{equation}
We shall show that the terms on the right hand side of (\ref{etdtnpc502;1})
has a kernel with the appropriate bounds and which extends continuously to 
$ \{ (x,y) \in \overline \Omega \times \overline \Omega : x \neq y \} $.

Let $t > 0$.
Then $T_t$ maps $L_2(\Omega)$ into $C^{1+\kappa}(\Omega)$ by 
Proposition~\ref{pdtnpc416}.
Hence $e^{-\lambda_n t} \, u_n = T_t u_n \in C^{1+\kappa}(\Omega)$
for all $n \in \Ni$.
In particular, $u_n \in C^{1+\kappa}(\Omega)$ and 
$\partial^\alpha u_n \in C^\kappa(\Omega) \subset C(\overline \Omega)$.
Since 
\[
P \, (A_D + V - \lambda \, I)^{-1} \, P u
= \sum_{n=1}^N \frac{(u,u_n)_{L_2(\Omega)}}{\lambda_n - \lambda} \, u_n
\]
for all $u \in L_2(\Omega)$, it follows that 
\[
\partial^\alpha \, P \, (A_D + V - \lambda \, I)^{-1} \, P \, \partial^\beta \, u
= (-1)^{|\beta|} \sum_{n=1}^N \frac{(u,\partial^\beta \, u_n)_{L_2(\Omega)}}{\lambda_n - \lambda} \, \partial^\alpha \, u_n
\]
for all $u \in W^{1,2}(\Omega)$, where we used that $u_n \in W^{1,2}_0(\Omega)$
for all $n \in \{ 1,\ldots,N \} $.
Therefore the operator $\partial^\alpha \, P \, (A_D + V - \lambda \, I)^{-1} \, P \, \partial^\beta$
has as kernel the function
\[
(x,y) \mapsto  (-1)^{|\beta|} \sum_{n=1}^N \frac{1}{\lambda_n - \lambda} \, 
    (\partial^\alpha \, u_n)(x) \, \overline{(\partial^\beta \, u_n)(y)}
,  \]
which is $\kappa$-H\"older continuous and extends to a continuous and bounded function 
on $\overline \Omega \times \overline \Omega$.
In particular, it can be estimated by $c \, |x-y|^{-d}$ for a suitable 
$c > 0$, since $\Omega$ is bounded.
This covers the first term on the right hand side of (\ref{etdtnpc502;1}).

We split the integral in (\ref{etdtnpc502;1}) in two parts: over
$(0,3]$ and $[3,\infty)$.
We start with the integral over $[3,\infty)$.
We shall show that the operator 
\begin{equation}
\int_3^\infty e^{\lambda t} \, \partial^\alpha \, (I-P) \, T_t \, (I-P) \, \partial^\beta \, dt
\label{etdtnpc502;20}
\end{equation}
has as kernel the function
\[
(x,y) \mapsto (-1)^{|\beta|} \int_3^\infty e^{\lambda t} 
   \sum_{n=N+1}^\infty e^{-\lambda_n t} \, (\partial^\alpha \, u_n)(x) \, 
          \overline{(\partial^\beta \, u_n)(y)} \, dt
.  \]
By Proposition~\ref{pdtnpc416} there exist $c,\omega > 0$ such that 
\begin{equation}
\|\partial^\alpha \, T_s u\|_{L_\infty(\Omega)}
\leq c \, s^{-d/4} \, s^{-|\alpha| / 2} \, e^{\omega s} \, \|u\|_{L_2(\Omega)}
\label{etdtnpc502;22}
\end{equation}
for all $s > 0$ and $u \in L_2(\Omega)$.
Therefore
$e^{-s \lambda_n} \, \|\partial^\alpha \, u_n\|_{L_\infty(\Omega)}
\leq c \, s^{-M} \, e^{\omega s}$
for all $n \in \Ni$ and $s > 0$, where $M = \frac{d}{4} + \frac{|\alpha|}{2}$.
Choosing $s = \lambda_n^{-1}$ gives 
$\|\partial^\alpha \, u_n\|_{L_\infty(\Omega)} \leq c \, e \, \lambda_n^M \, e^{\omega \lambda_n^{-1}}
\leq c \, e^{1 + \omega \omega_1^{-1}} \, \lambda_n^M$
if $n \geq N+1$.
Let $\varepsilon > 0$ be such that $\omega_1 (1 - 2 \varepsilon) > \lambda$.
There exists a $c_2 > 0$ such that $h^M \leq c_2 \, e^{\varepsilon h}$
for all $h \in (0,\infty)$.
Then
\[
\|\partial^\alpha \, u_n\|_{L_\infty(\Omega)}
\leq c \, c_2 \, e^{1 + \omega \omega_1^{-1}} \, t^{-M} \, e^{\varepsilon \lambda_n t}
\leq c_3 \, e^{\varepsilon \lambda_n t}
\]
for all $t \in [3,\infty)$ and $n \in \Ni$, 
where $c_3 = c \, c_2 \, e^{1 + \omega \omega_1^{-1}}$.
If $x,y \in \Omega$, then 
\begin{eqnarray*}
\int_3^\infty e^{\lambda t} 
   \sum_{n=N+1}^\infty e^{-\lambda_n t} \, |(\partial^\alpha \, u_n)(x) \, 
          \overline{(\partial^\beta \, u_n)(y)}| \, dt
& \leq & c_3^2 \int_3^\infty \sum_{n=N+1}^\infty e^{\lambda t} \, e^{-\lambda_n t} \, 
    e^{2 \varepsilon \lambda_n t} \, dt  \\
& = & c_3^2 \sum_{n=N+1}^\infty 
    \frac{e^{-3 (\lambda_n (1 - 2 \varepsilon) - \lambda)}}{\lambda_n (1 - 2 \varepsilon) - \lambda}  \\
& \leq & \frac{c_3^2 \, e^{3 \lambda}}{\omega_1 (1 - 2 \varepsilon) - \lambda}
   \sum_{n=N+1}^\infty e^{-3 \lambda_n (1 - 2 \varepsilon)}  
.  
\end{eqnarray*}
Next 
\[
\sum_{n=N+1}^\infty e^{-3 \lambda_n (1 - 2 \varepsilon)} 
\leq \Tr T_{3 (1 - 2 \varepsilon)}
= \int_\Omega H_{3 (1 - 2 \varepsilon)}(x,x) \, dx
< \infty
.  \]
Hence one can define $K \colon \Omega \times \Omega \to \Ci$ by
\[
K(x,y)
= (-1)^{|\beta|} \int_3^\infty e^{\lambda t} 
   \sum_{n=N+1}^\infty e^{-\lambda_n t} \, (\partial^\alpha \, u_n)(x) \, 
          \overline{(\partial^\beta \, u_n)(y)} \, dt
.  \]
Then $K$ is the kernel of (\ref{etdtnpc502;20}).
We already proved that $K$ is bounded on $\Omega \times \Omega$.

Using the $C^\kappa$-estimate in Proposition~\ref{pdtnpc416} instead of 
(\ref{etdtnpc502;22}), it follows similarly as above that 
there exists a $c_4 > 0$ such that 
$|||\partial^\alpha u_n|||_{C^\kappa(\Omega)} \leq c_4 e^{\varepsilon \lambda_n t}$
for all $t \in [3,\infty)$ and $n \in \Ni$.
Then 
\[
|(\partial^\alpha \, u_n)(x') \, \overline{(\partial^\beta \, u_n)(y')}
   - (\partial^\alpha \, u_n)(x) \, \overline{(\partial^\beta \, u_n)(y)}|
\leq 2 c_3 \, c_4 \, (|x-x'|^\kappa + |y-y'|^\kappa) \, e^{2 \varepsilon \lambda_n t}
\]
for all $n \in \Ni$, $t \in [3,\infty)$ and 
$x,x',y,y' \in \Omega$ with $|x-x'| \leq 1$ and $|y-y'| \leq 1$.
Arguing as before we obtain that there exists a $c_5 > 0$ such that 
\[
|K(x',y') - K(x,y)| \leq c_5 \, (|x-x'|^\kappa + |y-y'|^\kappa)
\]
for all $x,x',y,y' \in \Omega$ with $|x-x'| \leq 1$ and $|y-y'| \leq 1$.
Since $\Omega$ is bounded, there exists a $c_6 > 0$ such that 
\[
|K(x',y') - K(x,y)|
\leq c_6 \, \frac{(|x'-x| + |y'-y|)^\kappa}{|x-y|^{d-2+|\alpha| + |\beta| +\kappa}}   \label{etdtnpc502;140} 
\]
for all $x,x',y,y' \in \Omega$ with $x \neq y$ and
$|x-x'| + |y-y'| \leq \frac{1}{2} \ |x-y|$.
This completes the part of the integral in (\ref{etdtnpc502;1}) over $[3,\infty)$.

We split the part of the integral in (\ref{etdtnpc502;1}) over $(0,3]$ in two parts
\begin{equation}
\int_0^3 e^{\lambda t} \, \partial^\alpha \, (I-P) \, T_t \, (I-P) \, \partial^\beta \, dt
= \int_0^3 e^{\lambda t} \, \partial^\alpha \, T_t \, \partial^\beta \, dt
   - \int_0^3 e^{\lambda t} \, \partial^\alpha \, T_t \, P \, \partial^\beta \, dt
.  
\label{etdtnpc502;2}
\end{equation}
Since 
\[
\partial^\alpha \, T_t \, P \, \partial^\beta u 
= (-1)^{|\beta|} \sum_{n=1}^N e^{- \lambda_n t} (u, \partial^\beta \, u_n) \, \partial^\alpha \, u_n
\]
for all $u \in W^{1,2}(\Omega)$ it follows that the kernel of the 
second term in (\ref{etdtnpc502;2}) is 
\[
(x,y) \mapsto (-1)^{|\beta| + 1} \sum_{n=1}^N (\partial^\alpha \, u_n)(x) \, 
      \overline{(\partial^\beta \, u_n)(y)} \, 
    \int_0^3 e^{- \lambda_n t} \, dt
,  \]
which again is $\kappa$-H\"older continuous and
can be extended once more to a continuous and bounded 
function on $\overline \Omega \times \overline \Omega$.
Finally, it follows from Theorem~\ref{tdtnpc401} that 
there are $a,b,\omega > 0$ such that 
the operator $T_t$ has a kernel $H_t$ for all $t > 0$,
which is once differentiable in each entry,
satisfying the bounds
\[
|(\partial_x^\alpha \, \partial_y^\beta \, H_t)(x,y)|
\leq a \, t^{-d/2} \, t^{-(|\alpha| + |\beta|)/2} \, 
      e^{-b \, \frac{|x-y|^2}{t}} \, e^{-\omega t}
\]
for all $t > 0$ and $x,y \in \Omega$.
Moreover, $(x,y) \mapsto (\partial_x^\alpha \, \partial_y^\beta \, H_t)(x,y)$
extends to a continuous
function on $\overline \Omega \times \overline \Omega$.
Hence the operator $\int_0^3 e^{\lambda t} \, \partial^\alpha \, T_t \, \partial^\beta \, dt$
has kernel 
\[
(x,y) \mapsto 
(-1)^{|\beta|} \int_0^3 e^{\lambda t} \, (\partial_x^\alpha \, \partial_y^\beta \, H_t)(x,y) \, dt
\]
on $ \{ (x,y) \in \Omega \times \Omega : x \neq y \} $.
This kernel extends to a continuous function on 
$ \{ (x,y) \in \overline \Omega \times \overline \Omega : x \neq y \} $.
If $x \neq y$ and $d-2+|\alpha| + |\beta| \neq 0$, then 
\[
\Big| \int_0^3 e^{\lambda t} \, (\partial_x^\alpha \, \partial_y^\beta \, H_t)(x,y) \, dt \Big|
\leq a \, |x-y|^{-(d-2+|\alpha| + |\beta|)} \, e^{3 \lambda } \, 
      \int_0^\infty t^{-d/2} \, t^{-(|\alpha| + |\beta|)/2} e^{-b/t} \, dt
.  \]
The H\"older bounds follows similarly.
If $d-2+|\alpha| + |\beta| = 0$, then the obvious adjustments are needed to obtain 
a logarithmic term.
Then the resolvent kernel bounds follow by adding the terms.
\end{proof}

We next consider the operator $\partial_k (A_D + V)^{-1}$.
We obtain uniform bounds if $V = 0$.

\begin{prop} \label{pdtnpc503}
Let $\kappa \in (0,1)$ and $\mu,M > 0$.
Let $\Omega \subset \Ri^d$ be an open bounded set with a $C^{1+\kappa}$-boundary.
Let $p \in (d + 2 \kappa, \infty)$.
Then there exists a $c > 0$ such that the following is valid.
Let $C \in \ce^\kappa(\Omega,\mu,M)$ and $k \in \{ 1,\ldots,d \} $.
Then the operator $\partial_k \, A_D^{-1}$ is bounded from 
$L_p(\Omega)$ into $C^{2\kappa / p}(\Omega)$ with norm at most~$c$.
Moreover, if $V \in L_\infty(\Omega)$ and $0 \not\in \sigma(A_D + V)$,
then the operator $\partial_k \, (A_D + V)^{-1}$ is bounded from 
$L_p(\Omega)$ into $C^{2\kappa / p}(\Omega)$.
\end{prop}
\begin{proof}
Write $T_t = e^{-t A_D}$ for all $t > 0$.
Then it follows from Proposition~\ref{pdtnpc416} and Lemma~\ref{ldtnpc413} 
that there exist $c,\omega > 0$ 
such that the operator $\partial_k \, T_t$ is bounded
from $L_2(\Omega)$ into $C^\kappa(\Omega)$ with norm bounded by 
$c \, t^{-d/4} \, t^{-1/2} \, t^{-\kappa / 2} \, e^{-\omega t}$, uniformly for 
all $t \in (0,\infty)$.
The Gaussian kernel bounds with one derivative imply 
that $\partial_k \, T_t$ is bounded
from $L_\infty(\Omega)$ into $L_\infty(\Omega)$ norm with bounded by 
$c \, t^{-1/2} \, e^{-\omega t}$, possibly by increasing the value of $c$
and decreasing~$\omega$.
Hence by interpolation the operator $\partial_k \, T_t$
is bounded from $L_p(\Omega)$ into $C^{2\kappa / p}(\Omega)$
with norm bounded by $c \, t^{-d/(2p)} \, t^{-1/2} \, t^{-\kappa / p} \, e^{-\omega t}$ 
for all $t \in (0,\infty)$.
Since $p \in (d + 2 \kappa, \infty)$, the latter bound is 
integrable over $(0,\infty)$.
Hence $\partial_k \, A_D^{-1}$ is bounded from $L_p(\Omega)$ into $C^{2\kappa / p}(\Omega)$.
The norm is uniform for all $C \in \ce^\kappa(\Omega,\mu,M)$ by construction.

Finally, since $\partial_k \, (A_D + V)^{-1} = \partial_k \, A_D^{-1} \, A_D (A_D + V)^{-1}$ and 
the operator $A_D (A_D + V)^{-1} = I - M_V (A_D + V)^{-1}$ is bounded from 
$L_p(\Omega)$ into $L_p(\Omega)$, the last part follows.
\end{proof}

\begin{lemma} \label{ldtnpc430}
Let $\kappa \in (0,1)$ and $\mu,M > 0$.
Let $\Omega \subset \Ri^d$ be an open bounded set with a $C^{1+\kappa}$-boundary.
Then there exists a $c > 0$ such that the following is valid.
Let $C \in \ce^\kappa(\Omega,\mu,M)$, $p \in [1,\infty]$ and 
$k \in \{ 1,\ldots,d \} $.
Then $\|\partial_k \, A_D^{-1}\|_{p \to p} \leq c$.
Moreover, if $V \in L_\infty(\Omega)$ and $0 \not\in \sigma(A_D + V)$,
then the operator $\partial_k \, (A_D + V)^{-1}$ is bounded from $L_p(\Omega)$
into $L_p(\Omega)$.
\end{lemma}
\begin{proof}
Let $T$ be the semigroup generated by $-A_D$.
By Theorem~\ref{tdtnpc401} there exist $c,\omega > 0$, depending only of 
$\kappa$, $\mu$, $M$ and $\Omega$, such that 
$\|\partial_k \, T_t\|_{p \to p} \leq c \, t^{-1/2} \, e^{-\omega t}$
for all $t > 0$, $p \in [1,\infty]$ and $k \in \{ 1,\ldots,d \} $.
Then 
\[
\|\partial_k \, A_D^{-1}\|_{p \to p}
\leq \int_0^\infty \|\partial_k \, T_t\|_{p \to p} \, dt
\leq c \int_0^\infty t^{-1/2} \, e^{-\omega t} \, dt
.  \]
Finally, $\partial_k \, (A_D + V)^{-1} = \partial_k \, A_D^{-1} \, (I - M_V \, (A_D + V)^{-1})$
is bounded on $L_p(\Omega)$ for all $p \in [1,\infty]$ and $k \in \{ 1,\ldots,d \} $.
\end{proof}

\begin{prop} \label{pdtnpc431}
Let $\kappa \in (0,1)$ and $\mu,M > 0$.
Let $\Omega \subset \Ri^d$ be an open bounded set with a $C^{1+\kappa}$-boundary
and let $p \in (1,\infty)$.
Then there exists a $c > 0$ such that the following is valid.
Let $C \in \ce^\kappa(\Omega,\mu,M)$ be real symmetric and 
$k,l \in \{ 1,\ldots,d \} $.
Then the operator $\partial_k \, A_D^{-1} \, \partial_l$
extends to a bounded operator on $L_p(\Omega)$ with norm at most $c$.
Moreover, if $V \in L_\infty(\Omega)$ and $0 \not\in \sigma(A_D + V)$,
then the operator $\partial_k \, (A_D + V)^{-1} \, \partial_l$ 
extends to a bounded operator on $L_p(\Omega)$.
\end{prop}
\begin{proof} 
For $p=2$ the operator 
$\partial_k \, A_D^{-1} \, \partial_l = (\partial_k \, A_D^{-1/2}) (A_D^{-1/2} \, \partial_l)$
extends to a bounded operator with norm at most $\mu^{-1}$.
By Theorem~\ref{tdtnpc501} it follows that 
the kernel of $\partial_k \, A_D^{-1} \, \partial_l$ has Calder\'on--Zygmund estimates
uniformly in $C$.
Hence $\partial_k \, A_D^{-1} \, \partial_l$ extends to a bounded operator on $L_p(\Omega)$.

Finally, if $V \in L_\infty(\Omega)$ and $0 \not\in \sigma(A_D + V)$,
then
\[
\partial_k \, (A_D + V)^{-1} \, \partial_l 
= \partial_k \, A_D^{-1} \, \partial_l 
     - \Big( \partial_k \, (A_D + V)^{-1} \Big) \, M_V \, 
       \Big( A_D^{-1} \, \partial_l \Big)
\]
and use Lemma~\ref{ldtnpc430}.
\end{proof}

\section{The harmonic lifting} \label{Sdtnpc5new}

Let $C \in \ce^\kappa(\Omega)$ be real symmetric and $V \in L_\infty(\Omega,\Ri)$,
where $\Omega$ is a bounded Lipschitz domain.
Recall that the harmonic lifting $\gamma_V \colon \Tr(W^{1,2}(\Omega)) \to W^{1,2}(\Omega)$
is defined by 
\[ 
\gamma_V  \varphi = u
\]
for all $\varphi \in \Tr(W^{1,2}(\Omega))$, where 
$u \in W^{1,2}(\Omega)$ is such that $(\ca u + V) u = 0$ and $\Tr u = \varphi$. 
In this section we shall prove that $\gamma_V$ has a kernel 
and we obtain good kernel bounds if $\Omega$ has a $C^{1+\kappa}$-boundary.
We also show that the map $\gamma_V$ extends to a continuous map 
from $L_p(\Gamma)$ into $L_p(\Omega)$ for all $p \in [1,\infty]$.

For the proof of these results we need a delicate version of the 
divergence theorem.

\begin{lemma} \label{ldtnpc530}
Let $\Omega \subset \Ri^d$ be a bounded open set with $C^1$-boundary.
Let $F \colon \Omega \to \Ci^d$ be a function.
Suppose $F \in C(\overline \Omega,\Ci^d) \cap C^1(\Omega,\Ci^d)$ and suppose that 
$\divv F \in L_1(\Omega)$.
Then $\int_\Omega \divv F = \int_\Gamma n \cdot F$.
\end{lemma} 
\begin{proof} 
See \cite{Altdiv}.
\end{proof}

We use this divergence theorem to obtain a classical expression  of the normal 
derivative.

\begin{lemma} \label{ldtnpc531} 
Let $\kappa \in (0,1)$.
Let $\Omega \subset \Ri^d$ be an open bounded set with a $C^{1+\kappa}$-boundary.
Let $C \in \ce(\Omega)$ be real symmetric and suppose that $c_{kl} \in C^{1 + \kappa}(\Omega)$ 
for all $k,l \in \{ 1,\ldots,d \} $.
Let $p \in (d,\infty)$ and $u \in C^1(\overline \Omega)$.
Suppose that $\ca u \in L_2(\Omega)$.
Then $u$ has a weak conormal derivative and 
\[
\partial_\nu^C u 
= \sum_{k,l=1}^d n_k \, ( c_{kl} \, \partial_l u)|_\Gamma
.  \]
\end{lemma}
\begin{proof}
Interior regularity gives $u \in C^2(\Omega)$.
Let $v \in C_b^\infty(\Omega)$.
Define $F = (F_1,\ldots,F_d) \colon \overline \Omega \to \Ci^d$ by 
$F_k = \sum_{l=1}^d c_{kl} \, (\partial_l u) \, \overline v$.
Then $F \in C(\overline \Omega,\Ci^d) \cap C^1(\Omega,\Ci^d)$.
Moreover, 
\[
\divv F
= - (\ca u) \, \overline v
   + \sum_{k,l=1}^d c_{kl} \,  (\partial_l u) \, \overline{\partial_k v}
\in L_2(\Omega) \subset L_1(\Omega)
.  \]
Then the divergence theorem, Lemma~\ref{ldtnpc530} gives
\[
\int_\Omega \sum_{k,l=1}^d c_{kl} \, (\partial_k u) \, \overline{\partial_l v}
   - \int_\Omega (\ca u) \, \overline v
= \int_\Omega \divv F
= \int_\Gamma n \cdot F
= \int_\Gamma \sum_{k,l=1}^d n_k \, c_{kl} \, (\partial_l u) \, \overline v
.  \]
Since $\sum_{k,l=1}^d n_k \, (c_{kl} \, \partial_l u)|_\Gamma \in C(\Gamma) \subset L_2(\Gamma)$,
this proves the lemma.
\end{proof}

Note that we required $c_{kl} \in C^{1 + \kappa}(\Omega)$ in Lemma~\ref{ldtnpc531}, 
which is much more
than the condition $c_{kl} \in C^\kappa(\Omega)$ in Theorem~\ref{thm1.1}.
This is the reason why we use a regularisation of the coefficients below.

\begin{prop} \label{pdtnpc533}
Let $\kappa \in (0,1)$.
Let $\Omega \subset \Ri^d$ be an open bounded set with a $C^{1+\kappa}$-boundary.
Let $C \in \ce^\kappa(\Omega)$ be real symmetric and $V \in L_\infty(\Omega,\Ri)$.
Suppose that $0 \not\in \sigma(A_D + V)$.
Let $p \in (d + 2 \kappa,\infty)$ and $u \in L_p(\Omega)$.
Then $(A_D + V)^{-1} u$ has a weak conormal derivative and 
\begin{equation} 
\partial_\nu^C (A_D + V)^{-1} u 
= \sum_{k,l=1}^d n_k \, \Tr( c_{kl} \, \partial_l \, (A_D + V)^{-1} u)
.  
\label{epdtnpc533;1}
\end{equation}
\end{prop}
\begin{proof}
\firststep\label{pdtnpc533step1}\ {\bf Suppose $V = 0$ and $c_{kl} \in C^{1 + \kappa}(\Omega)$ 
for all $k,l \in \{ 1,\ldots,d \} $.}  \\
Let $u \in L_p(\Omega)$.
Then $A_D^{-1} u \in C^1(\overline \Omega)$ by Proposition~\ref{pdtnpc503}.
So by Lemma~\ref{ldtnpc531} one deduces that $A_D^{-1} u$ has a conormal derivative 
and (\ref{epdtnpc533;1}) is valid.

\nextstep\label{pdtnpc533step2}\ {\bf Suppose $V = 0$.}  \\
We can extend the function $c_{kl}$ to a  
$C^\kappa$-function $\tilde c_{kl} \colon \Ri^d \to \Ri$
such that $\tilde c_{kl} = \tilde c_{lk}$ for all $k,l \in \{ 1,\ldots,d \} $.
Let $(\rho_n)_{n \in \Ni}$ be a bounded approximation of the identity.
For all $n \in \Ni$ and $k,l \in \{ 1,\ldots,d \} $ define 
$c^{(n)}_{kl} = (\tilde c_{kl} * \rho_n)|_\Omega$ and set 
$C^{(n)} = (c^{(n)}_{kl})_{k,l}$.
Then there are $\mu,M > 0$ such that 
$C^{(n)} \in \ce^\kappa(\Omega,\mu,M)$ for all large $n \in \Ni$
and without loss of generality for all $n \in \Ni$.
Define $A_D^{(n)} = A_D^{C^{(n)}}$ for all $n \in \Ni$.

Let $u \in L_p(\Omega)$.
Then it follows from from Step~\ref{pdtnpc533step1} that 
\begin{equation}
\int_\Omega \sum_{k,l=1}^d c^{(n)}_{kl} \, (\partial_k \, (A_D^{(n)})^{-1} u) \, \overline{\partial_l v}
   - \int_\Omega u \, \overline v
= \int_\Gamma \sum_{k,l=1}^d n_k \, \Tr( c^{(n)}_{kl} \, \partial_l \, (A_D^{(n)})^{-1} u) \, \overline{\Tr v}
\label{epdtnpc533;2}
\end{equation}
for all $n \in \Ni$ and $v \in W^{1,2}(\Omega)$.
Clearly $\lim c^{(n)}_{kl} = c_{kl}$ uniformly on $\overline \Omega$ for all 
$k,l \in \{ 1,\ldots,d \} $.
Let $k \in \{ 1,\ldots,d \} $.
If $w \in D(A_D)$ and $v \in L_2(\Omega)$, then 
\[
(w,v)_{L_2(\Omega)} 
= (A_D w, A_D^{-1} v)_{L_2(\Omega)}
= \sum_{j,l=1}^d (c_{jl} \, \partial_l w, \partial_j \, A_D^{-1} v)_{L_2(\Omega)}
.  \]
Hence by density 
$(w,v)_{L_2(\Omega)} = \sum_{j,l=1}^d (c_{jl} \, \partial_l w, \partial_j \, A_D^{-1} v)_{L_2(\Omega)}$
for all $w \in W^{1,2}_0(\Omega)$ and $v \in L_2(\Omega)$.
Substituting $w = (A_D^{(n)})^{-1} u$, replacing $v$ by $\partial_k v$
and integration by parts gives
\[
- (\partial_k \, (A_D^{(n)})^{-1} u,v)_{L_2(\Omega)} 
= \sum_{j,l=1}^d
    (c_{jl} \, \partial_l \, (A_D^{(n)})^{-1} u, \partial_j \, A_D^{-1} \, \partial_k v)_{L_2(\Omega)}
\]
for all $v \in C_c^\infty(\Omega)$.
Similarly and slightly easier one proves
\[
- (\partial_k \, A_D^{-1} u,v)_{L_2(\Omega)} 
= \sum_{j,l=1}^d
    (c^{(n)}_{jl} \, \partial_l \, (A_D^{(n)})^{-1} u, \partial_j \, A_D^{-1} \, \partial_k v)_{L_2(\Omega)}
\]
for all $v \in C_c^\infty(\Omega)$.
Therefore 
\[
(\partial_k \, (A_D^{(n)})^{-1} u - \partial_k \, A_D^{-1} u, v)_{L_2(\Omega)}
= \sum_{j,l=1}^d
    ((c^{(n)}_{jl} - c_{jl}) \, \partial_l \, (A_D^{(n)})^{-1} u, \partial_j \, A_D^{-1} \, \partial_k v)_{L_2(\Omega)}
\]
for all $v \in C_c^\infty(\Omega)$.
Let $q$ be the dual exponent of $p$.
By Proposition~\ref{pdtnpc431} the operator 
$\partial_j \, A_D^{-1} \, \partial_k$ extends to a bounded operator $T_{jk}$ on $L_q(\Omega)$
for all $j \in \{ 1,\ldots,d \} $.
Then 
\begin{eqnarray*}
\lefteqn{
|(\partial_k \, (A_D^{(n)})^{-1} u - \partial_k \, A_D^{-1} u, v)_{L_2(\Omega)}|
} \hspace*{20mm} \\*
& \leq & \sum_{j,l=1}^d
    \|c^{(n)}_{jl} - c_{jl}\|_{L_\infty(\Omega)} \, 
    \|\partial_l \, (A_D^{(n)})^{-1}\|_{p \to p} \, \|u\|_{L_p(\Omega)} \, 
   \|T_{jk}\|_{q \to q} \, \|v\|_{L_q(\Omega)}
\end{eqnarray*}
for all $v \in C_c^\infty(\Omega)$.
Since the operators $\partial_l \, (A_D^{(n)})^{-1}$ are bounded on $L_p(\Omega)$ 
uniformly in $n$ by Lemma~\ref{ldtnpc430}, one deduces that 
\begin{equation}
\lim_{n \to \infty} \partial_k \, (A_D^{(n)})^{-1} u = \partial_k \, A_D^{-1} u
\label{pdtnpc533;8}
\end{equation}
in $L_p(\Omega)$.
Therefore the left hand side of (\ref{epdtnpc533;2}) converges to 
$\int_\Omega \sum_{k,l=1}^d c_{kl} \, (\partial_k \, A_D^{-1} u) \, \overline{\partial_l v}$
for all $v \in C^\infty_b(\Omega)$.

Next we consider the right hand side of (\ref{epdtnpc533;2}).
Let $l \in \{ 1,\ldots,d \} $.
Then $(\partial_l \, (A_D^{(n)})^{-1} u)_{n \in \Ni}$ is bounded in $C^{2 \kappa / p}(\Omega)$ 
and in $C(\overline \Omega)$ by Proposition~\ref{pdtnpc503}.
So by the Arzel\`a--Ascoli theorem and passing to a subsequence if 
necessary there exists a $w \in C(\overline \Omega)$ such that 
$\lim_{n \to \infty} \partial_l \, (A_D^{(n)})^{-1} u = w$ in $C(\overline \Omega)$.
Since $\lim_{n \to \infty} \partial_l \, (A_D^{(n)})^{-1} u = \partial_l \, A_D^{-1} u$
in $L_p(\Omega)$, one deduces that $w = \partial_l \, A_D^{-1} u$.
So 
\[
\lim_{n \to \infty} 
\int_\Gamma \sum_{k,l=1}^d n_k \, \Tr( c^{(n)}_{kl} \, \partial_l \, (A_D^{(n)})^{-1} u) \, \overline{\Tr v}
= \int_\Gamma \sum_{k,l=1}^d n_k \, \Tr( c_{kl} \, \partial_l \, A_D^{-1} u) \, \overline{\Tr v}
\]
for all $v \in C^\infty_b(\Omega)$.
Then the equality in (\ref{epdtnpc533;2}) implies that 
$A_D^{-1} u$ has a weak conormal derivative and (\ref{epdtnpc533;1}) is valid.

\nextstep\label{pdtnpc533step3}\ {\bf Suppose $V \in L_\infty(\Omega,\Ri)$.}  \\
Let $u \in L_p(\Omega)$. 
Then apply Step~\ref{pdtnpc533step2} to $A_D \, (A_D + V)^{-1} u \in L_p(\Omega)$.
\end{proof}

\begin{lemma} \label{ldtnpc534}
Let $\kappa \in (0,1)$.
Let $\Omega \subset \Ri^d$ be an open bounded set with a $C^{1+\kappa}$-boundary.
Let $C \in \ce^\kappa(\Omega)$ be real symmetric and $V \in L_\infty(\Omega,\Ri)$.
Suppose that $0 \not\in \sigma(A_D + V)$.
Let $p \in (d + 2 \kappa,\infty)$ and $v \in L_p(\Omega)$.
Then
\[
(\gamma_V \varphi, v)_{L_2(\Omega)}
= - (\varphi, \partial_\nu^C \, (A_D + V)^{-1} v)_{L_2(\Gamma)}
\]
for all $\varphi \in \Tr (W^{1,2}(\Omega))$.
\end{lemma} 
\begin{proof}
It follows from Proposition~\ref{pdtnpc533} that 
$(A_D + V)^{-1} v$ has a weak conormal derivative.
Then the equality follows as in \cite{BeE1} Corollary~5.4. 
For more details, see \cite{AE7} Proposition~6.4.
\end{proof}

Let $\kappa \in (0,1)$, $\Omega \subset \Ri^d$ be an open bounded set with a $C^{1+\kappa}$-boundary,
$C \in \ce^\kappa(\Omega)$ and $V \in L_\infty(\Omega,\Ri)$.
Suppose that $0 \not\in \sigma(A_D + V)$.
Let $G_V$ be the Green kernel of $(A_D + V)^{-1}$.
Then $G_V$ is differentiable on $ \{ (x,y) \in \Omega \times \Omega : x \neq y \} $ 
by Theorem~\ref{tdtnpc502} and the derivative extends to a continuous
function on $ \{ (x,y) \in \overline \Omega \times \overline \Omega : x \neq y \} $.
Define the function $K_{\gamma_V} \colon \Omega \times \Gamma \to \Ci$ by 
\[
K_{\gamma_V}(x,z) 
= - \sum_{k,l=1}^d n_k(z) \, c_{kl}(z) (\partial^{(1)}_l G_V)(z,x)
.  \]
We next show that $K_{\gamma_V}$ is the kernel of $\gamma_V$.

\begin{prop} \label{pdtnpc535}
Let $\kappa \in (0,1)$, $\Omega \subset \Ri^d$ be an open bounded set with a $C^{1+\kappa}$-boundary,
$C \in \ce^\kappa(\Omega)$ real symmetric and $V \in L_\infty(\Omega,\Ri)$.
Suppose that $0 \not\in \sigma(A_D + V)$.
Then one has the following.
\begin{tabel}
\item \label{pdtnpc535-1}
The map $K_{\gamma_V}$ is continuous.
\end{tabel}
Define $T \colon L_1(\Gamma) \to C(\Omega)$ by 
\[
(T \varphi)(x) 
= \int_\Gamma K_{\gamma_V}(x,z) \, \varphi(z) \, dz
.  \]
\begin{tabel}
\setcounter{teller}{1}
\item \label{pdtnpc535-2}
If $\varphi \in \Tr (W^{1,2}(\Omega))$, then 
$\gamma_V \varphi = T \varphi$ a.e.
\item \label{pdtnpc535-3}
There exists a $c > 0$ such that 
\[
|K_{\gamma_V}(x,z)| \leq \frac{c}{|x-z|^{d-1}}
\quad \mbox{and} \quad
|K_{\gamma_V}(x',z') - K_{\gamma_V}(x,z)| 
\leq c \, \frac{(|x'-x| + |z'-z|)^\kappa}{|x-z|^{d-1+\kappa}}
\]
for all $x,x' \in \Omega$ and $z,z' \in \Gamma$ with $|x'-x| + |z'-z| \leq \frac{1}{2} \, |x-z|$.
\item \label{pdtnpc535-4}
Let $p \in [1,\infty]$.
Then the map $\gamma_V|_{L_p(\Gamma) \cap \Tr (W^{1,2}(\Omega))}$ 
extends to a bounded map from $L_p(\Gamma)$ into $L_p(\Omega)$.
Explicitly, the restriction $T|_{L_p(\Gamma)}$ is continuous from 
$L_p(\Gamma)$ into $L_p(\Omega)$.
\end{tabel}
\end{prop}
\begin{proof}
`\ref{pdtnpc535-1}'.
This follows immediately from Theorem~\ref{tdtnpc502}. 

`\ref{pdtnpc535-2}'.
Let $\varphi \in \Tr (W^{1,2}(\Omega))$ and $v \in C_c^\infty(\Omega)$.
Then Lemma~\ref{ldtnpc534} and Proposition~\ref{pdtnpc533} give
\begin{eqnarray*}
(\gamma_V \varphi, v)_{L_2(\Omega)}
& = & - (\varphi, \partial_\nu^C \, (A_D + V)^{-1} v)_{L_2(\Gamma)}  \\
& = & - \int_\Gamma \varphi(z) \sum_{k,l=1}^d n_k(z) \, c_{kl}(z) \, 
             \overline{ (\partial_l \, (A_D + V)^{-1} v)(z) } \, dz  \\
& = & - \int_\Gamma \int_\Omega \sum_{k,l=1}^d 
     \varphi(z) \, n_k(z) \, c_{kl}(z) \, (\partial^{(1)}_l G_V)(z,x) \, 
             \overline{ v(x) }  \, dx \, dz  \\
& = & \int_\Omega (T \varphi)(x) \,  \overline{ v(x) }  \, dx 
= (T \varphi, v)_{L_2(\Omega)}
.
\end{eqnarray*}
So $\gamma_V \varphi = T \varphi$ a.e.

`\ref{pdtnpc535-3}'.
This is a consequence of Theorem~\ref{tdtnpc502}.

`\ref{pdtnpc535-4}'.
We divide the proof in several steps.
\firststep\label{pdtnpc535step1}\ {\bf Suppose $p = 1$.}  \\
Since $\sup_{z \in \Gamma} \int_\Omega |K_{\gamma_V}(x,z)| \, dx < \infty$
the operator $T$ is bounded from $L_1(\Gamma)$ into $L_1(\Omega)$.

\nextstep\label{pdtnpc535step2}\ {\bf Suppose $p = \infty$ and $V = 0$.}  \\
We shall show that $T|_{L_\infty(\Gamma)}$ is bounded from $L_\infty(\Gamma)$
into $L_\infty(\Omega)$.
The maximum principle, \cite{GT} Theorem~8.1, gives that 
$\|\gamma_0 \varphi \|_{L_\infty(\Omega)} \leq \|\varphi\|_{L_\infty(\Gamma)}$
for all $\varphi \in C(\Gamma) \cap \Tr (W^{1,2}(\Omega))$.
So $\|T \varphi\|_{L_\infty(\Omega)} \leq \|\varphi\|_{L_\infty(\Gamma)}$
for all $\varphi \in C(\Gamma) \cap \Tr (W^{1,2}(\Omega))$.
Now let $\varphi \in L_\infty(\Gamma)$.
Since $\Omega$ has a Lipschitz boundary, one can regularise $\varphi$.
On a special Lipschitz domain one can regularise an $L_\infty$-function
$\psi$ on the boundary to obtain a sequence of continuous $W^{1,2}_\loc$-functions
on the boundary which converges to $\psi$ in the weak$^*$-topology on $L_\infty$
and such that the $L_\infty$-norm of the approximants is bounded by $\|\psi\|_\infty$.
Since $\Omega$ is bounded and Lipschitz one can use a partition of the unity so 
that $\Omega$ is split as a finite number, say~$N$, of parts of special 
Lipschitz domains. 
Summing up the corresponding smooth approximants one obtains a sequence 
$(\varphi_n)_{n \in \Ni}$ in $W^{1,2}(\Gamma) \cap C(\Gamma)$ such that 
$\lim \varphi_n = \varphi$ weak$^*$ in $L_\infty(\Gamma)$ and 
$\|\varphi_n\|_{L_\infty(\Gamma)} \leq N \, \|\varphi\|_{L_\infty(\Gamma)}$
for all $n \in \Ni$.
Now let $x \in \Omega$.
Then $z \mapsto K_{\gamma_0}(x,z)$ is an element of $L_1(\Gamma)$.
So 
\begin{eqnarray*}
|(T \varphi)(x)|
& = & \lim_{n \to \infty} |(T \varphi_n)(x)|
\leq \limsup_{n \to \infty} \|T \varphi_n\|_{L_\infty(\Omega)}
\leq \limsup_{n \to \infty} \|\varphi_n\|_{L_\infty(\Gamma)}
\leq N \, \|\varphi\|_{L_\infty(\Gamma)}
.  
\end{eqnarray*}
So $T|_{L_\infty(\Gamma)}$ is a bounded extension of 
$\gamma_0|_{L_\infty(\Gamma) \cap \Tr (W^{1,2}(\Omega))}$ from 
$L_\infty(\Gamma)$ into $L_\infty(\Omega)$.

\nextstep\label{pdtnpc535step3}\ {\bf Suppose $p \in [1,\infty]$ and $V = 0$.}  \\
This follows by interpolation from Steps~\ref{pdtnpc535step1} and \ref{pdtnpc535step2}.

\nextstep\label{pdtnpc535step4}\ {\bf Suppose $p \in [1,\infty]$ and $V \in L_\infty(\Omega,\Ri)$.}  \\
Since $\gamma_V = \gamma_0 - (A_D + V)^{-1} \, M_V \, \gamma_0$ by (\ref{eSdtnpc2;30})
the general case follows.
\end{proof}

As a consequence we deduce that $\cn_V$ is a perturbation of $\cn$.

\begin{cor} \label{cdtnpc536}
Let $\kappa \in (0,1)$, $\Omega \subset \Ri^d$ be an open bounded set with a $C^{1+\kappa}$-boundary,
$C \in \ce^\kappa(\Omega)$ real symmetric and $V \in L_\infty(\Omega,\Ri)$.
Suppose that $0 \not\in \sigma(A_D + V)$.
Then $\cn_V = \cn + \gamma_0^* \, M_V \, \gamma_V$.
\end{cor}
\begin{proof}
Let $\varphi \in D(\cn_V)$ and $\psi \in D(\cn)$.
Write $u = \gamma_V \varphi$ and $v = \gamma_0 \varphi$.
Then 
\begin{eqnarray*}
(\cn_V \varphi, \psi)_{L_2(\Gamma)} - (\varphi, \cn \psi)_{L_2(\Gamma)}
& = & \gota_V(u,v) - \gota_0(u,v)  \\
& = & \int_\Gamma V \, u \, \overline v
= (M_V \, \gamma_V \varphi, \gamma_0 \psi)_{L_2(\Gamma)}
= (\gamma_0^* \, M_V \, \gamma_V \varphi, \psi)_{L_2(\Gamma)}
.
\end{eqnarray*}
Since $\gamma_0^* \, M_V \, \gamma_V$ is bounded on $L_2(\Gamma)$ 
by Proposition~\ref{pdtnpc535}\ref{pdtnpc535-4} it follows that 
$\varphi \in D(\cn^*) = D(\cn)$ and similarly $\psi \in D(\cn_V^*) = D(\cn_V)$.
Then 
\[
((\cn_V - \cn) \varphi, \psi)_{L_2(\Gamma)} 
= (\cn_V \varphi, \psi)_{L_2(\Gamma)} - (\varphi, \cn \psi)_{L_2(\Gamma)}
= (\gamma_0^* \, M_V \, \gamma_V \varphi, \psi)_{L_2(\Gamma)}
.  \]
Since $D(\cn)$ is dense in $L_2(\Gamma)$ the corollary follows.
\end{proof}

\section{The Schwartz kernel of the Dirichlet-to-Neumann operator} \label{Sdtnpc6new}

Our main aim in this section is to show that the Dirichlet-to-Neumann operators $\cn$ and 
$\cn_V$ are given by  Schwartz  kernels that satisfy Calder\'on--Zygmund-type bounds.
The principle step in the proof is that the Schwartz kernel of $\cn$
can be expressed in terms of the coefficients 
$c_{kl}$ and partial derivatives of the Green kernel. 
We start with a definition.
Let $\kappa \in (0,1)$, $\Omega \subset \Ri^d$ be an open bounded set with a $C^{1+\kappa}$-boundary
and $C \in \ce^\kappa(\Omega)$ real symmetric.
Then by Theorem~\ref{tdtnpc501} the elliptic operator $A_D$ has a Green kernel $G$
which is differentiable in each entry and the partial derivatives extend to 
a continuous function on 
$ \{ (x,y) \in \overline \Omega \times \overline \Omega : x \neq y \} $.
We define $ K_\cn \colon \{(z,w) \in  \Gamma \times \Gamma : z \not= w \} \to \Ri$ by
\begin{equation} \label{d61} 
K_\cn(z,w) = -\sum_{k,l,k',l'=1}^d 
n_{k'}(w) \, n_k(z) \, c_{k'l'}(w) \, c_{kl}(z) \, (\partial_l^{(1)} \partial_{l'}^{(2)} G)(z,w) .
 \end{equation}
Our first aim is to prove that  $K_\cn$ is the Schwartz kernel of $\cn$. 
In the literature sometimes $K_\cn$ is written as $\partial^C_\nu \, \partial^C_{\nu'} G$, 
the conormal derivatives with respect to the two variables. 
It far from clear, however, whether the weak conormal derivatives
of $G$ exist in the sense of~(\ref{elpbdton201;15}).
Even if these weak conormal derivatives would exist, then it is again unclear
whether they coincide with (\ref{d61}).

We  use the definition of $\cn$ by the symmetric form $\gota_0$, 
see (\ref{eSdtnpc2;40}) and (\ref{formaV}) with   $V= 0$. 

\begin{lemma}\label{lem6.2}
Let $\kappa \in (0,1)$, $\Omega \subset \Ri^d$ be an open bounded set with a $C^{1+\kappa}$-boundary
and $C \in \ce(\Omega)$ real symmetric.
Suppose that $c_{kl} \in C^\infty_{\rm b}(\Omega)$ for all $k,l \in \{ 1,\ldots,d \} $.
Let $\varphi \in \Tr(W^{1,2}(\Omega))$.
Then 
\[
\gota_0(\gamma_0 \varphi, v) 
= \int_\Gamma \int_\Gamma K_\cn(z,w) \, \varphi(w) \, \overline{(\Tr v)(z)} \, dw \, dz
\]
for all $v \in W^{1,2}(\Omega) $ with $\supp \varphi \cap \supp v = \emptyset$.
\end{lemma}
\begin{proof} 
Let $\tau \in C_c^\infty(\Ri^d)$ be such that $\supp \varphi \cap \supp \tau = \emptyset$. 
Set $v = \tau|_\Omega$.
By definition
\begin{equation}\label{eq6.1}
\gota_0(\gamma_0 \varphi, v) 
= \sum_{k,l=1}^d \int_\Omega c_{kl} \, (\partial_k \gamma_0 \varphi) \, \overline{\partial_l v}.
\end{equation}
Let $k \in \{ 1,\ldots,d \} $.
Define $F_k \colon \Omega \to \Ci$ by
$F_k = \sum_{l=1}^d \overline v \, c_{kl} \, \partial_l (\gamma_0 \varphi)$.
Then $F_k \in C^1(\Omega)$ by classical elliptic regularity and the fact that 
$\ca(\gamma_0 \varphi)= 0$ weakly in $\Omega$. 
We next show that 
$F_k \in C(\overline{\Omega})$. 
Indeed, by Proposition~\ref{pdtnpc535}\ref{pdtnpc535-2} we have
\begin{equation}\label{eq601}
 (\partial_l \gamma_0 \varphi)(x) 
= -\sum_{k',l'=1}^d \int_\Gamma n_{k'}(z) \, c_{k'l'}(z) \, 
    (\partial_l^{(2)} \partial_{l'}^{(1)} G)(z,x) \, \varphi(z) \, dz
 \end{equation}
for all $x \in \supp v$. 
This integral is actually taken over $z \in \supp \varphi$. 
Since $\supp \varphi \cap \supp v = \emptyset$ we can apply Theorem~\ref{tdtnpc501}, 
which shows immediately that $F_k \in C(\overline{\Omega})$. 
Let $F = (F_1, \dots, F_d)$.
Then 
$\divv F = \sum_{k,l=1}^d c_{kl} \, (\partial_k \gamma_0 \varphi) \, \overline{\partial_l v} 
\in L_2(\Omega)$.  
Hence we can apply Lemma~\ref{ldtnpc530} to write the RHS of (\ref{eq6.1}) 
by
\[
\sum_{k,l=1}^d \int_\Omega c_{kl} \, (\partial_k \gamma_0 \varphi) \, \overline{\partial_l v}
= \int_\Omega \divv F 
= \int_\Gamma n \cdot F
= \sum_{k,l=1}^d \int_\Gamma n_k(w) \, 
      (\Tr( \overline v \, c_{kl} \, (\partial_l \gamma_0 \varphi)))(w) \, dw.
\]
Hence
\begin{eqnarray*}
\gota_0(\gamma_0 \varphi, v) 
&=& - \sum_{k,l,k',l'=1}^d \int_\Gamma \int_\Gamma 
   n_k(w) \, n_{k'}(z) \, \overline{(\Tr v)(w)} \, c_{kl}(w) \, c_{k'l'}(z) 
(\partial_{l'}^{(1)} \partial_l^{(2)} G)(z,w) \, \varphi(z) \, dz \, dw \\
&=&  \int_\Gamma \int_\Gamma K_\cn(z,w) \, \varphi(w) \, \overline{(\Tr v)(z)} \, dw \, dz ,
\end{eqnarray*}
which proves the lemma if 
$v = \tau|_\Omega $ with $\tau \in C_c^\infty( \Ri^d)$. 
This extends to all $v \in W^{1,2}(\Omega)$ with 
$\supp \varphi \cap \supp v = \emptyset$ by a standard approximation argument. 
\end{proof}

Next we extend the previous lemma to the case of H\"older continuous coefficients.

\begin{lemma}\label{lem6.3}
Let $\kappa \in (0,1)$, $\Omega \subset \Ri^d$ be an open bounded set with a $C^{1+\kappa}$-boundary
and $C \in \ce(\Omega)$ real symmetric.
Let $\varphi \in \Tr(W^{1,2}(\Omega))$.
Then 
\[
\gota_0(\gamma_0 \varphi, v) 
= \int_\Gamma \int_\Gamma K_\cn(z,w) \, \varphi(w) \, \overline{(\Tr v)(z)} \, dz \, dw
\]
for all $v \in W^{1,2}(\Omega) $ with $\supp \varphi \cap \supp v = \emptyset$.
\end{lemma}
\begin{proof} 
As one expects, we proceed by a regularization argument. 
For all $n \in \Ni$ let $C^{(n)}$ be as in Step~\ref{pdtnpc533step2} in the proof 
of Proposition~\ref{pdtnpc533}.
There exist $\mu,M > 0$ such that 
$C^{(n)} \in \ce^\kappa(\Omega,\mu,M)$ for all $n \in \Ni$.
We denote by $A^{(n)}$, $\gota_0^{(n)}$, $\gamma_0^{(n)}$, $K_\cn^{(n)}$, $G^{(n)}$ 
the same quantities  as before with $c_{kl}$ replaced by 
the new coefficients $c_{kl}^{(n)}$.  
We apply Lemma \ref{lem6.2} to obtain
\begin{equation}\label{eq6.3}
\gota_0^{(n)}(\gamma_0^{(n)} \varphi, v) 
= \int_\Gamma \int_\Gamma K_\cn^{(n)}(z,w) \, \varphi(w) \, \overline{(\Tr v)(z)} \, dz \, dw
\end{equation}
for all $\varphi \in \Tr(W^{1,2}(\Omega))$ and 
$v \in W^{1,2}(\Omega) $ such that $\supp \varphi \cap \supp v = \emptyset$.
Since 
\[
K_\cn^{(n)}(z,w) 
= -\sum_{k,l,k',l'=1}^d 
n_{k'}(w) \,  n_k(z) \, c_{k'l'}^{(n)}(w)\, c_{kl}^{(n)}(z) \, 
    (\partial_l^{(1)} \partial_{l'}^{(2)} G^{(n)})(z,w)
\]
it follows from Proposition~\ref{prop6.1} below that 
$\lim_{n \to \infty} K_\cn^{(n)}(z,w) = K_\cn(z,w)$ uniformly in 
$z \in \Gamma \cap \supp v$ and $w \in \supp \varphi$.
On the other hand, by  (\ref{eq601}) and again Proposition~\ref{prop6.1} we see that 
$\lim_{n \to \infty} (\partial_k \gamma_0^{(n)} \varphi)(x) = (\partial_k \gamma_0 \varphi)(x)$ 
uniformly for all $x \in \supp v$.
Since
\[
\gota_0^{(n)}(\gamma_0^{(n)} \varphi, v) 
= \sum_{k,l=1}^d \int_\Omega 
    c_{kl}^{(n)} \, (\partial_k \gamma_0^{(n)} \varphi) \, \overline{\partial_l v}
\]
for all $n \in \Ni$ one deduces that 
$\lim \gota_0^{(n)}(\gamma_0^{(n)} \varphi, v) = \gota_0(\gamma_0 \varphi, v)$. 
Hence passing to the limit in (\ref{eq6.3}) gives the lemma.
\end{proof}

\begin{cor}\label{cor6.1} 
Let $\kappa \in (0,1)$, $\Omega \subset \Ri^d$ be an open bounded set with a $C^{1+\kappa}$-boundary
and $C \in \ce(\Omega)$ real symmetric.
Then $K_\cn$ is the  Schwartz kernel of $\cn$.
\end{cor}
\begin{proof} 
Let $\varphi \in D(\cn)$ and $v \in W^{1,2}(\Omega)$ with 
$\supp \varphi \cap \supp v = \emptyset$. 
Then by definition of $\cn$ and Lemma~\ref{lem6.3} one deduces that 
\[
(\cn \varphi, \Tr v)_{L_2(\Gamma)} 
= \gota_0 (\gamma_0 \varphi, v) = 
\int_\Gamma \int_\Gamma K_\cn(z,w) \, \varphi(w) \, \overline{(\Tr v)(z)} \, dw \, dz.
\]
This gives the corollary.
 \end{proof}

\begin{prop}\label{prop6.0}
Let $\kappa \in (0,1)$, $\Omega \subset \Ri^d$ be an open bounded set with a $C^{1+\kappa}$-boundary
and $C \in \ce(\Omega)$ real symmetric.
Then there exists a $c > 0$ 
such that the Schwartz kernel $K_{\cn}$ of 
$\cn$ satisfies
\[ 
| K_{\cn}(z,w) | \le \frac{c}{| z-w|^d}
\]
and 
\[
| K_{\cn}(z,w) - K_{\cn}(z',w') | \le c \, \frac{ (| z-z'| + | w-w'|)^\kappa}{ | z-w|^{d+ \kappa}}
\]
for all $z, z', w, w' \in \Gamma$ with $z \not= w$ and $|z-z'| + |w-w'| \le \frac{1}{2} \, |z-w|$.
\end{prop}
\begin{proof}
We have seen in Corollary~\ref{cor6.1} that $\cn$ has a Schwartz kernel $K_\cn$
given by 
\[
K_\cn(z,w) = -\sum_{k,l,k',l'=1}^d  n_{k'}(w) \, n_k(z) \, c_{k'l'}(w) \, c_{kl}(z) 
   \, (\partial_l^{(1)} \partial_{l'}^{(2)} G)(z,w)
\]
for all $z, w \in \Gamma$ with $z \not= w$. 
Here  $G$ is the Green kernel of the elliptic operator $A_D$. 
It follows immediately from (\ref{etdtnpc501;13}) in Theorem~\ref{tdtnpc501} 
and the fact that the coefficients $c_{kl}$ are all bounded on $\overline{\Omega}$ that 
\[
| K_\cn(z,w) | \le a \, |z-w|^{-d} 
\]
for a suitable constant $a > 0$ and all $z, w \in \Gamma$ with $z \not= w$.  
On the other hand, the bounds (\ref{etdtnpc501;14}) of the same theorem show that 
there exists a $c > 0$ such that 
\begin{eqnarray*}
\lefteqn{
| n_{k'}(w) \, n_k(z) \, c_{k'l'}(w) \, c_{kl}(z) \, (\partial_l^{(1)} \partial_{l'}^{(2)} G)(z,w)
} \hspace*{5mm} \\*
& & \hspace{40mm} {}
  - n_{k'}(w') \, n_k(z') \, c_{k'l'}(w') \, c_{kl}(z') \, 
         (\partial_l^{(1)} \partial_{l'}^{(2)} G)(z',w') |  \\
&=& \Big| \Big( n_{k'}(w) \, n_k(z) \, c_{k'l'}(w) \, c_{kl}(z)
          - n_{k'}(w') \, n_k(z') \, c_{k'l'}(w') \, c_{kl}(z') \Big) \, 
      (\partial_l^{(1)} \partial_{l'}^{(2)}  G)(z,w)   
  \\*
& & \hspace*{15mm} {}
  + n_{k'}(w') \, n_k(z') \, c_{k'l'}(w') \, c_{kl}(z')
         \, \Big( (\partial_l^{(1)} \partial_{l'}^{(2)}  G)(z,w) 
                - (\partial_l^{(1)} \partial_{l'}^{(2)}  G)(z',w') \Big) \Big| \\
& \leq & c \, \frac{ (| z-z'| + | w-w'|)^\kappa}{ | z-w|^d}
   + c \, \frac{ (| z-z'| + | w-w'|)^\kappa}{ | z-w|^{d+ \kappa}}
\end{eqnarray*}
for all $z, z', w, w'  \in \Gamma$
with $z \not= w$, $z' \not=  w'$  and $|z - z'| + |w-w'| \le \frac{1}{2} \, |z-w|$.
Using the fact that $\Gamma$ is bounded we obtain  the bound
\begin{eqnarray*}
\lefteqn{
| n_{k'}(w) \, n_k(z) \, c_{k'l'}(w) \, c_{kl}(z) \, (\partial_l^{(1)} \partial_{l'}^{(2)} G)(z,w)
} \hspace*{5mm} \\*
& & \hspace{40mm} {}
  - n_{k'}(w') \, n_k(z') \, c_{k'l'}(w') \, c_{kl}(z') \, 
         (\partial_l^{(1)} \partial_{l'}^{(2)} G)(z',w') |  \\
& \leq & c_1 \, \frac{ (| z-z'| + | w-w'|)^\kappa}{ | z-w|^{d+ \kappa}}
\end{eqnarray*}
for all $z, z', w, w'  \in \Gamma$
with $z \not= w$, $z' \not=  w'$  and $|z - z'| + |w-w'| \le \frac{1}{2} \, |z-w|$.
This shows  the second bounds of the proposition.
\end{proof}

Next we extend the previous estimates to  the Schwartz kernel $K_{\cn_V}$ of the 
Dirichlet-to-Neumann operator with a potential 
$V \in L_\infty(\Omega)$. 
    
\begin{prop}\label{prop6.01}
Let $\kappa \in (0,1)$, $\Omega \subset \Ri^d$ be an open bounded set with a $C^{1+\kappa}$-boundary
and $C \in \ce(\Omega)$ real symmetric.
Let $V \in L_\infty(\Omega, \Ri)$ 
and assume that $0 \notin \sigma(A_D + V)$. 
Then there exists a constant $c > 0$ such that the Schwartz kernel $K_{\cn_V}$ of $\cn_V$  
satisfies
\[
| K_{\cn_V}(z,w) | \le \frac{c}{| z-w|^d}
\]
and
\[
| K_{\cn_V}(z,w) - K_{\cn_V}(z',w') | 
\le c \, \frac{ (| z-z'| + | w-w'|)^\kappa}{ | z-w|^{d+ \kappa}}
\]
for all $z, z', w, w' \in \Gamma$ with $z \not= w$ and 
$|z-z'| + |w-w'| \le \frac{1}{2} \, |z-w|$.
\end{prop}
\begin{proof} 
First we have $\cn_V = \cn + \gamma_V^* \, M_V \, \gamma_0$ by Corollary~\ref{cdtnpc536}.
We already have the desired estimates for the kernel $K_\cn$. 
It remains to prove the same estimates for the Schwartz kernel $K_Q$ of
 $Q = \gamma_V^* \, M_V \, \gamma_0$. 
Recall that $K_{\gamma_V}$ is the kernel of $\gamma_V$. 
Obviously,
\[
K_Q(z,w) = \int_\Omega K_{\gamma_V}(x,z) \, V(x) \, K_{\gamma_0} (x,w) \, dx
\]
for all $z,w \in \Gamma$ with $z \neq w$.
 We use Proposition~\ref{pdtnpc535}\ref{pdtnpc535-3}.
There exists a constant $c > 0$ such that 
\[
| K_Q(z,w) | 
\le c \int_\Omega \frac{1}{|x-z|^{d-1}} \, \frac{1}{|x-w|^{d-1}} \, dx
\leq c \, \diam(\Omega) \int_\Omega 
    \frac{1}{|x-z|^{d-\frac{1}{2}}} \, \frac{1}{|x-w|^{d-\frac{1}{2}}} \, dx
\]
for all $z,w \in \Gamma$ with $z \neq w$.
We apply \cite{Fri} (Lemma 2, Section 4, Chapter 1) to estimate the RHS by 
$\frac{c_1}{| z-w|^{d-1}}$, uniformly for all $z,w \in \Gamma$ with $z \neq w$.
Next we prove H\"older bounds. 
Let $z,z',w,w' \in \Gamma$ with $z \neq w$, $z' \neq w'$ and $|z-z'| \le \frac{1}{2} \, |z-w|$.
We write
\begin{eqnarray*}
| K_Q(z,w) - K_Q(z',w') | 
&=& \int_\Omega \Big( K_{\gamma_V}(x,z)-  K_{\gamma_V}(x,z') \Big)  
     V(x) \, K_{\gamma_0} (x,w) \, dx \\*
&& \hspace*{10mm} {} + \int_\Omega K_{\gamma_V}(x,z') \, V(x) 
   \Big(  K_{\gamma_0}(x,w)-  K_{\gamma_0}(x,w') \Big) \,  dx \\
&=& I + II.
\end{eqnarray*}
The estimates of $I$ and $II$ are similar. 
We spilt the integral $I$ into two parts
\begin{eqnarray}
I 
&=&  \int_{[|z-z'| \le \frac{1}{2} |x-z|]} 
    \Big( K_{\gamma_V}(x,z)-  K_{\gamma_V}(x,z') \Big) \, V(x) \, K_{\gamma_0} (x,w) \, dx \nonumber \\
&& \hspace*{10mm} {} + \int_{[|z-z'| > \frac{1}{2}|x-z|]} 
    \Big( K_{\gamma_V}(x,z)-  K_{\gamma_V}(x,z') \Big) \, V(x) \, K_{\gamma_0} (x,w) \, dx.
  \label{eprop6.01;1}
\end{eqnarray}
For the first term we use Proposition~\ref{pdtnpc535}\ref{pdtnpc535-3}
and it can be estimated by
\[
c \, | z-z'|^\kappa  \int_\Omega \frac{1}{|x-z|^{d-1 +\kappa}} \, \frac{1}{|x-w|^{d-1}} \, dx
\]
for a suitable $c > 0$.
We apply again  \cite{Fri} (Lemma 2, Section 4, Chapter 1) to estimate the latter integral by
$c' \, |z-z'|^\kappa \, \frac{1}{|z-w|^{d-2 +\kappa}}$
for a suitable $c' > 0$.

The second integral in (\ref{eprop6.01;1}) is more delicate.
If $x \in \Omega$ and $|z-z'| > \frac{1}{2}|x-z|$, then $|x-z'| \le |z-z'| + |x-z| \le 3 |z-z'|$.
Moreover, $|z-w| \le 2 |z'-w|$, since $|z-z'| \le  \frac{1}{2} |z-w|$ by assumption.
Then by Proposition~\ref{pdtnpc535}\ref{pdtnpc535-3} and 
\cite{Fri} (Lemma 2, Section 4, Chapter 1) there are suitable $c_1,c_2 > 0$ such that
\begin{eqnarray*}
\lefteqn{
\int_{[|z-z'| > \frac{1}{2}|x-z|]} 
   \Big| \Big( K_{\gamma_V}(x,z)-  K_{\gamma_V}(x,z') \Big) \, V(x) \, K_{\gamma_0} (x,w) \Big| \, dx
} \hspace*{10mm}  \\*
&\le & c_1 \, \int_{[|z-z'| > \frac{1}{2}|x-z|]} 
    \Big( \frac{1}{|x-z|^{d-1}} +  \frac{1}{|x-z'|^{d-1}} \Big) \,  \frac{1}{|x-w|^{d-1}} \, dx\\
& \le & 3^\kappa \, c_1 \, |z-z'|^\kappa \int_\Omega 
    \Big( \frac{1}{|x-z|^{d-1+ \kappa}} +  \frac{1}{|x-z'|^{d-1 + \kappa}} \Big) 
     \frac{1}{|x-w|^{d-1}} \, dx\\
& \le & c_2 \, |z-z'|^\kappa 
    \Big( \frac{1}{|z-w|^{d-2+ \kappa}} + \frac{1}{|z'-w|^{d-2+ \kappa}} \Big) \\
& \le & (1 + 2^{d-2+\kappa}) \, c_2 \, \frac{|z-z'|^\kappa}{|z-w|^{d-2+ \kappa}}
.
\end{eqnarray*}
The same estimate holds for $II$. 
\end{proof}

We finish this section with the proof of the approximation property used in the 
proof of Lemma~\ref{lem6.3}. 

\begin{prop}\label{prop6.1}
Let $\kappa \in (0,1)$, $\Omega \subset \Ri^d$ be an open bounded set with a $C^{1+\kappa}$-boundary
and $C \in \ce(\Omega)$ real symmetric.
Let $c_{kl}^{(n)}$, $A^{(n)}$, $G^{(n)}$,\ldots\ be as in the  proof of Lemma~\ref{lem6.3}.
Let $K_1$ and $K_2$ be compact and disjoint subsets of $\overline{\Omega}$ such that 
$\overline{\mathring{K_1} \times \mathring{K_2} } = K_1 \times K_2$. 
Let $k,l \in \{1,\dots,d\}$. 
Then
\[
\lim_{n \to \infty} (\partial_k^{(1)} \partial_l^{(2)} G^{(n)})(x,y) 
= (\partial_k^{(1)} \partial_l^{(2)} G)(x,y)
\]
uniformly for all $x \in K_1$ and $y \in K_2$.
\end{prop}
\begin{proof} 
Let $u, v \in W^{1,2}(\Omega)$ be such that $\supp u \subset K_1$ and $\supp v \subset K_2$. 
Then
\begin{eqnarray*}
\int_{K_1 \times K_2} (\partial_k^{(1)} \partial_l^{(2)} G)(x,y) \, u(x) \, \overline{v(y)} \, d(x,y) 
&=& \int_\Omega \int_\Omega
   (\partial_k^{(1)} \partial_l^{(2)} G)(x,y) \, u(x) \, \overline{v(y)} \, dx \, dy \\ 
&=& \int_\Omega \int_\Omega G(x,y) \, (\partial_k u)(x) \, \overline{(\partial_l v)(y)} \, dx \, dy \\
&=& ( A_D^{-1} \partial_k u, \partial_l v)_{L_2(\Omega)}\\
&=& \lim_{n\to \infty} ( (A_D^{(n)})^{-1} \partial_k u, \partial_l v)_{L_2(\Omega)}\\
&=& \lim_{n\to \infty} \int_{K_1 \times K_2} 
    (\partial_k^{(1)} \partial_l^{(2)} G^{(n)} )(x,y) \, u(x) \, \overline{v(y)}  \, d(x,y), 
\end{eqnarray*}
where the forth equality follows as in (\ref{pdtnpc533;8}).  
Hence 
\begin{equation}\label{eq602}
\int_{K_1 \times K_2} (\partial_k^{(1)} \partial_l^{(2)} G)(x,y) \, u(x) \, \overline{v(y)} \, d(x,y) 
= \lim_{n\to \infty} \int_{K_1 \times K_2} (\partial_k^{(1)} \partial_l^{(2)} G^{(n)} )(x,y) \, 
    u(x) \, \overline{v(y)} \, d(x,y).
\end{equation}
By Theorem~\ref{tdtnpc501}  there exists  a $c > 0$ such that
\[
| (\partial_k^{(1)} \partial_l^{(2)} G^{(n)} )(x,y) | 
\le \frac{c}{| x-y|^d} 
\le  \frac{c}{{\rm dist}(K_1, K_2)^d}
\]
uniformly for all $n \in \Ni$ and $(x,y) \in K_1 \times K_2$. 
Note that we  also have H\"older bounds for 
$(\partial_k^{(1)} \partial_l^{(2)} G^{(n)} )(x,y)$ which are uniform in $n$ and  
$(x,y) \in K_1 \times K_2$ by the same theorem. 
Therefore $(\partial_k^{(1)} \partial_l^{(2)} G^{(n)} )_{n \in \Ni}$ is equicontinuous on 
$K_1 \times K_2$. 
By the Ascoli--Arzel\`a theorem there exists a
$\Phi \in C(K_1 \times K_2)$ such that after passing to a subsequence
if necessary
$(\partial_k^{(1)} \partial_l^{(2)} G^{(n)} )_{n \in \Ni}$ converges  
to $\Phi$ uniformly on $K_1 \times K_2$. 
Then (\ref{eq602}) gives
\begin{eqnarray*}
\int_{K_1 \times K_2} \Phi(x,y) \, u(x) \, \overline{v(y)} \, dx \, dy 
&=&  \lim_{n\to \infty} \int_{K_1 \times K_2} 
    (\partial_k^{(1)} \partial_l^{(2)} G^{(n)} )(x,y) \, u(x) \, \overline{v(y)} \, dx \, dy\\
&=& \int_{K_1 \times K_2} 
    (\partial_k^{(1)} \partial_l^{(2)} G)(x,y) \, u(x) \, \overline{v(y)} \, dx \, dy.
\end{eqnarray*}
Note that by $\partial_k^{(1)} \partial_l^{(2)} G$ is continuous on $K_1 \times K_2$ 
by Theorem~\ref{tdtnpc501}.
Hence
$\Phi(x,y) = (\partial_k^{(1)} \partial_l^{(2)} G)(x,y)$ for all 
$(x,y) \in (K_1 \times K_2)^{\circ}$. 
By continuity this equality extends to all
$(x,y) \in K_1 \times K_2$. 
\end{proof}

\section{$L_p$-commutator estimates}\label{S4}

In this section we aim to derive good bounds on $L_p(\Gamma)$ for the commutator of 
the Dirichlet-to-Neumann operator $\cn_V$ and a multiplication operator $M_g$,
where $g$ is a Lipschitz continuous function on $\Gamma$.
A key ingredient is a commutator estimate by Shen \cite{She2}. 

If the boundary $\Gamma$ is $C^\infty$ and $c_{kl} = \delta_{kl}$ it is 
well known that $\cn$ is a pseudo-differential operator.
In this case 
a well known result of Calder\'on shows that $[\cn, M_g]$ acts boundedly on 
$L_p(\Gamma)$ with norm bounded by $C \, \| \nabla g \|_{L_\infty(\Gamma)}$ for 
some constant $C > 0$.
See also Coifman and Meyer \cite{CM2} for more results on commutators of 
pseudo-differential operators. 

It is our aim here to obtain similar results for less smooth domains and  
variable coefficients $c_{kl}$. 
We start with the following recent result of Z.~Shen who treated the case 
of $L_2$-estimates for bounded Lipschitz domains.
An additional  problem is that it is unclear whether the domain of $\cn$ is 
invariant under the multiplication operator.
Also that is a consequence of the same theorem.  
We formulate the commutator estimate of Shen, \cite{She2} Theorem~1.1, 
in the quadratic form sense.

\begin{thm} \label{thm4.1}
Let $\Omega \subset \Ri^d$ be an open bounded set with Lipschitz boundary.
Let $C \in \ce(\Omega)$ be real symmetric and suppose that each $c_{kl}$ is 
H\"older continuous on $\Omega$.
Then there exists a $c > 0$ such that 
\begin{equation}
|\gotb_V(g \, \varphi, \psi) - \gotb_V(\varphi, \overline{g} \, \psi)|
\leq c \, \|g\|_{C^{0,1}(\Gamma)} \, \|\varphi\|_{L_2(\Gamma)} \, \|\psi\|_{L_2(\Gamma)}
\label{etdtnpc701;1}
\end{equation}
for all $g,\varphi \in C^{0,1}(\Gamma)$ and $\psi \in H^{1/2}(\Gamma)$.
\end{thm}

The theorem gives invariance of the domain of the Dirichlet-to-Neumann operator
under multiplication with a Lipschitz function and commutator estimates. Recall that 
\[
{\rm Lip}_\Gamma(g) = \sup_{z,w \in \Gamma, z \not= w} \frac{ | g(z) - g(w) |}{| z-w|}
\]
for every $g \in C^{0,1}(\Gamma)$.

\begin{thm} \label{tdtnpc701}
Let $\Omega \subset \Ri^d$ be an open bounded set with Lipschitz boundary.
Let $C \in \ce(\Omega)$ be real symmetric and suppose that each $c_{kl}$ is 
H\"older continuous on $\Omega$.
Then $g \, \varphi \in D(\cn)$ for all $g \in C^{0,1}(\Gamma)$ and 
$\varphi \in D(\cn)$.
Moreover, there exists a $c > 0$ such that 
\[
\|[\cn, M_g]\|_{2 \to 2} \leq c \, {\rm Lip}_\Gamma(g)
\]
for all $g \in C^{0,1}(\Gamma)$.
\end{thm}
\begin{proof}
Let $c > 0$ be as in Theorem~\ref{thm4.1}.
Let $g \in C^{0,1}(\Gamma)$.
Then $M_g$ is bounded from $L_2(\Gamma)$ into $L_2(\Gamma)$ 
and from $H^1(\Gamma)$ into $H^1(\Gamma)$.
Hence by interpolation the operator $M_g$ is bounded from 
$H^{1/2}(\Gamma)$ into $H^{1/2}(\Gamma)$.
Since $C^{0,1}(\Gamma)$ is dense in $H^{1/2}(\Gamma)$ it 
follows that (\ref{etdtnpc701;1}) extends to all 
$\varphi,\psi \in H^{1/2}(\Gamma)$.

If $\varphi \in D(\cn)$, then 
\begin{eqnarray*}
|\gotb_V(g \, \varphi, \psi)|
& \leq & |\gotb_V(\varphi, \overline{g} \, \psi)| 
   + c \, \| g \|_{C^{0,1}(\Gamma)} \, \|\varphi\|_{L_2(\Gamma)} \, \|\psi\|_{L_2(\Gamma)}  \\
& \leq & \|\cn \varphi\|_{L_2(\Gamma)} \, \|\overline{g} \, \psi\|_{L_2(\Gamma)}
   + c \,  \| g \|_{C^{0,1}(\Gamma)} \, \|\varphi\|_{L_2(\Gamma)} \, \|\psi\|_{L_2(\Gamma)}
\leq c' \, \|\psi\|_{L_2(\Gamma)}
\end{eqnarray*}
for all $\psi \in H^{1/2}(\Gamma) = D(\gotb_V)$,
where $c' = \|\cn \varphi\|_{L_2(\Gamma)} \, \|g\|_{L_\infty(\Gamma)}
   + c \,  \| g \|_{C^{0,1}(\Gamma)} \, \|\varphi\|_{L_2(\Gamma)}$.
Hence $g \, \varphi \in D(\cn)$.
Then the extended version of (\ref{etdtnpc701;1}) gives
\[
\|[\cn, M_g] \varphi\|_{L_2(\Gamma)}
\leq c \,  \| g \|_{C^{0,1}(\Gamma)} \, \|\varphi\|_{L_2(\Gamma)}
\]
for all $\varphi \in D(\cn)$.
So 
\begin{equation}
\|[\cn, M_g]\|_{2 \to 2}
\leq c \,  \| g \|_{C^{0,1}(\Gamma)}.
\label{etdtnpc701;2}
\end{equation}
We next observe  that one can replace $ \| g \|_{C^{0,1}(\Gamma)}$ by 
${\rm Lip}_\Gamma(g)$.

Fix  $z_0 \in \Gamma$ and apply (\ref{etdtnpc701;2}) to $g- g(z_0)$ to obtain
\begin{eqnarray*}
 \| [\cn, M_g] \|_{2 \to 2} 
&=& \| [\cn, M_{g-g(z_0)}] \|_{2 \to 2}\\  
&\le& c \, \| g - g(z_0) \|_{C^{0,1}(\Gamma)}
= c \, \Big( \| g - g(z_0) \|_{L_\infty(\Gamma)} 
   + {\rm Lip}_\Gamma(g) \Big).
\end{eqnarray*}
On the other hand, since  $\Gamma$ is bounded 
\[
|g(z) - g(z_0)| 
\le |z-z_0| \,  {\rm Lip}_\Gamma(g)
\le {\rm diam}(\Omega) \, {\rm Lip}_\Gamma(g)
\]
for all $z \in \Gamma$.
Thus  the desired estimate holds. 
\end{proof}

Next we extend this result to  $L_p(\Gamma)$ for all $p \in (1,\infty)$ when the 
underlying domain is  more regular. 
We even have the result for $\cn_V$ with $V \in L_\infty(\Omega)$.

\begin{thm}\label{thm4.2V}
Let $\kappa \in (0,1)$, $\Omega \subset \Ri^d$ be an open bounded set with a $C^{1+\kappa}$-boundary
and $C \in \ce(\Omega)$ real symmetric.
Let $V \in L_\infty(\Omega, \Ri)$ 
and assume that $0 \notin \sigma(A_D + V)$. 
Then for all $p \in (1, \infty)$ 
 there exists a $c > 0$ such that 
\[
 \| [\cn_V, M_g] \|_{p \to p}  
\le c \, {\rm Lip}_\Gamma(g)
\]
for all $g \in C^{0,1}(\Gamma)$.

In addition the operator  $[\cn_V, M_g]$ is of weak type $(1,1)$ with an estimate
$ c \, {\rm Lip}_\Gamma(g)$ as before.
\end{thm} 
\begin{proof}
First, recall from Corollary~\ref{cdtnpc536} that
\[
\cn_V  =  \cn + Q,
\]
where $Q =  \gamma_0^* \, M_V \, \gamma_V = \gamma_V^* \, M_V \, \gamma_0$.
Then
\[
[\cn_V, M_g] = [\cn, M_g] + [Q, M_g].
\]
On the other hand, by Proposition~\ref{pdtnpc535}\ref{pdtnpc535-4}
the operators $\gamma_V$ and $\gamma_0$ 
have bounded extensions from
$L_p(\Gamma)$ to $L_p(\Omega)$ for all $p \in [1, \infty]$. 
Therefore $[Q, M_g]$ is a bounded operator on $L_p(\Gamma)$ with norm 
$\|[Q, M_g]\|_{p \to p} \leq 2 \|Q\|_{p \to p} \, \| g \|_{L_\infty(\Gamma)}$.
Then as in the proof of Theorem~\ref{tdtnpc701}, we fix $z_0 \in \Gamma$ and 
apply the above estimate with the function  
$g- g(z_0)$ and use the trivial bound 
$\|g - g(z_0)\|_{L_\infty(\Omega)} 
\le {\rm diam}(\Omega) \,  {\rm Lip}_\Gamma(g)$ to obtain
\[
 \| [Q, M_g] \|_{p \to p}  \le c' \,  {\rm Lip}_\Gamma(g),
\]
where $c' = 2 \|Q\|_{p \to p} \, {\rm diam}(\Omega)$.
Hence it remains to prove the correspond estimates for $[\cn, M_g]$
if $p \in (1,\infty)$.

The Schwartz kernel $K$ of the commutator $[\cn, M_g]$ is given by
$ K(z,w) := (g(w) - g(z))K_\cn(z,w)$, 
where $K_\cn$ denotes the Schwartz kernel of $\cn$. 
It follows from Proposition~\ref{prop6.0} that there is a suitable $c > 0$ such that 
\begin{equation}\label{4.3}
| K(z,w) | \le c \, {\rm Lip}_\Gamma(g) |z-w|^{-(d-1)} 
\end{equation}
for all $z, w \in \Gamma$ with $z \not= w$. 
On the other hand, using the second bounds in Proposition \ref{prop6.0}
we see that there exists a suitable $c' > 0$ such that 
\begin{eqnarray*}
| K(z,w) - K(z',w)|
& = & | (g(w)-g(z)) \, K_\cn(z,w) - (g(w)-g(z'))K_\cn(z,w)(z',w) |  \\
&\le& (|g(w)| + |g(z')|) \, | K_\cn(z,w)(z,w) - K_\cn(z,w)(z',w)  |   \\*
& & \hspace*{30mm} {}
     + |(g(z') -g(z)) \, K_\cn(z,w)(z,w) | \\
&\le& 2c' \, \| g  \|_{L_\infty(\Gamma)} \, |z-z'|^\kappa \, |z-w|^{-(d + \kappa)} 
   + c' \, |z-z'| \, {\rm Lip}_\Gamma(g) \, |z-w|^{-d} 
\end{eqnarray*}
for all $z, z', w  \in \Gamma$
with $z \not= w$ and $|z - z'| \le \frac{1}{2} \, |z-w|$.
Again as in 
the proof of Theorem~\ref{tdtnpc701}, we fix $z_0 \in \Gamma$ and 
apply the above estimate with the function  
$g - g(z_0)$ and use the trivial bound 
$\|g - g(z_0)\|_{L_\infty(\Gamma)} \le {\rm diam}(\Omega) \,  {\rm Lip}_\Gamma(g)$,  
to deduce that there exists a suitable $M > 0$ such that 
\begin{equation}\label{4.4}
 | K(z,w) - K(z',w) | 
\le M \, |z-z'|^\kappa \, |z-w|^{-(d-1 + \kappa)} \,  {\rm Lip}_\Gamma(g)
 \end{equation}
for all $z, z', w \in \Gamma$ with $z \not = w$ and $|z - z'| \le \frac{1}{2}|z'-w|$.
 Similarly, one proves that
 \begin{equation}\label{4.5}
 | K(z,w) - K(z,w') | 
\le M_1 |w-w'|^{\varepsilon} |z-w|^{-(d-1 + \varepsilon)} \,  {\rm Lip}_\Gamma(g)
 \end{equation}
for all $z, w, w' \in \Gamma$ with $z \not = w$ and $|w - w'| \le \frac{1}{2}|z-w|$. 
 It follows form (\ref{4.3}), (\ref{4.4}) and (\ref{4.5}) that $K(\cdot,\cdot)$ is a 
Calder\'on--Zygmund kernel.
We obtain from this and Theorem~\ref{tdtnpc701} that $[\cn, M_g]$ acts as a 
bounded operator on $L_p(\Gamma)$ for all $p \in (1, \infty)$ with norm estimated by 
$c'' \, {\rm Lip}_\Gamma(g)$.
This shows Theorem \ref{thm4.2V}.  
\end{proof}

\section{Proof of Poisson bounds}\label{S5}

Throughout this section we adopt the notation and assumptions of Theorem~\ref{thm1.1}.

We start with a lemma.

\begin{lemma} \label{ldtnpc801}
For all $p \in [1,\infty)$ the semigroup $S^V$ extends consistently 
to a $C_0$-semigroup on $L_p(\Gamma)$.
\end{lemma}
\begin{proof}
By Theorem \ref{th2.2}\ref{th2.2-2}, the semigroup $S$ generated by $-\cn$ is sub-Markovian and 
hence it acts as a contraction semigroup on 
$L_p(\Gamma)$ for all $p \in [1, \infty]$ and it is strongly continuous if $p \in [1, \infty)$. 
We know from  Proposition~\ref{pdtnpc535}\ref{pdtnpc535-4} 
that the operator $Q = \gamma_V^* \, M_V \, \gamma_0$ 
is bounded on $L_p(\Gamma)$ for all $p \in [1, \infty]$. 
Since  
$\cn_V  =  \cn + Q$ by Corollary~\ref{cdtnpc536} it follows from standard perturbation theory 
that the semigroup $S^V$ generated by $-\cn_V$ acts as a $C_0$-semigroup on $L_p(\Gamma)$ 
for all $p \in [1, \infty)$. 
\end{proof}

Let $S^V$ be the semigroup generated by $-\cn_V$ on $L_2(\Gamma)$.

\begin{lemma}\label{lem8.00}
For all $t > 0$ the operator $S_t^V$ has a kernel $K_t^V$. 
In addition, there exist $\omega \in \Ri$ and $c > 0$ such that
\[
| K_t^V(z,w) | \le c \, t^{-(d-1)} \, e^{\omega t} 
\]
for all $t > 0$ and a.e. $z,w \in \Gamma$. 
\end{lemma}
\begin{proof}
If $d \ge 3$ the Sobolev embedding of 
$\Tr(W^{1,2}(\Omega))$ into $L_{\frac{2(d-1)}{d-2}}(\Gamma)$ 
of \cite{Nec2} Theorem~2.4.2
implies easily that there exist $c,\omega > 0$ such that 
\[
\gotb_V (\varphi, \varphi) + \omega \int_\Gamma |\varphi |^2 
\ge c \, \| \varphi \|_{L_{\frac{2(d-1)}{d-2}}(\Gamma)}
\]
for all $\varphi \in \Tr(W^{1,2}(\Omega))$. 
Since the semigroup $S^V$ acts on $L_p(\Gamma)$ for all $p \in [1, \infty]$ it is well known 
that the later inequality implies  ultracontractivity estimates. 
More precisely, there exist $c,\omega >0$ such that 
\begin{equation}\label{eq8.00}
\|S_t^V \|_{1 \to \infty} 
\le c \, t^{-(d-1)} \, e^{\omega t} 
\end{equation}
for all $t > 0$. 
Note that (\ref{eq8.00}) implies that $S^V_t$ is given by a kernel $K_t^V$ such that
\[
| K_t^V(z,w) | \le c \, t^{-(d-1)} \, e^{\omega t} 
\]
for all $t > 0$ and a.e. $z,w \in \Gamma$. 

If $d = 2$, then the proof is precisely the same as in the proof 
of Theorem~2.6 in \cite{EO4}.
\end{proof}

Using the previous two lemmas, duality and interpolation one 
deduces easily the next lemma.

\begin{lemma} \label{ldtnpc810}
There exists a $c > 0$ such that 
\[
\|S^V_t \|_{p \to q} 
\le c \, t^{-(d-1)(\frac{1}{p} - \frac{1}{q})} 
\]
for all $t \in (0,1]$ and $p,q \in [1,\infty]$ with $p \leq q$.
\end{lemma}

Let $g \in C^{0,1}(\Gamma)$.
Define $\delta_g(\cn_V) = [M_g,\cn_V]$ and for all $j \in \Ni$ 
define inductively $\delta_g^{j+1}(\cn_V) = [M_g, \delta^j(\cn_V)]$.
Define similarly $\delta^j_g(S_t^V)$.

\begin{lemma}\label{lem8.2}
Suppose either $p,q \in (1,\infty)$ with 
$p \leq q$ and $(d-1)(\frac{1}{p} - \frac{1}{q}) \in \{ 0,1,\ldots,d-1 \} $,
or $p=1$ and $q=\infty$.
Then there exists a $c_{p,q} > 0$ such that 
\[
\|\delta_g^j(\cn_V)\|_{p \to q}
\leq c_{p,q} \, ({\rm Lip}_\Gamma(g))^j
\]
for all $g \in C^{0,1}(\Gamma)$,
where $j = 1 + (d-1)(\frac{1}{p} - \frac{1}{q})$.
\end{lemma}
\begin{proof}
The kernel $\widetilde{K}$ of $\delta^j_g(\cn)$ is  given by
$\widetilde{K}(z,w) = (g(w) - g(z))^j \, K_{\cn_V}(z,w)$, where we use again $K_{\cn_V}$ to 
denote the Schwartz kernel of $\cn_V$. 
It follows immediately from Proposition~\ref{prop6.01} that there is a suitable $c > 0$
such that 
\begin{equation}\label{5.1}
| \widetilde{K}(z,w) | 
\le c \, |g(z)-g(w)|^j \, |z-w|^{-d} 
\le c \, |z-w|^{-(d-j)} \, ({\rm Lip}_\Gamma(g))^j
\end{equation}
for all $z,w \in \Gamma$ with $z \neq w$.

If $1 < p < q < \infty$, then (\ref{5.1}) implies that 
$\widetilde{K}$ is a Riesz potential. 
Then the boundedness of $\delta^j_g(\cn_V)$ from $L_p(\Gamma)$ to $L_q(\Gamma)$
follows from \cite{Ste1}, Theorem~V.1.

If $ 1 < p = q < \infty$, then the statement of the lemma is given by 
Theorem \ref{thm4.2V}.

Finally, if $p = 1$ and $q= \infty$ then $j= d$.
In this case
\[
| \widetilde{K}(z,w) | 
\le c \, ({\rm Lip}_\Gamma(g))^d
\]
and hence $\delta^d_g(\cn_V)$ is bounded from $L_1(\Gamma)$ into $L_\infty(\Gamma)$ 
with norm estimates by 
$c \, ({\rm Lip}_\Gamma(g))^d$. 
\end{proof}

In order to prove the Poisson bound for the kernel $K^V_t(x,y)$  we proceed  
as in Section~4 of \cite{EO4}.
For the reader's convenience, we repeat the arguments.
Let $c > 0$ be as in Lemma~\ref{ldtnpc810}.
Further let $c_{p,q}$ be as in Lemma~\ref{lem8.2}.

Let $t \in (0,1]$.
Then 
\begin{eqnarray*}
\delta_g^d(S_t^V)
& = & \sum_{k=1}^d (-t)^k
   \sum_{\scriptstyle j_1,\ldots,j_k \in \Ni \atop
         \scriptstyle j_1 + \ldots + j_k = d}
   \int_{H_k}
     S^V_{t_{k+1} \, t} \, \delta^{j_k}(\cn_V) \, S^V_{t_k \, t} 
       \circ \ldots \circ  \\*
& & \hspace*{50mm} {}
   \circ
     S^V_{t_2 \, t} \, \delta^{j_1}(\cn_V) \, S^V_{t_1 \, t} \, d\lambda_k(t_1,\ldots,t_{k+1}),
\end{eqnarray*}
where 
\[
H_k = \{ (t_1,\ldots,t_{k+1}) \in (0,\infty)^{k+1} : t_1 + \ldots + t_{k+1} = 1 \}
\]
and  $d\lambda_k$ denotes Lebesgue measure of the $k$-dimensional surface $H_k$.
We estimate each term in the sum.
Let $k \in \{ 1,\ldots,d \} $, $(t_1,\ldots,t_{k+1}) \in H_k$,
$g \in C^{0,1}(\Gamma)$, $t \in (0,1]$ and $j_1,\ldots,j_k \in \Ni$ with 
$j_1 + \ldots + j_k = d$ and ${\rm Lip}_\Gamma(g) \leq 1$.

If $k = 1$ then $j_1 = d$ and we have for $t \in (0, 1]$ and each $g$ with 
${\rm Lip}_\Gamma(g) \le 1$
\begin{eqnarray*}
t \, \|S^V_{t_2 t} \, \delta_g^{d}(\cn_V) \, S^V_{t_1 t}\|_{1 \to \infty}
& \leq & t \, \|S^V_{t_2 t}\|_{\infty \to \infty} \, 
      \|\delta_g^{d}(\cn_V)\|_{1 \to \infty} \, \|S^V_{t_1 t}\|_{1 \to 1}  \\
& \leq & c^2 \, c_{1,\infty} \, t.  
\end{eqnarray*}
Suppose that $k \in \{ 2,\ldots,d \} $.
There exists an $M \in \{ 1,\ldots,k+1 \} $ such that $t_M \geq \frac{1}{k+1}$.
Note that $\sum_{\ell = 1}^k (j_\ell - 1) = d-k < d-1$.
First suppose $M \not\in \{ 1,k+1 \} $.
Fix $1 = q_0 < p_1 \leq q_1 = p_2 \leq q_2 = p_3 \leq \ldots \leq 
q_{M-2} = p_{M-1} \leq q_{M-1} \leq p_M \leq q_M = p_{M+1} \leq q_{M+1} \leq \ldots \leq 
q_{k-1} = p_k \leq q_k < p_{k+1} = \infty$
such that 
\[
1 - \frac{1}{p_1} = \frac{1}{2(d-1)} = \frac{1}{q_k}
\quad , \quad
\frac{1}{p_\ell} - \frac{1}{q_\ell} = \frac{j_\ell - 1}{d-1}
\quad \mbox{and} \quad
\frac{1}{q_{M-1}} - \frac{1}{p_M} = \frac{k-2}{d-1}
\]
for all $\ell \in \{ 1,\ldots,k \} $.
Then 
\begin{eqnarray*}
\lefteqn{
t^k \, \|S^V_{t_{k+1} t} \, \delta_g^{j_k}(\cn_V)  \ldots \delta_g^{j_1}(\cn_V) \, S^V_{t_1 t}\|_{1 \to \infty}
} \hspace*{5mm}  \\*
& \leq & t^k \, \|S^V_{t_1 t}\|_{q_0 \to p_1}
       \prod_{\ell = 1}^k  \|S^V_{t_{\ell+1} t}\|_{q_\ell \to p_{\ell + 1}} 
                           \, \|\delta_g^{j_\ell}(\cn_V)\|_{p_\ell \to q_\ell} \\
& \leq & t^k \, c \,  (t_1 t)^{-(d-1)(\frac{1}{q_0} - \frac{1}{p_1})}
    \prod_{\ell = 1}^k c_{p_\ell , q_\ell} \, 
       c \, (t_{\ell + 1} t)^{-(d-1)(\frac{1}{q_\ell} - \frac{1}{p_{\ell + 1}})}  \\
& = & c' \,  t^k \, t^{-(k-1)} \, 
   t_1^{-1/2} \, t_K^{-(k-2)} \, t_{k+1}^{-1/2}  \\
& \leq & c' \, (k+1)^{k-2} \,   t\, 
   t_1^{-1/2} \, t_{k+1}^{-1/2},
\end{eqnarray*}
where $c' = c^{k+1} \, \prod_{\ell = 1}^k c_{p_\ell , q_\ell}$.
If $M \in \{ 1,k+1 \} $ then a similar estimate is valid with possibly 
a different constant for $c'$.
Integration and taking the sum gives 
\[
\| \delta^d_g(S^V_t) \|_{1 \to \infty} \le c'' \, t
\]
for a suitable $c'' > 0$, uniformly for all $t \in (0,1]$ and $g \in C^{0,1}(\Gamma)$ such that 
${\rm Lip}_\Gamma(g) \le 1$.
Therefore
\begin{equation}\label{optim}
|g(w) -g(z)|^d \, | K^V_t(z,w)| \le c'' \, t
\end{equation}
for all $w,z \in \Gamma$. 

Note that the metric $d_\Gamma \colon \Gamma \times \Gamma \to [0,\infty)$
given by
\[
d_\Gamma(z,w) := \sup \{ g(z) - g(w):  g \in C^{0,1}(\Gamma), \;  {\rm Lip}_\Gamma(g) \le 1 \}
\]
is equivalent to the Euclidean one. 
Indeed, by the definition of ${\rm Lip}_\Gamma(g) \le 1$ one has
$|g(z) - g(w)| \le |z-w|$ for all $z, w \in \Gamma$. 
Hence $d_\Gamma(z,w) \le |z-w|$.  
To obtain the reverse inequality, let $k \in \{ 1,\ldots,d \} $ and choose
$g(z) = z_k$, where $z = (z_1,\dots,z_d)$. 
Then $|z_k - w_k| \le d_\Gamma(z,w)$ and hence $|z-w| \le d^{d/2} \, d_\Gamma(z,w)$.

Using this fact and optimizing over $g$ in  (\ref{optim}) we  obtain
\[
|w-z|^d \, | K^V_t(z,w) | \le c'' \, d^{d/2} \, t
\]
for all $t \in (0, 1]$ and $z,w \in \Gamma$.
We combine this with Lemma~\ref{lem8.00} and obtain that there is a $c > 0$
such that 
\[
| K^V_t(z,w) | 
\leq \frac{c \, (t \wedge 1)^{-(d-1)} \, e^{\omega t}}
         {\displaystyle \Big( 1 + \frac{|z-w|}{t} \Big)^d }
\]
for all $t \in (0, 1]$ and $z,w \in \Gamma$.
By \cite{Ouh5} Lemma~6.5 and the fact that $\Gamma$ is bounded we improve 
this bound and there is a $c > 0$
\[
| K^V_t(z,w) | 
\leq \frac{c \, (t \wedge 1)^{-(d-1)} \, e^{-\lambda_1 t}}
         {\displaystyle \Big( 1 + \frac{|z-w|}{t} \Big)^d }
\]
for all $t > 0$ and $z,w \in \Gamma$. 
This completes the proof of Theorem~\ref{thm1.1}. \hfill$\Box$

\subsection*{Acknowledgements} 
The authors wish to thank Christophe Prange for several interesting discussions.
This work was carried out when the second named author was visiting the University of
Auckland and the first named author was visiting the University of Bordeaux.
Both authors wish to thank the universities for hospitalities.
The research of A.F.M. ter  Elst  is partly supported by the 
Marsden Fund Council from Government funding, 
administered by the Royal Society of New Zealand. 
The research of E.M.  Ouhabaz  is partly supported by the ANR 
project `Harmonic Analysis at its Boundaries',  ANR-12-BS01-0013-02.

\end{document}